\documentclass[12pt,a4paper,final,reqno]{amsart}

\newcommand{\includefigure}[2]{\includegraphics[#1]{#2}}

\usepackage{amsmath}
\usepackage{amsthm}
\usepackage{amssymb}
\usepackage{graphicx}
\usepackage{mathrsfs}
\usepackage[english]{babel}
\usepackage[latin1]{inputenc}
\usepackage[T1]{fontenc}
\usepackage{amsrefs}
\usepackage[usenames,dvipsnames]{xcolor}
\usepackage{enumitem}
\usepackage{bbm}
\usepackage{float}
\usepackage{tikz}
\usepackage{multirow}
\usepackage{mathtools}
\usepackage{hyperref}

\numberwithin{equation}{section}

\newtheorem{theorem}{Theorem}[section]
\newtheorem{proposition}[theorem]{Proposition}

\newtheorem{lemma}[theorem]{Lemma}

\newtheorem{remark}[theorem]{Remark}
\newtheorem{example}[theorem]{Example}

{\theoremstyle{definition}
{

\newtheorem{defn}[theorem]{Definition}
}}

\newcommand{\cal}{\mathcal}

\newcommand{\A}{{\cal C}^\omega}

\newcommand{\CC}{{\cal C}}

\newcommand{\FF}{{\cal F}}

\newcommand{\HH}{{\cal H}}

\newcommand{\RR}{{\cal R}}

\newcommand{\UU}{{\mathscr D}}

\newcommand{\fX}{{\mathfrak X}^\omega}

\newcommand{\Nn}{{\mathbb{N}}}

\newcommand{\Rr}{{\mathbb{R}}}

\newcommand{\Zz}{{\mathbb{Z}}}

\newcommand{\inter}{{\rm int}}

\def\e{\mathrm{e}}

\usetikzlibrary{decorations.markings}

\tikzset{->-/.style={decoration={
  markings,
  mark=at position #1 with {\arrow{>}}},postaction={decorate}}}

\renewcommand{\d}{d}

\newcommand{\Mat}{{\rm Mat}}

\newtheorem*{theorem*}{Theorem}

\newcommand{\abs}[1]{\left| #1\right|}
\newcommand{\norm}[1]{\left\| #1\right\|}

\newcommand{\id}  {\operatorname{id}}

\newcommand{\comment}[1]{}

\newcommand{\Bscr}{\mathscr{B}}

\newcommand{\Dscr}{\mathscr{D}}

\newcommand{\Poin}[2]{P_{#2}}

\newcommand{\skPoin}[2]{\pi_{#2}}
\newcommand{\skDomPoin}[2]{\Pi_{#2}}
\newcommand{\etal}{\textit{et al}. }
\newcommand{\ie}{\textit{i}.\textit{e}., }

\newcommand{\Ker}{\mathrm{Ker}}
\newcommand{\Eq}{\mathrm{Eq}}

\begin{document}
\selectlanguage{english}

\title[Asymptotic Poincare Maps]{Asymptotic Poincar\'e Maps along the edges of Polytopes}

\author[Alishah]{Hassan Najafi Alishah}
\address{Departamento de Matem\'atica, Instituto de Ci\^encias Exatas\\
Universidade Federal de Minas Gerais \\
Belo Horizonte, 31270-901, Brazil}
\email{halishah@mat.ufmg.br}

\author[Duarte]{Pedro Duarte}
\address{Departamento de Matem\'atica and CMAF \\
Faculdade de Ci\^encias\\
Universidade de Lisboa\\
Campo Grande, Edificio C6, Piso 2\\
1749-016 Lisboa, Portugal 
}
\email{pmduarte@fc.ul.pt}

\author[Peixe]{ Telmo Peixe}
\address{REM and CEMAPRE \\
ISEG Lisbon School of Economics \& Management\\
Universidade de Lisboa\\
Rua do Quelhas, 6 \\
1200-781 Lisboa, Portugal 
}
\email{telmop@iseg.ulisboa.pt}

\begin{abstract}
	For a class of flows on polytopes, including many examples from Evolutionary Game Theory, we  describe a piecewise linear model which encapsulates the asymptotic dynamics along the heteroclinic network formed out of the polytope's vertexes and edges. This  piecewise linear flow is easy to compute even in higher dimensions, which allows the  usage of numeric algorithms to find invariant dynamical structures such as periodic, homoclinic or heteroclinic orbits, which if robust persist as invariant dynamical structures of the original flow. We apply this method to prove the existence of chaotic behavior in some Hamiltonian replicator systems on the five dimensional simplex.
 \end{abstract}

\maketitle
\tableofcontents


\section{Introduction}
\label{intro}


Given a flow on a polytope, leaving  all  its faces invariant, we  call  {\em flowing edge} to any edge of the polytope  consisting of a single orbit flowing between the two endpoint singularities. The purpose of this paper is to present a new method to encapsulate and analyze the asymptotic dynamics of the flow along the heteroclinic network formed by the flowing edges and the vertex singularities of the polytope.

Natural examples of such dynamical systems arise in Evolutionary Game Theory (EGT), which was the background motivation for the present work.  Even though the phase space of the dynamical systems arising from EGT are usually simplexes or products of simplexes, we state our results in the more general context of simple polytopes in order to have a mathematically comprehensive approach which is also open to new examples. 

The replicator equations, introduced by P. Taylor and L. Jonker~\cite{TJ1978}, as well as the polymatrix replicator equation, studied by authors in~\cite{AD2014,ADP2015}, induce flows on simple polytopes in the scope of applicability of the present results. Polymatrix replicator equations extend the class of bimatrix replicator equations~\cite{SS1981,SSHW1981}. 


The use of cross-sections and return maps
to analyze dynamics along heteroclinic cycles is an old tool going back to Poincar\'e. In the context of EGT
there is already an extensive literature on the study
of boundary he\-te\-ro\-clinic cycles~\cite{Br1994,Ch1995,Ch1996,Hof1987,Hof1994,KS1994}. 
The dynamics along he\-te\-ro\-clinic networks has also been widely studied in the context of flows with symmetries, e.g.~\cite{KM95,KM04} or~\cite[Chapter 6]{Field}. 
All these studies use the Poincar\'e map itself~\cite[Chapter 6]{MR1691840} or its linearization~\cite[Chapter 17]{MR1635735}  usually for a stability or bifurcation analysis of families of vector fields. 
However the flows on polytopes addressed in our work have no symmetries in general and hence lie outside the scope of these studies.
In fact, the heteroclinic networks that we consider are robust because our focus is on flows that leave invariant all faces of the polytope.

Our 
method developed to analyze the dynamics along the vertex-edge heteroclinic network  is
applicable to a wide class of flows on polytopes, with few \textit{apriori} hypothesis on their dynamics.
For instance, some of the conditions in the previously referenced works can be reformulated in our setting as
appropriate assumptions on the  (computable) asymptotic dynamics. 
Moreover, as mentioned below, our method applies to heteroclinic networks with degenerate saddles,
in a setting which, as far as we know, is not covered by existing results.

The method presented here was first announced (without proofs) by the second author in~\cite{Du2011}. In his PhD thesis~\cite{Peixe2015} the third author
applied this method to the class of polymatrix replicators.
 In the following paragraphs, we give an overview of the method.

The Poincar\'e map defined along a heteroclinic or homoclinic orbit  is a composition of two types of maps,  the global  and the local Poincar\'e maps. 
The global ones, $P_\gamma$, are defined in  tubular neighborhoods of the flowing edges $\gamma$, see
Figure~\ref{local-global-poincare}. They map points between two cross sections $\Sigma_\gamma^-$ and $\Sigma_\gamma^+$ transversal to the flow along $\gamma$. 
The local ones, $P_v$, are defined in  neighborhoods of vertex singularities  $v$. For any pair of flowing  edges  $\gamma,\gamma'$ 
such that $v$ is both the endpoint of $\gamma'$ and the start-point of $\gamma$, the local map $P_v$ takes points from $\Sigma_{\gamma'}^+$
to  $\Sigma_{\gamma}^-$.

\begin{center}
\begin{figure}[h]
\includegraphics[width=6cm]{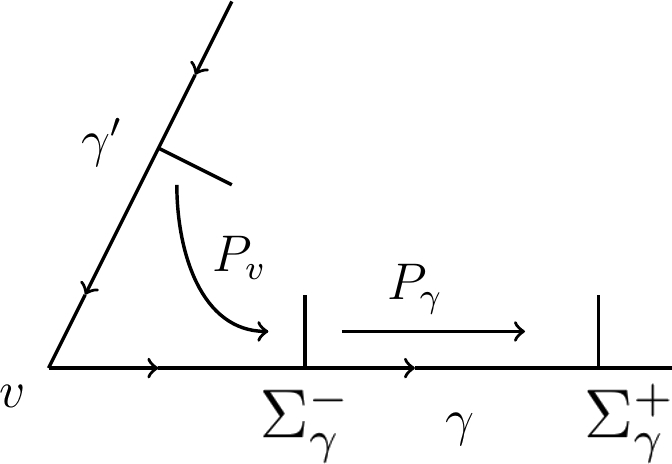}
\caption{The local, $P_{v}$, and global, $P_{\gamma}$, 
Poincar\'e maps along a heteroclinic orbit.} \label{local-global-poincare}
\end{figure}
\end{center}

The nonlinearities of the  global Poincar\'e maps $P_\gamma$ fade  away asymptotically, as one approaches the heteroclinic orbit, becoming  identity maps in the limit. 
Local data is extracted from the underlying vector field $X$ at the vertexes, which will be referred to as the {\em skeleton character} of $X$. The asymptotic behavior of the flow is completely determined by this  skeleton character and in particular the local Poincar\'e maps $P_v$  can be asymptotically linearized (in the sense that  near the vertexes, trajectories become lines after a change of variables) in terms of this data. 

In the generic case, the skeleton character of $X$ at a vertex $v$ consists of the eigenvalues of $X$ along the (eigen) directions of  edges through $v$. The signs of these eigenvalues were used to form   the characteristic matrix discussed in \cite[Chapter 17]{MR1635735}. The skeleton characters considered here are a bit more general.
The skeleton character of a \textit{corner} (a vertex $v$ and an edge containing it) may be non-zero even if the associated eigenvalue  is zero. This makes the method  applicable in certain degenerate situations with many zero eigenvalues.
One such example is the compatification of a Hamiltonion Lokta-Voltra system, where all eigenvalues along directions transversal  to the facet at infinity vanish~\cite{Du2011}.

To stage the asymptotic piecewise linear dynamics we introduce a geometric space  referred to as the {\em dual cone} of a polytope. This space is a subset of $\Rr^{F}$, where $F$ is the set of the polytope's facets. 
The dual cone will be a union of sectors\footnote{ We call sector to any closed convex cone bounded by $d$ transversal facets, where $d$ is the sector's dimension.}, one for each vertex of the polytope, see Figure~\ref{dualcone1}.
Given a vector field $X$  on a 
$d$ dimensional polytope $\Gamma^d\subset \Rr^d$, we describe a rescaling change of coordinates $\Psi_\epsilon^X$, depending on a blow-up parameter $\epsilon$, see Figure~\ref{asym-linearization}. This change of coordinates  maps  tubular neighborhoods of edges and vertexes to  the dual cone of $\Gamma^d$. 
For instance,  the tubular neighborhood $N_v$ of a vertex $v$ is defined by
$$ N_v:=\{ p\in\Gamma^d\colon  0\leq x_j(p) \leq 1 \,\text{ for }\, 1\leq j\leq d  \} $$
where $(x_1,\ldots, x_d)$ is a system of affine coordinates around $v$
which assigns   coordinates $(0,\ldots, 0)$ to $v$ and such that the hyperplanes $x_j=0$ are precisely the facets of the polytope through $v$. The sets $\{x_j=0\}\cap N_v$  are
referred to as  outer facets of $N_v$. The remaining facets of $N_v$, defined by equations like $x_i=1$, are called the inner facets of $N_v$.
The previous cross sections $\Sigma_\gamma^\pm$  are chosen to match these inner facets of the neighborhoods  $N_v$.

\begin{center}
	\begin{figure}[h]
		\includegraphics[width=10cm]{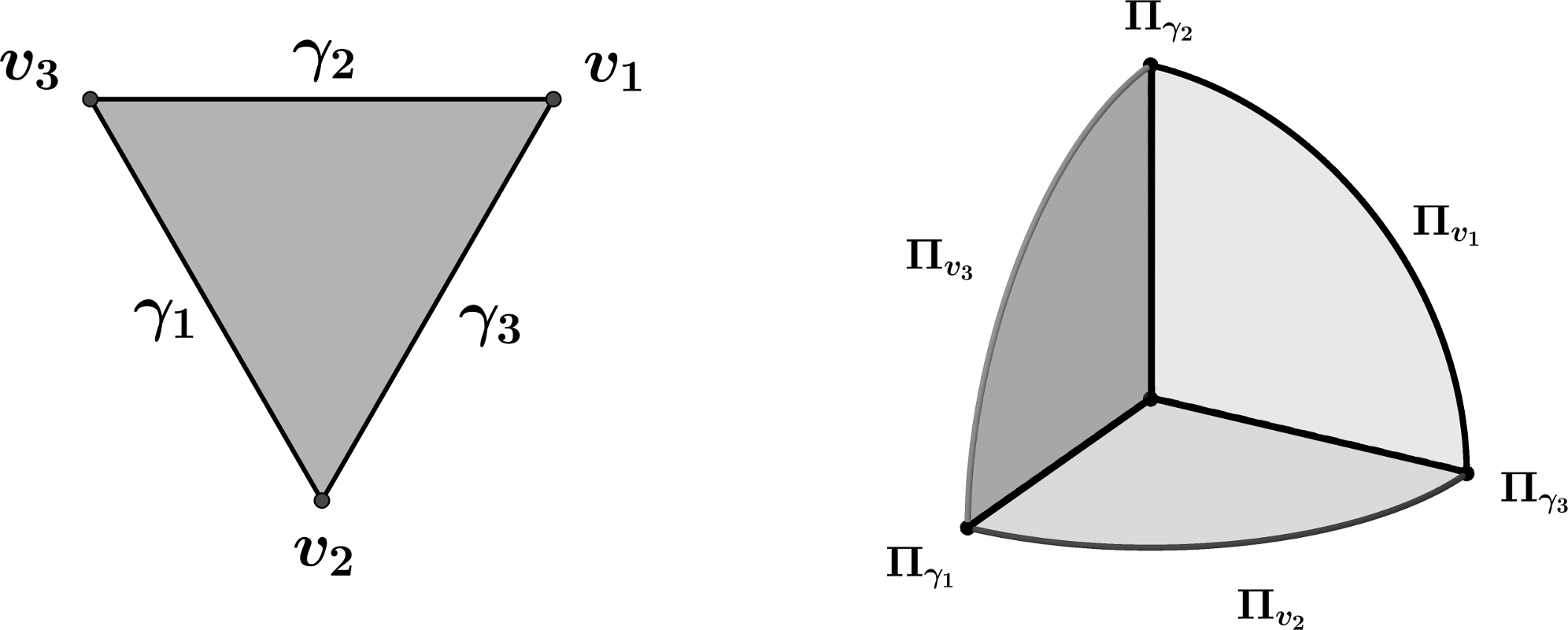}
		\caption{Dual cone of a triangle in $\Rr^{F}$} \label{dualcone1}
	\end{figure}
\end{center}


\begin{center}
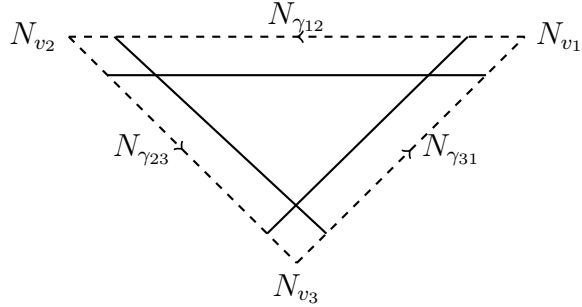
\begin{figure}[h]
\begin{tikzpicture}[scale=.3]
\draw[->-=.5,dashed,thick] (15,-20)--(25,-10);\draw[->-=.5,dashed,thick] (5,-10)--(15,-20);\draw[->-=.5,dashed,thick] (25,-10)--(5,-10);
\draw[thick] (13.7,-18.7)--(22.5,-10);
\draw[thick] (23.3,-11.7)--(6.7,-11.7);
\draw[thick] (16.3,-18.7)--(7,-10);
\node[below] at (15,-20) {$N_{v_3} $};\node[right] at (25,-10) { $N_{v_1} $};\node[left] at (5,-10) {$N_{v_2} $};
\node[above] at (15,-10.3) {$N_{\gamma_{12}} $};
\node[left] at (10,-15) {$N_{\gamma_{23}} $};
\node[right] at (20,-15) {$N_{\gamma_{31}} $};
\end{tikzpicture}
\caption{Tubular neighborhoods along the edges in a two dimensional polytope (dashed lines are the dges)} \label{tubular}
\end{figure}
\end{center}
The rescaling change of coordinates $\Psi^X_\epsilon$
takes points in  $N_v$  to points   in the sector $\Pi_v$ of all $y=(y_\sigma)_{\sigma\in F}\in \mathbb{R}^{F}$ with $y_\sigma\geq 0$ and 
where $y_\sigma=0$ whenever the facet $\sigma$ does not contain $v$.
In the generic case, assuming we have enumerated $F$  so that the  facets through $v$ are precisely   $\sigma_1,\ldots, \sigma_d$,  
 the map $\Psi^X_\epsilon$ is defined on the neighborhood $N_v$  by
$$ \Psi^X_\epsilon(q):=(-\epsilon^2\,\log x_1(q), \ldots,
-\epsilon^2\,\log x_d(q),0,\ldots, 0 )$$
where $(x_1,\ldots, x_d)$ stand for the system of affine coordinates introduced above.
Similarly, given an edge $\gamma$, the map  $\Psi^X_\epsilon$
takes points in the tubular neighborhood $N_\gamma$  of  $\gamma$ to points   in the sector $\Pi_\gamma$ of all $(y_\sigma)_{\sigma\in F}\in \mathbb{R}^{F}$ with $y_\sigma\geq 0$ and where $y_\sigma=0$ whenever the facet $\sigma$ does not contain $\gamma$.  The map $\Psi^X_\epsilon$ sends interior facets of $N_v$ and $N_\gamma$, respectively, to boundary facets of
$\Pi_v$ and $\Pi_\gamma$    while  it takes outer facets 
of $N_v$ and $N_\gamma$   to infinity. 
Figure~\ref{tubular} represents the tubular neighborhoods of vertexes and edges of a triangle, where dashed lines stand for   the outer boundary facets  of these neighborhoods.
Figure~\ref{after-log}
depicts the range of the map $\Psi^X_\epsilon$ where
the dashed lines stand for infinity.

\begin{center}
\begin{figure}[h] 
\begin{tikzpicture}[scale=.3]
\draw[thick] (15,-20)--(25,-10);
\draw[thick] (5,-10)--(15,-20);
\draw[thick] (25,-10)--(5,-10);

\draw[thick] (15,-20)--(20,-25);\draw[thick] (15,-20)--(10,-25);
\path[fill=gray] (15,-20)--(20,-25)--(10,-25)--(15,-20);\node[below] at (15,-25) {$\Pi_{v_3}$};
\draw[thick] (25,-10)--(25,-5);\draw[thick] (25,-10)--(30,-15);
\path[fill=gray]  (25,-10)--(25,-5)--(30,-15)--(25,-10);
\node[right] at (27,-10) {$\Pi_{v_1}$};
\draw[thick] (5,-10)--(0,-15);\draw[thick] (5,-10)--(5,-5);
\path[fill=gray] (5,-10)--(0,-15)--(5,-5)--(5,-10);
\node[left] at (1,-10) {$\Pi_{v_2}$};
\draw[->-=.5,dashed,thick] (20,-25)--(30,-15);\node[right] at (25,-20) {$\Pi_{\gamma_{31}}$};
\draw[->-=.5,thick,dashed] (25,-5)--(5,-5);\node[above] at (15,-5) {$\Pi_{\gamma_{12}}$};
\draw[->-=.5,dashed,thick] (0,-15)--(10,-25);\node[left] at (5,-20) {$\Pi_{\gamma_{23}}$};
\end{tikzpicture}
\caption{Range of the rescaling change of coordinates in the dual cone } \label{after-log}
\end{figure}
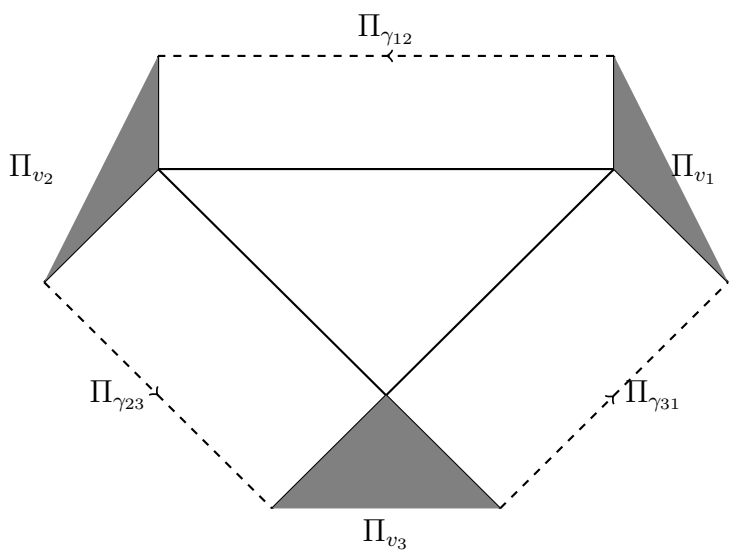
\end{center}

As the rescaling parameter $\epsilon$ tends to $0$, 
 the rescaled push-forward $\epsilon^{-2}\,(\Psi^X_\epsilon)_\ast X$ of the vector field $X$ converges to a constant vector field $\chi^v$  on each sector $\Pi_v$. This means that asymptotically, as $\epsilon\to 0$, trajectories become lines in the coordinates   $(y_\sigma)_{\sigma\in F}=\Psi^X_\epsilon$.
 Given a flowing edge $\gamma$ between vertexes $v$ and $v'$, the map $\Psi^X_\epsilon$ over $N_\gamma$ depends only on the  coordinates 
transversal to $\gamma$. Moreover,
as  $\epsilon\to 0$ the global Poincar\'e map $P_\gamma$ converges to the identity map in the coordinates $(y_\sigma)_{\sigma\in F} =\Psi^X_\epsilon$.
 Hence the sector
$\Pi_\gamma$ is  naturally  identified as the common facet between the sectors $\Pi_v$ and $\Pi_{v'}$, see Figure~\ref{after-log}.
In fact, from the above definitions one gets that  
$ \Pi_\gamma= \Pi_v\cap \Pi_{v'}$. Hence the asymptotic dynamics along the vertex-edge heteroclinic network is completely determined by the vector field's geometry near the vertex singularities and can be 
described by a piecewise constant vector field $\chi$ on the dual cone, whose components are precisely those of the skeleton character of  $X$. We  refer to this piecewise constant vector field as the {\em skeleton vector field} of $X$.  This  vector field $\chi$ induces a piecewise linear flow on the dual cone   whose dynamics can be  computationally explored.

The flows associated with these piecewise constant vector fields are in general  open dynamical systems. Some of them may have no recurrence at all. For instance attracting or repelling vertex singularities, or the existence of attracting or repelling singularities interior to non flowing edges,
may divert trajectories and prevent the existence of cycles in the  vertex-edge heteroclinic network.

We use  Poincar\'e maps to analyze the asymptotic dynamics of the flow of $X$.
For that we consider a subset $S$ of flowing edges such that every vertex-edge heteroclinic cycle goes through at least one edge in $S$. 
We call structural set to any such set $S$, see Definition~\ref{structural:set}.
Then the flow of $X$ induces a Poincar\'e  map $P_S$
on the system of cross sections $\Sigma_S:= \cup_{\gamma\in S}\Sigma^-_\gamma$. Each branch of the Poincar\'e  map $P_S$ is associated with a vertex-edge heteroclinic path  starting with an edge in $S$ and ending at its first return to another edge in $S$. These heteroclinic paths are referred to as branches of $S$. Similarly, the flow of the skeleton vector field $\chi$ induces a  first return map
$\pi_S:D_S\subset \Pi_S\to\Pi_S$ on the system of cross section $\Pi_S:= \cup_{\gamma\in S}\Pi_\gamma$ of the dual cone. This map $\pi_S$,
referred to as the \textit{skeleton flow map}, is piecewise linear and its domain is a finite union of open convex cones, one for each branch of $S$.
Proposition~\ref{partition} provides a simple sufficient condition 
for the domain $D_S$ of $\pi_S$ to have full Lebesgue measure in $\Pi_S$. In this sense  $\pi_S$ becomes a closed dynamical system.  We emphasize that the hypothesis of this proposition is not a requisite for the applicability of our method.
Violation of the  hypothesis simply allows the existence of open convex cones  in $\Pi_S\setminus D_S$ corresponding to orbits   which never return to $\Pi_S$.

The main result of this manuscript asserts that the 
Poincar\'e map $P_S$ in the rescaled coordinates
$\Psi^X_\epsilon$ converges in the $C^\infty$ topology to
the skeleton flow map $\pi_S$. 
More precisely, the following limit holds
$$ \lim_{\epsilon\to 0 }  \Psi^X_\epsilon \circ \Poin{X}{S}\circ
(\Psi^X_\epsilon)^{-1} = \pi_S  $$
with uniform convergence of the map and its derivatives over any compact set contained in the (open) domain $D_S\subset \Pi_S$,
see  Theorem~\ref{asymp:main:theorem}.

 Consider now, for each facet $\sigma$ of the polytope, an affine function $\Rr^d\ni q\mapsto x_\sigma(q)\in\Rr$ which vanishes on $\sigma$  and is strictly positive on the rest of the polytope. With this family of affine functions we can present the polytope  as 
 $\Gamma^d=\cap_{\sigma\in F} \{  x_\sigma \geq 0  \}$.
In the generic case any function $h:\inter(\Gamma^d)\to\Rr$ of the form
$$h(q)=\sum_{\sigma\in F} c_\sigma\, \log x_\sigma(q)\quad (c_\sigma\in\Rr) $$
rescales to the following  piecewise linear function on the dual cone  
$$\eta(y):=\sum_{\sigma\in F} c_\sigma\, y_\sigma $$
in the sense that
$\eta = \lim_{\epsilon\to 0} \epsilon^{-2}\, (h\circ
(\Psi^X_\epsilon)^{-1} ) $.
When all coefficients $c_\sigma$ have the same sign
then $\eta$ is a proper function on the dual cone.
This means in particular that all levels of $\eta$ are compact sets.
Finally, if the function $h$ is invariant under the flow
of  $X$, \ie
$h\circ P_S=h$ then the piecewise linear function
$\eta$ is also invariant under th skeleton flow,
\ie $\eta\circ \pi_S=\eta$.
Thus integrals of motion of conservative systems, which have the previous form, carry over as piecewise linear  integrals of the skeleton flow.

As a general principle, any robust structure invariant under the skeleton flow map persists as an invariant structure for the Poincar\'e map  of the original flow. 
Since the former can be detected through linear algebra tools (e.g. algorithms for computing eigenvalues and eigenvectors of the skeleton flow map's branches)
this approach provides a method to analyze the dynamics of the original flow along the vertex-edge heteroclinic network, a method which can be equally well applied to higher dimensional cases. For Hamiltonian systems, to be discussed in a sequel paper, their conservative nature (in the context of Poisson geometry)  is inherited by the skeleton flow map. In these cases the analysis of the dynamics reduces to the dimension of the symplectic leaves. 
We provide here a couple of Hamiltonian examples where this method  proves the co-existence of chaotic behavior with elliptic islands.

Embedding the dual cone in the Euclidean space
$\Rr^{F}$ is formally and computationally convenient.
 Poincar\'e maps of the skeleton vector field along paths of the heteroclinic network  are  represented  by  $F\times F$ flow matrices on convex cone domains which can be explicitly determined. Both these flow matrices and their convex cone domains are expressed in terms of
the vector field's skeleton character.

Next we provide an alternative and more geometric realization of the
dual cone  as the {\em normal fan} of a polytope~\cite[Chapter 7]{Ziegle1995}. A (complete) fan is roughly  a family  of polyhedral  convex closed cones\footnote{ The precise definition of fan requires that	the intersection of any two family members is either empty of else  a common face  and that faces of family members are also in the family. } 
 in some Euclidean space $\Rr^d$ with disjoint interiors and such that their union is the whole space.  The {\em normal cone} of a polytope $\Gamma^d\subset \Rr^d$ at a vertex $v$ is the closed convex cone
$$ \Pi_v:=\{ u\in\Rr^d: u\cdot (q-v)\leq 0, \forall\; q\in\Gamma^d \} .$$
The family of all normal cones of a polytope's vertexes (with all their faces) is always a complete fan, referred to as its {\em normal fan}.
Figure~\ref{fan} shows the normal fan of a triangle in $\Rr^2$.   Given a vertex $v$ of $\Gamma^d$ let $(x_1,\ldots, x_d)$ be
the previously mentioned  system of affine coordinates around $v$.
Let $\vec n_i\in\Rr^d$ be the unit outward normal to the facet of
$\Gamma^d$ represented by the equation $x_i=0$. Then the restriction of the rescaling map $\Psi^X_\epsilon$ to the neighborhood  $N_v$
with values in the normal cone $\Pi_v$ is defined in the generic case by 
$$ \Psi^X_\epsilon(q):= -\epsilon^2\,\sum_{j=1}^d
\log x_i(q)\, \vec n_j . $$
In this construction, the skeleton vector field of $X$ is a piecewise constant vector field on $\Rr^d$, \ie  one which is constant on each normal cone $\Pi_v$ for a vertex $v$ of $\Gamma^d$.
\begin{center}
	\begin{figure}[h]
		\includegraphics[width=8cm]{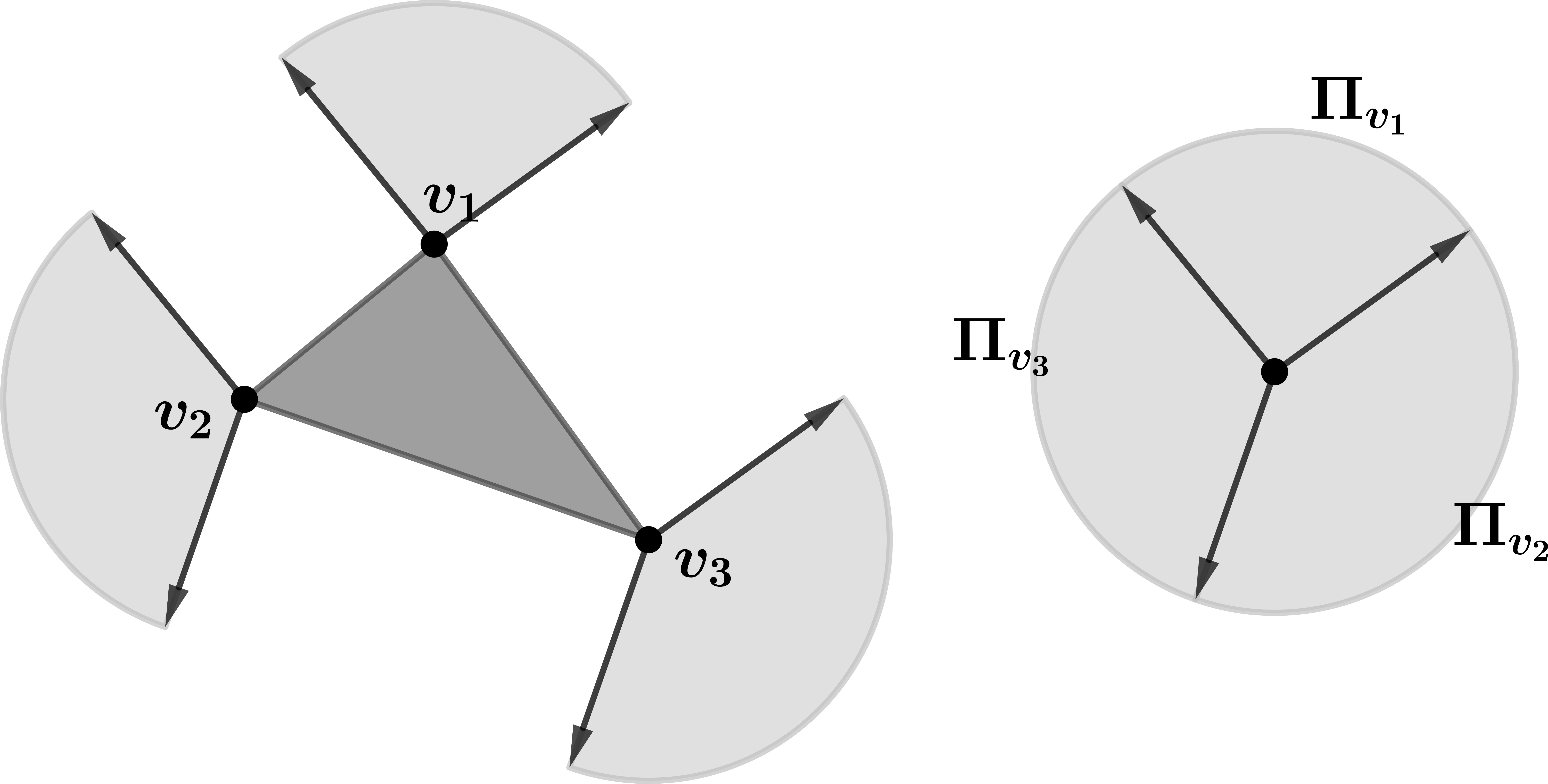}
		\caption{ Normal fan of a triangle in $\Rr^2$.  } 
		\label{fan}
	\end{figure}
\end{center}
Figure~\ref{cons222}  illustrates  a  Hamiltonian vector field $X$ (a
polymatrix replicator system) on the standard $3$-dimensional cube,
with a proper Hamiltonian function $h$.
The left  of Figure~\ref{cons222} depicts the cube with a few orbits of $X$ on some  level set of $h$. 
As mentioned above, the function $h$ rescales to a 
 piecewise linear proper function $\eta$ on the normal fan of the  cube. All level sets of  $\eta$ are octahedra (the cube's dual).
 On the right of Figure~\ref{cons222},  a few orbits of the skeleton flow on some level of the invariant function $\eta$ are shown.
\begin{center}
	\begin{figure}[h]
		\includegraphics[width=12cm]{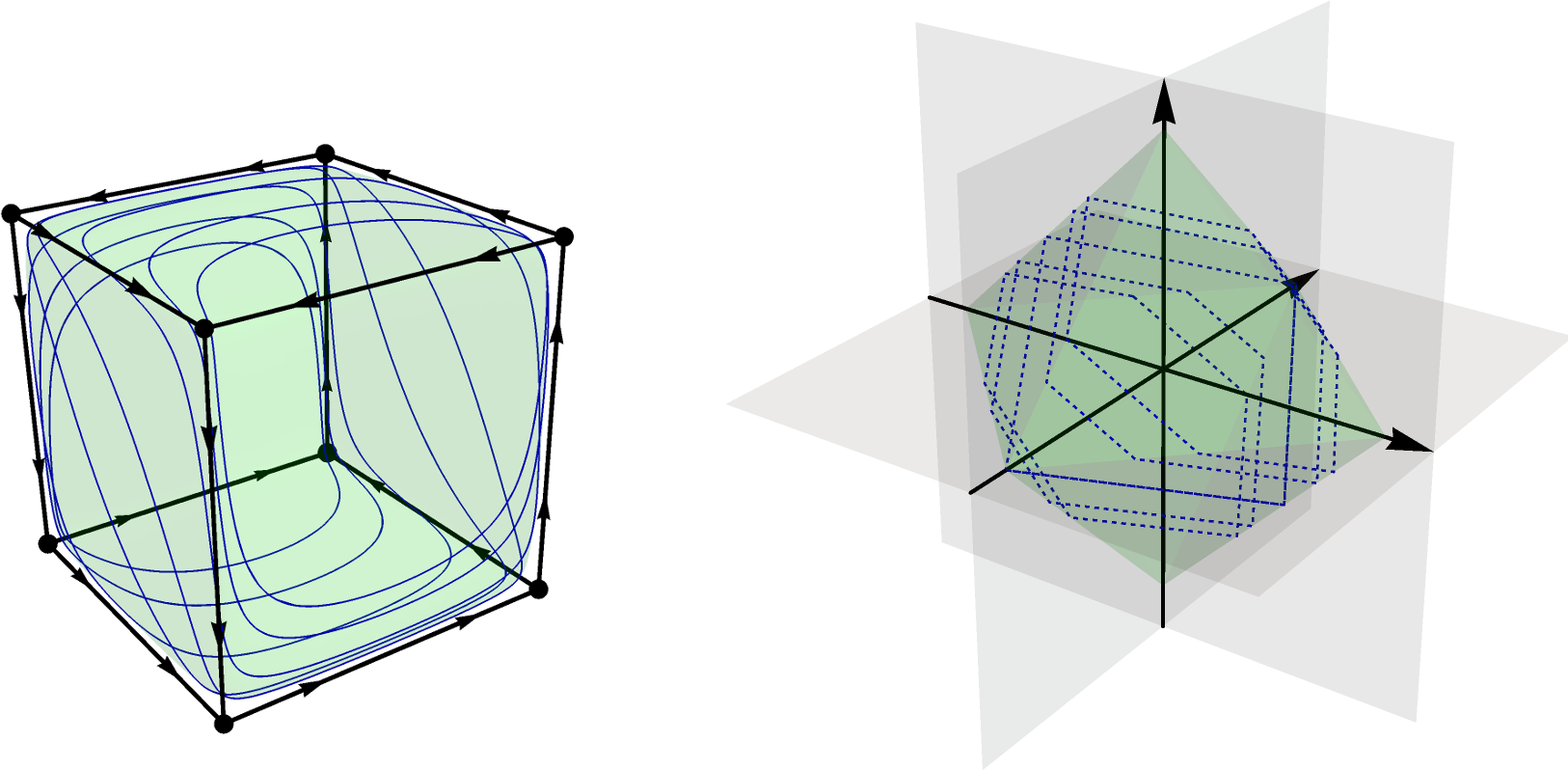}
		\caption{Asymptotic linearization on the normal fan.}\label{asym-linearization}
		\label{cons222}
	\end{figure}
\end{center}


\begin{center}
\begin{figure}[h]
\includegraphics[width=11cm]{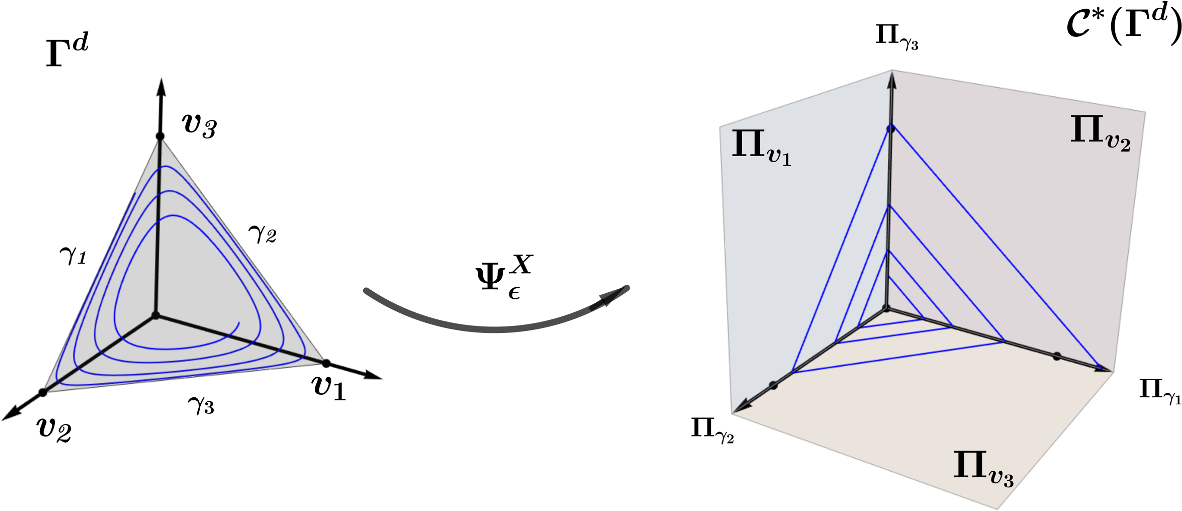}
\caption{Asymptotic linearization on the dual cone.
The left image represents an orbit on the simplex $\Delta^2$ and the right one
the corresponding image under $\Psi_\epsilon^X$ on the dual cone.
}\label{asym-linearization} 
\end{figure}
\end{center}

All graphics of this manuscript were produced with \textit{Mathematica} and \textit{Geogebra} software. 
We provide the \textit{Mathematica} code~\cite{MathematicaCodeDAE} used to analyze the examples and to make the graphics. This code can be used to  numerically analyze specific examples, providing hints for analytic results. 

This paper is organized as follows.
In Section~\ref{polytopes}  we define polytopes and all their associated notations, terminology and concepts.
In Section~\ref{vectorfields}  we introduce the class of vector fields on polytopes,  the skeleton character of a vector field and other related concepts.
In Section~\ref{rescaling}, we define the family of rescaling coordinates $\Psi_\epsilon$ and the dual cone of a polytope.
In Section~\ref{skeleton} we introduce the class of skeleton vector fields (piecewise constant vector fields) on the dual cone, whose dynamics encapsulate the asymptotic behavior of the original non-linear  flow.
We also define the concept of structural set  and characterize those vector fields whose skeleton flow map is a closed dynamical system.
In Section~\ref{asymptotics} we define the  Poincar\'e return maps 
of a vector field, and then state and prove the main thorem, Theorem~ \ref{asymp:main:theorem}.
In Section~\ref{AFI} we introduce a probe space of integrals of motion, describing their asymptotics on the dual cone of the polytope.
We also describe a sufficient condition on the skeleton flow map for the existence of horse-shes regarding  the dynamics of the original vector field, see  Theorem~ \ref{thm horse-shoes}.
In Section~\ref{dyn-analysis} we summarize a procedure to
detect chaotic behavior by checking the assumptions of Theorem~ \ref{thm horse-shoes}.
In Section~\ref{example section} we describe a couple of replicator Hamiltonian  examples in the five dimensional simplex, where the previous procedure  is applied.
Finally, in Section~\ref{furtherwork} we discuss a few possible developments of this work.
 
%

\section{Polytopes}
\label{polytopes}
In this section we provide preliminary definitions and notations
about polytopes.

Given a convex subset $K\subseteq \Rr^N$,
we call \textit{affine support} of $K$ to the the affine subspace spanned by $K$. The \textit{dimension} of $K$, $\dim(K)$, is by definition  
the dimension of its affine support. 

We call {\em polytope} to any compact intersection of finitely many half-spaces.
A {\em face} of a polytope $K$ is any non-empty intersection of $K$ with a hyperplane that does not separate $K$, 
i.e., such that $K$ does not have points on both open sides of the hyperplane. 
A face of $K$ with dimension $0\leq j\leq \dim(K)$ will be referred to as a $j$-face of $K$.
A {\em vertex} of $K$ is any $0$-face of $K$. An {\em edge} of $K$ is any $1$-face of $K$. 
We call {\em facet} of $K$
to any $(d-1)$-face where $d=\dim(K)$. A polytope $K$ of dimension $d$ is called {\em simple} if every vertex belongs exactly to $d$ facets (and hence also to $d$-edges). From now on all polytopes will be simple polytopes.

\begin{defn}\label{defn:polytope}
Given a simple polytope $\Gamma^d$ with affine support $E\subset\Rr^N$ and  $d=\dim(E)$  
 we call \emph{defining family} of $\Gamma^d$ any family of affine functions $\{f_i:E\to\Rr\}_{i\in I}$ such that
\begin{itemize}
\item[(a)]$\Gamma^d=\cap_{i\in I}f_i^{-1}([0,+\infty[)$.
\item[(b)] $\Gamma^d\cap f_i^{-1}(0)\neq\emptyset\quad\quad \forall i\in I$.
\item[(c)] Given $J\subseteq I$ such that $\Gamma^d\cap(\cap_{j\in J}f_j^{-1}(0))\neq\emptyset$, the linear $1$-forms $(\d f_j)_p$ are linearly independent at every point $p\in\cap_{j\in J}f_j^{-1}(0)$.
\end{itemize}
\end{defn}

Given $J\subset I$, because of item (c),
if non-empty, the intersection $\Gamma^d\cap(\cap_{j\in J}f_j^{-1}(0)) $ is a $(d-\abs{J})$-dimensional face of the polytope.
In particular for each $i\in I$, the set  $\sigma_i:=\Gamma^d\cap f_i^{-1}(0)$
is   a \textit{facet} of the polytope.
We denote the sets of vertexes, edges and facets, respectively, by $V$, $E$ and $F$. Since
$F=\{\sigma_i : i\in I\}\simeq I$, we will assume from now on that
the defining family of $\Gamma^d$ is indexed in $F$ in a way that
$\sigma= \Gamma^d\cap f_\sigma^{-1}(0)$, for all $\sigma\in F$.

Given a vertex $v$, the sets $F_v$ and $E_v$ are defined to be the set of all facets, respectively edges,  which contain $v$. 

\begin{remark}\label{notation-coordinate}
By (c) of Definition~\ref{defn:polytope} at any given vertex $v$, the co-vectors $(df_i)_v$ are linearly independent. So in a small enough neighborhood of $v$ the functions $\{f_\sigma\}_{\sigma\in F_v}$ can be used as a system of coordinates. 
\end{remark}

\begin{remark}
We have adopted here the standard terminology where
a {\em polyhedron} is any convex set bounded by finitely many hyperplanes and a {\em polytope} is a compact polyhedron,
see for instance~\cite{Ziegle1995}. We note that in~\cite{Du2011} the term `polyhedron' was used to mean compact polyhedron.
\end{remark}

The elements of the set 
$$C:=\{\, (v,\gamma,\sigma)\in V\times E\times F\,\colon\,  \gamma\cap \sigma= \{v\}\, \}$$ are referred to as \emph{corners}, see Figure~\ref{corner}.
\begin{center}
\begin{figure}[h]
\includegraphics[width=5cm]{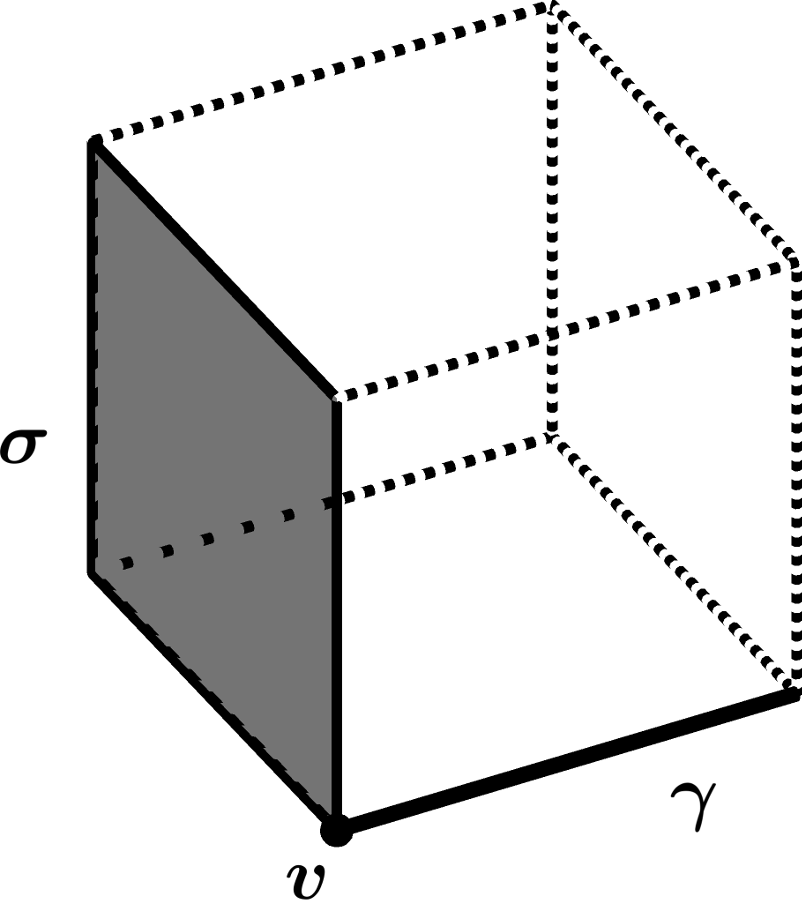}
\caption{A corner $(v,\gamma,\sigma)$ in a three dimensional polytope.} \label{corner}
\end{figure}
\end{center}

\begin{remark}
	\label{corner rmk}
Any pair of the elements in a corner  uniquely determines the third one. Therefore, we will sometimes refer to the corner $(v,\gamma,\sigma)$ shortly as $(v,\gamma)$ or $(v,\sigma)$. An edge $\gamma$ with endpoints $v,v'$ determines two corners $(v,\gamma,\sigma)$ and $(v,\gamma,\sigma')$, referred to as the
{\em end corners of} $\gamma$.  The facets $\sigma,\sigma'$ will be referred to as the {\em opposite facets of} $\gamma$. 
\end{remark}

\begin{example}
The $d$-dimensional simplex is the polytope defined by
$$ \Delta^d:=\left\{ (x_0,x_1,\ldots, x_d) : x_j\geq 0,\; \sum_{j=0}^d x_j=1\,\right\} . $$
\end{example}

The affine support of $\Delta^d$ is the hyperplane 
$$ E^d:=\left\{ (x_0,x_1,\ldots, x_d)\in\Rr^{d+1}  :  \sum_{j=0}^d x_j=1\,\right\} . $$
The defining family of $\Delta^d$ are the coordinate functions
$f_i:E^d\to\Rr$, $f_i(x_0,x_1,\ldots, x_d)=x_i$. The  simplex $\Delta^d$ has $d+1$ vertexes $v_0, v_1,\ldots, v_d$ and $d+1$ facets
$\sigma_0,\sigma_1,\ldots, \sigma_d$, where $v_j=(0,\ldots, 1,\ldots, 0)$ is the vertex opposed to the facet $\sigma_j=\Delta^d\cap\{x_j=0\}$ for each $j=0,1,\ldots, d$.

\section{Vector Fields on Polytopes}
\label{vectorfields}

In this section we introduce the general class of vector fields on polytopes to which our theory applies.

Let $\Gamma^d$ be a simple $d$-dimensional polytope.
A function $f:\Gamma^d\to\Rr$ is said to be {\em analytic}
if it can be analytically extended to a neighborhood of $\Gamma^d$.
We denote by $\A(\Gamma^d)$ the space of all analytic functions on $\Gamma^d$.
Similarly, we denote by $\fX(\Gamma^d)$
 the space of all analytic vector fields  $X:\Gamma^d\to \Rr^N$ 
 such that for every face $\rho\subset \Gamma^d$ and all $x\in\rho$,  the vector $X(x)$ is tangent to $\rho$. 
  This tangency requirement on the vector fields   $X\in \fX(\Gamma^d)$ implies that for every facet $\sigma \in F$, $\d f_{\sigma}(X)=0$ along $\sigma$.
By compactness the flow $\varphi_X^t$ of any vector field $X\in \fX(\Gamma^d)$ is complete on $\Gamma^{d}$ with singularities at the vertexes of the polytope.

Given a vertex $v$,  consider the coordinate system introduced in Remark~\ref{notation-coordinate},
$(x_1,\ldots, x_d)=(f_{\sigma_1}(q),\ldots, f_{\sigma_d}(q))$ 
where $F_v=\{\sigma_1,\ldots, \sigma_d\}$.
In these coordinates the analytic function $df_{\sigma_l}(X)$   vanishes along the hyperplane $x_l=0$. By Weierstrass division theorem either there exist a positive integer $\nu_l=\nu(X,\sigma_l),$ and the function $H_{\sigma_l}\in\A(\Gamma^d)$ which is non-identically zero along the face $\sigma_l$ and such that
\begin{equation}\label{def:order}
\d f_{\sigma_l}(X)=(f_\sigma)^{\nu_l}H_{\sigma_l},\quad \text{i.e.}\quad \dot{x_l}=x_l^{\nu_l} H_{\sigma_l} ,
\end{equation}
or else $\d f_{\sigma_l}(X)$  is identically zero. In the later case, we set  $\nu_l=\infty$. 
We say that \emph{$X$ has tangency contact of order} $\nu(X,\sigma)$ with $\sigma$ and will refer to it as the order of $X$ at the facet $\sigma$. The map $$\nu:\fX(\Gamma^d)\times F\to \{1,2,3,\ldots,\infty\}$$ is called \emph{order function} of $X$.

\begin{remark}
We have assumed analyticity for the sake of simplicity, also because the EGT models we have in mind are analytic (and even algebraic) vector fields. The results obtained in this work extend easily to smooth flows and vector fields. The main difference  is that for a smooth vector field $X$ the concept of order must first be defined locally.\footnote{
For a smooth vector field $X$, the order $\nu(X,v,\sigma)$  at a corner 
 $(v,\sigma)$ is the minimum integer $k\geq 1$
such that $ (\d f_\sigma)_v( D^{k} X)_v\neq 0$.
The order of  $\sigma$ is defined as
$$ \nu(X,\sigma) := \min\{ \nu(X,v,\sigma)\,\colon\,
\sigma\in F_v\,\}\;. $$
}
\end{remark}

For every corner $(v,\sigma,\gamma)$ there exists a unique vector $e_{(v,\sigma)}$ tangent to $\gamma$ at $v$
such that $(\d f_{\sigma})_v( e_{(v,\sigma)})=1$ and  for any other facet $\sigma'\in F_v$, $\sigma'\neq \sigma$,   $(df_{\sigma'})_v(e_{(v,\sigma)})=0$.
Hence,  
$\{\e_{(v,\sigma)}\}_{\sigma\in F_v}$ is the dual basis
 of the $1$-form  basis  $\{(d f_\sigma)_v\}_{\sigma\in F_v}$. 
The vectors $\e_{(v,\sigma)}$ are eigenvectors of the derivative $D X_v$.
If $\nu(X,\sigma)=1$ then $H_\sigma(v)$ is the eigenvalue of the derivative $D X_v$ associated to $e_{(v,\sigma)}$. In the case   
 $\nu=\nu(X,\sigma)\geq 2$, the eigenvalue associated to $e_{(v,\sigma)}$ is zero but we have 
 $$
 H_\sigma(v)=\frac{1}{\nu!}  (\d  f_\sigma)_v\, (D^{\nu}X)_v(\underbrace{e_{(v,\sigma)},\ldots,e_{(v,\sigma)}}_{\nu\; \text{ times} }) \;.
 $$
To see this consider the coordinate system
introduced in Remark~\ref{notation-coordinate},
$(x_1,\ldots, x_d)=(f_{\sigma_1}(q),\ldots, f_{\sigma_d}(q))$,
where $F_v=\{\sigma_1,\ldots, \sigma_d\}$. 
Then the $l^{\text{th}}$  component of the vector field $X$ is
$X_l(x)=x_l^{\nu}\, H_{\sigma_l}(x)$ and we have
$$ H_{\sigma_l}(v)= \frac{1}{\nu!}\,\frac{\partial^\nu X_l}{\partial x_l^\nu}(0) =\frac{1}{\nu!}  (\d  f_{\sigma_l})_v\, (D^{\nu}X)_v( e_{(v,\sigma_l)},\ldots,e_{(v,\sigma_l)} ) . $$ 
 
\begin{defn}\label{skeletoncharacter}
 The \emph{skeleton character} of  
 $X\in\fX(\Gamma^d)$ is defined to be the matrix
  $\chi:=(\chi^v_\sigma)_{(v,\sigma)\in V\times F}$ where 
 \begin{equation*}
 \chi^v_\sigma:=
 \left\{\begin{array}{ccc}-H_\sigma(v)&\quad&\sigma\in F_v\\0&\quad&\mbox{otherwise}\end{array}\right.\;.
 \end{equation*}
We set  $\chi^v_\sigma=0$ when $\nu(X,\sigma)=\infty$. For a fixed vertex $v$, the vector $\chi^v:= (\chi^v_\sigma)_{\sigma\in F}$ is referred to as the \emph{skeleton character} at $v$.
\end{defn}

\begin{remark}\label{sign-character}
For a given corner $(v,\gamma,\sigma)$ if $\chi^v_\sigma<0$   then $v$  is the $\alpha$-limit of an orbit in $\gamma$, and if $\chi^v_\sigma>0$  then $v$  is the $\omega$-limit  of an orbit in $\gamma$. Assuming that $X$ does not have singularities in the interior of an edge $\gamma$, if $\gamma$ connects the corners $(v,\sigma)$ and $(v',\sigma')$, then it consists of a single a heteroclinic orbit with $\alpha$-limit $v$ and $\omega$-limit $v'$  if and only if $\chi^{v}_{\sigma}<0$ and $\chi^{v'}_{\sigma'}>0$. 
\end{remark}
\bigskip

The  replicator equation provides a class of analytic vector 
fields in the space $\fX(\Delta^d)$. In the rest of this section we recall this equation and describe its skeleton character and order function.

Given a payoff matrix $A\in\Mat_{d+1}(\Rr)$ the system of differential equations
\begin{equation}
\label{rep}
\frac{d x_i}{dt} = x_i\,\left( (A x)_i - \sum_{k=0}^d x_k \, (A x)_k \right), \; 0\leq i \leq d
\end{equation}
is called the  {\em replicator equation}. The associated vector field
$X_A$ is called the {\em replicator vector field} of $A$ and lies in our class, $X_A\in \fX(\Delta^{d})$.
For a brief interpretation of this equation consider a population whose individuals interact with each other according to the set of pure strategies $\{0,\ldots, d\}$. A point  $x=(x_0,\ldots, x_d)\in \Delta^{d}$ represents a state of the population where $x_i$ measures the frequency of usage of strategy $i$.
Each entry $a_{ij}$ of $A$ represents the payoff of strategy $i$ against $j$ and this model governs the time evolution of the frequency distribution  of each pure strategy. 
The equation says that the growth rate of each frequency $x_i$ is the difference between its
payoff $(A x)_i=\sum_{j=0}^d a_{ij} x_j$  and the  average population's payoff  $\sum_{k=0}^d x_k \, (A x)_k$.

The next proposition characterizes the skeleton character of  $X_A$.

\begin{proposition} 
\label{repli sk char}
Given $A\in \Mat_{d+1}(\Rr)$,  every facet $\sigma_i$ of $\Delta^{d}$ has order $1$, $2$ or $\infty$.
More precisely
\begin{enumerate}
\item  $\nu(X_A,\sigma_i)=1$  iff 
$a_{i j}\neq a_{j j}$  for some $j$  or else
$(a_{k j} -a_{j j})_{k,j\neq i}$
is  not skew-symmetric. 
In this case $ \chi^{v_j}_{\sigma_i}= a_{j j} - a_{ i j}$ for all $j$.

\item  $\nu(X_A,\sigma_i)=2$  iff 
$a_{i j}=a_{j j}$  for all $j$  and
$(a_{k j} -a_{j j})_{k,j\neq i}$
is  skew-symmetric, but $(a_{k j} -a_{j j})_{k,j}$
is  not skew-symmetric. 
In this case $ \chi^{v_j}_{\sigma_i}= a_{j i}-a_{i i} $ for all $j$.

\item $\nu(X_A,\sigma_i)=\infty$  iff 
$a_{i j}=a_{j j}$  for all $j$  and
$(a_{k j} -a_{j j})_{k,j}$
is  skew-symmetric. 
In this case 
$ \chi^{v_j}_{\sigma_i}=0$ for all $j$.
\end{enumerate}
\end{proposition}

\begin{proof}
Consider the conditions
\begin{enumerate}
\item[(C1)]   
$a_{i j}\neq a_{j j}$  for some $j$  or else
$(a_{k j} -a_{j j})_{k,j\neq i}$
is  not skew-symmetric.

\item[(C2)]  $a_{i j}=a_{j j}$  for all $j$  and
$(a_{k j} -a_{j j})_{k,j\neq i}$
is  skew-symmetric, but $(a_{k j} -a_{j j})_{k,j}$
is not skew-symmetric.

\item[(C3)]  $a_{i j}=a_{j j}$  for all $j$  and
$(a_{k j} -a_{j j})_{k,j}$
is  skew-symmetric. 
\end{enumerate}

It is clear that (C1), (C2) and (C3) are exhaustive and mutually exclusive conditions. Hence it is enough to prove that (C1) $\Rightarrow$ $\nu(X_A,\sigma_i)=1$, (C2) $\Rightarrow$ $\nu(X_A,\sigma_i)=2$ and (C3) $\Rightarrow$ $\nu(X_A,\sigma_i)=\infty$.

Let
$H_i:\Delta^{d}\to\Rr$ be the function
$$ H_i(x):=  (A x)_i - \sum_{k=0}^d x_k \, (A x)_k .$$
A simple computation shows that
$$ H_i(v_j)= a_{i j} - a_{j j} $$
where the $v_j$ are the vertexes of $\Delta^{d}$.
Thus, if for some $j$, $a_{i j}\neq a_{j j}$  then  $\nu(X_A,\sigma_i)=1$ and  
$ \chi^{v_j}_{\sigma_i}= a_{j j} - a_{ i j}$ for all $j$.
Assume now that (C1) holds and let
$\tilde A=(a_{kj}-a_{jj})_{k,j}$. Then
$H_i(x) =  (\tilde A x)_i - \sum_{k=0}^d x_k \, (\tilde A x)_k$.
If   $a_{i j}= a_{j j}$  for all  $j$ then 
$(\tilde A x)_i=0$ for all $x\in\Delta^{d}$.
Also, if   $a_{i j}= a_{j j}$  for all  $j$ then 
 the matrix $(\tilde a_{k j})_{k,j\neq i}$
is  not skew-symmetric. 
This implies that the closed cone
$C_i$ defined by the conditions $x_i=0$ and $x^T \tilde  A x=0$
 has zero  Lebesgue measure in the hyperplane
 $\{x_i=0\}\subset \Rr^{d+1}$.
Therefore $C_i\cap \sigma_i$ has zero  Lebesgue measure in the facet
$\sigma_i$, which implies that $H_i(x)=-x^T \tilde A x$ is not identically zero on $\sigma_i$. 
Hence  $\nu(X_A,\sigma_i)=1$ and  
$ \chi^{v_j}_{\sigma_i}= -H_i(v_j)=  a_{j j} - a_{ i j}=0$ for all $j$.

Assuming (C2) holds we have $(\tilde A x)_i=0$
and 
$$H_i(x)= -x^T \tilde A x = -x_i\sum_{j=0}^d (a_{j i}-a_{i i})\,x_j.$$
Because $(a_{k j} -a_{j j})_{k,j}$
is not skew-symmetric  we have $a_{j i}\neq a_{i i}$ for some $j$.
Thus $\nu(X_A,\sigma_i)=2$ and
$ \chi^{v_j}_{\sigma_i}=   a_{j i} - a_{ i i}$ in this case.

Finally, if (C3) holds then $H_i\equiv 0$, which implies 
$\nu(X_A,\sigma_i)=\infty$.
\end{proof}

\section{Rescaling Coordinates}
\label{rescaling}
In this section we define the dual cone of a polytope
and introduce the family of rescaling  coordinates $\Psi^X_\epsilon$
described  in the introduction.

Consider a polytope $\Gamma^d$ and its defining family $\{f_\sigma\}_{\sigma\in F}$, see Definition~\ref{defn:polytope}. 
By Remark~\ref{notation-coordinate}, the co-vectors $\{(\d f_\sigma)_v: \sigma\in F_v\}$ are linearly independent  at every vertex $v$. Multiplying each affine function of this family by some large positive number  we may assume that the neighborhoods 
\[N_v:=\{q\in \Gamma^d :\,\, f_\sigma(q)\leq 1, \; \forall \sigma\in F_v\}, \]
with  $v\in V$, are pairwise disjoint, and  that the functions $\{f_\sigma: \sigma\in F_v\}$ define a coordinate system for $\Gamma^d$ on $N_v$. For any edge $\gamma$ connecting two vertexes $v,v'\in V$ we can define a tubular neighborhood  connecting $N_{v}$ to $N_{v'}$ by 
\[N_{\gamma}:=\{q\in\Gamma^d\backslash (N_{v}\cup N_{v'}):f_\sigma(q)\leq 1\,\text{ for all }\, \gamma\subset \sigma\}. \]
As before, we may assume that these neighborhoods are pairwise disjoint between themselves. Furthermore,
fixing a smooth submersion $t:\overline{N_\gamma}\to [0,1]$ such that
$t^{-1}(0)\subset \partial N_{v}$ and
$t^{-1}(1)\subset \partial N_{v'}$, whose restriction induces
a diffeomorphism between 
$\gamma\setminus \inter(N_v\cup N_{v'})$ and $[0,1]$, the family of functions $\{t, \{f_\sigma\}_{\gamma\subset\sigma}\}$ defines a coordinate system for the polytope on $N_{\gamma}$.
 The edge skeleton's  tubular  neighborhood
\begin{equation}\label{N-Gamma}
N_{\Gamma^d}:=(\cup_{v\in V}N_v)\cup(\cup_{\gamma\in E}N_\gamma)
\end{equation} 
 will be the domain of our rescaling maps $\Psi^X_\epsilon$, see Figure~\ref{tubular}.  

\begin{remark}\label{notation-coordinate1}
We can turn the previous local coordinate systems over
the neighborhoods $N_v$ and $N_\gamma$ into a global system of  coordinates over $N_{\Gamma^d}$ with values in $\Rr^{F}$ as follows:

Each point $q\in N_v$ has coordinates $x=(x_\sigma)_\sigma \in \Rr^{F}$, where $x_\sigma=f_\sigma(q)$ if $v\in \sigma$ and 
$x_\sigma=0$ otherwise. 

Similarly, a point $q\in N_\gamma$
has coordinates $x=(x_\sigma)_\sigma\in \Rr^{F}$, where $x_\sigma=f_\sigma(q)$ if $\gamma\subset \sigma$ and 
$x_\sigma=0$ otherwise. Note that we have dropped the coordinate
$t=t(q)$ and hence this `coordinate system' fails to be injective.
The missing coordinate will not really matter because, as explained in the introduction,  
 global Poincar\'e maps become identity maps asymptotically.
\end{remark}

We  use the following family of functions to define the rescaling coordinates. For every $n=1,2,\ldots$,   let   $h_n:(0,1]\to\Rr$  be the function
\begin{equation}\label{defn:hn}
h_1(x)=-\log \,x \quad\mbox{and}\quad h_n(x)=-\frac{1}{n-1}\left(1-\frac{1}{x^{n-1}}\right)\quad n\geq2.
\end{equation}

\begin{remark} This family is characterized by the  properties:\\
 $h_n'(x)= - x^{-n}$,  $h_n(0)=+\infty$ and   $h_n(1)=0$,
which imply that the function  $h_n:(0,1]\to [0,+\infty)$ is a diffeomorphism. A straightforward computation  yields
\begin{equation*}
(h_1)^{-1}(y)=e^{-y}\; \mbox{and}\; (h_n)^{-1}(y)=(1+(n-1)y)^{-\frac{1}{n-1}}\; \text{ if }\;  n\geq2. 
\end{equation*}
\end{remark}

\begin{defn}\label{defn:vcoor.ch.}
Given  $X\in\fX(\Gamma^d)$  we define  the 
$\epsilon$-rescaling  coordinate system $\Psi_\epsilon^X:N_{\Gamma^d}\setminus \partial\Gamma^d\to\Rr^F$ which maps $q\in N_{\Gamma^d}$ to $y:=(y_{\sigma})_{\sigma\in F}$ where
\begin{itemize}
\item[$\bullet$] if $q\in N_v$ for some vertex $v$:
$$y_{\sigma}=\left\{ \begin{array}{cll}
\epsilon^2 h_{\nu(X,\sigma) }(f_\sigma(q)) & \text{ if } &\sigma\in F_v \\
0 &   \text{ if } &\sigma\notin F_v
\end{array}
\right. $$
\item[$\bullet$] if $q\in N_\gamma$ for some edge $\gamma$:
$$y_{\sigma}=\left\{ \begin{array}{cll}
\epsilon^2 h_{\nu(X,\sigma) }(f_\sigma(q)) & \text{ if } &\gamma\subset \sigma \\
0 &   \text{ if } &\gamma\not\subset \sigma
\end{array}
\right.$$
\end{itemize}
\end{defn}

 For a given vertex $v\in V$ we define
\begin{equation}\label{pi-v}
\Pi_v:=\{\, (y_\sigma)_{\sigma\in F}\in \Rr_+^{F}\,\colon\, y_\sigma=0 \quad \forall \sigma\notin F_v\,\}\;.
\end{equation}
Since $\{f_\sigma : \sigma \in F_v\}$ is a coordinate system for $\Gamma^d$ in $N_v$ and the functions $h_n:(0,1]\to [0,+\infty)$ are diffeomorphisms, the restriction of $\Psi_\epsilon^X$ to $N_v\setminus \partial\Gamma^d$ is a diffeomorphism onto $\Pi_v$.

Next consider an edge $\gamma$ connecting two corners $(v,\sigma)$ and $(v', \sigma')$. Note that  $F_{v}\cap F_{v'}=\{\sigma \in F : \gamma\subset\sigma\}$, which means that  the image
\[\Psi_\epsilon^X(N_\gamma\setminus \partial\Gamma^d)=\{\, (y_\sigma)_{\sigma\in F}\in \Rr_+^{F}\,\colon\, y_\sigma=0 \quad \mbox{when}\, \gamma\not\subset \sigma\,\}\;.\]
is equal to $\Pi_{v}\cap \Pi_{v'}$. 
We  denote this image by $\Pi_\gamma$. Notice that $N_\gamma$ has dimension $d$, while $\Pi_\gamma$ has dimension $d-1$.
In particular the map $\Psi^X_\epsilon$ is not injective over $N_\gamma$.

Let us  explain the use of the term `coordinate system' here.
As mentioned above, the family  of functions $\{t, \{f_\sigma\}_{\gamma\subset\sigma}\}$ defines a coordinate system for $\Gamma^d$ on $\overline{N_\gamma}$. For any $t_0\in (0,1)$,
let $\Sigma_{t_0}:=\{q\in N_\gamma: t(q)=t_0\}$. This set is a transversal cross-section to $\gamma$ at the point $q=\gamma \cap t^{-1}(t_0)$.
Between the boundary transversal cross-sections $\Sigma_0, \Sigma_1$, we have 
\[\Psi_\epsilon^X(\Sigma_t)=\Psi_\epsilon^X(N_\gamma)\quad \forall t\in [0,1].\]
\begin{center}
\begin{figure}[h]
\includegraphics[width=9cm]{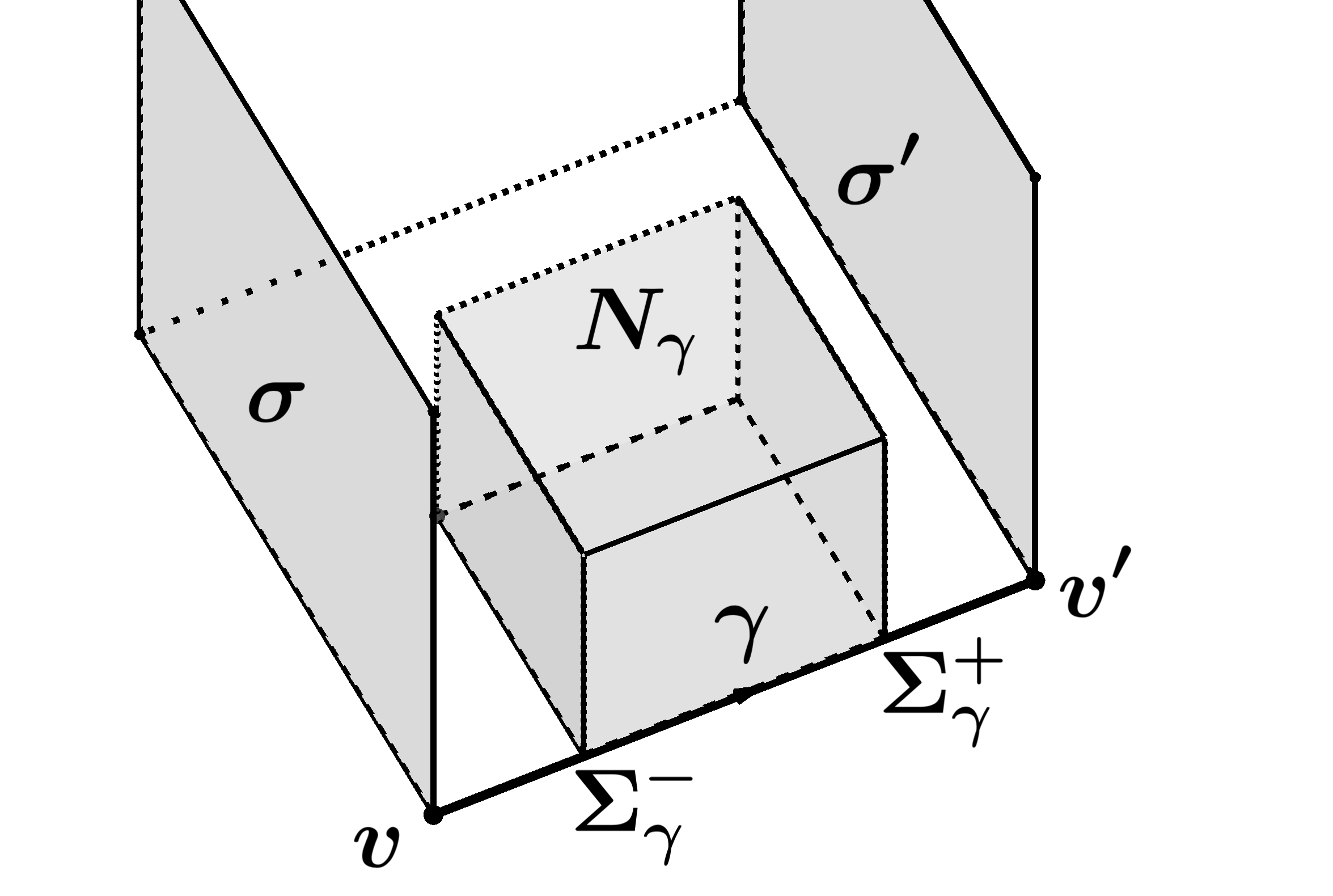}
\caption{An edge connecting two corners} \label{edge-corners}
\end{figure}
\end{center}
As  mentioned in the introduction,   asymptotically the global Poincar\'e maps are  identity maps, see Lemma~\eqref{pgamma}. Thus
the asymptotic flow identifies all cross-sections $\Sigma_t$, $t\in [0,1]$. This makes the map $\Psi_\epsilon^X$ a suitable `coordinate system'  for our purposes.  

\begin{defn}\label{defn:dualcones}
The {\em dual cone} of   $\Gamma^d$ is defined to be
$$\CC^\ast(\Gamma^d) :=\bigcup_{v\in V}\Pi_v\;, $$
where $\Pi_v$ is the sector defined at~\eqref{pi-v}. Points of the dual cone  will always be denoted by  $y=(y_\sigma)_{\sigma\in F}$.
\end{defn}

By construction, the  dual cone  is the range of the $\epsilon$-rescaling coordinate system, \ie $\Psi^X_\epsilon( N_{\Gamma^d}\setminus \partial\Gamma^d) = \CC^\ast(\Gamma^d)$. In particular these coordinates determine  a family of maps
$\Psi^X_\epsilon:N_{\Gamma^d}\setminus \partial\Gamma^d\to \CC^\ast(\Gamma^d)$. We will write  $\Psi^X_{v,\epsilon}$ instead of
$\Psi^X_{\epsilon}$  to emphasize that we are
dealing with the restriction of the $\epsilon$-rescaling coordinates
to the neighborhood $N_v$, which is a diffeomorphism
$\Psi^X_{v,\epsilon}:N_{v}\setminus \partial\Gamma^d\to \Pi_v$.

To explain the term `dual' notice first that $\Pi_\gamma=\Pi_{v}\cap\Pi_{v'}$, whenever $\gamma$ is an edge 
connecting the vertexes $v$ and $v'$. Similar relations hold  for higher dimensional faces. In fact for any face $\rho\subset \Gamma^d$, 
we can define 
$$\Pi_\rho:=\{\, (y_\sigma)_{\sigma\in F}\in \Rr_+^{F}\,\colon\, y_\sigma=0 \quad \mbox{when}\, \rho\not\subset \sigma\,\} .$$
The dual cone $\CC^\ast(\Gamma^d)$ has a simplicial structure
where $\Pi_\rho$ is a face of $\CC^\ast(\Gamma^d)$  for every face
$\rho$ of $\Gamma^d$. Moreover, for any faces $\rho, \rho'$ of $\Gamma^d$,
$$ \rho\subset \rho' \quad  \Leftrightarrow \quad  \Pi_{\rho'}\subset \Pi_\rho .$$
The dual cone of a polytope can be identified with the polytope's normal fan, which in turn coincides with the face fan of its dual polytope,
see~\cite[Chapter 7]{Ziegle1995}. This gives a short explanation for the inherent duality between a polytope and its dual cone.


The following technical lemma will be used to control the asymptotic behavior of $\Psi_\epsilon^X$. 
 
%
\begin{lemma}\label{hninv}
For any  $n\geq 1$  and   $k\geq 1$, there exists $0<r(k,n)\leq 1$ such that the diffeomorphisms $h_n:(0,1]\to [0,+\infty)$ satisfy
\begin{enumerate}
	\item $\displaystyle \; \lim_{\epsilon\to 0^+} \max_{0\leq i\leq k} 
	\sup_{y\geq\epsilon^r} 
	\abs{ \frac{d^i}{d y^i} h^{-1}_n\left(\frac{y}{\epsilon^2}\right)  } =0$,
	\item $\displaystyle \;  \lim_{\epsilon\rightarrow 0^+}  \max_{0\leq i\leq k}
	\sup_{y\geq \epsilon^r} \abs{ \frac{d^i}{dy^i} 
		\left[\, \epsilon^2 \left( h_l\circ (h_n)^{-1}\right)\left(
		\frac{y}{\epsilon^2} 
		\right) 
		\right] } 
	= 0  $ \,  for 	$1\leq l<n$.
\end{enumerate}
Moreover $r(k,1)=1$ for all $k\geq 1$.
\end{lemma}

\begin{proof}
For $n=1$ take $r=1$ regardless of $k$. The $k^{\text{th}}$ derivative of
$e^{-y/\epsilon^2}$   is bounded, over $y\geq \epsilon$,  by $\epsilon^{-2k}\,e^{-1/\epsilon}$, which tends to $0$ as
$\epsilon\to 0^+$. In this case the conclusion (2) is empty.

For $n>1$ and $y\geq \epsilon^r$  the  $k^{\text{th}}$ derivative of
$h_n^{-1}(y/\epsilon^2)$ is bounded by 
\begin{align*}
\frac{(n-1)^k}{\epsilon^{2k}}\, \prod_{j=0}^{k-1}\left(-\frac{1}{n-1}-j\right)\,\left(1+ \frac{\epsilon^r}{\epsilon^2} \right)^{-\frac{1}{n-1}-k }   & \asymp \epsilon^{(2-r)(\frac{1}{n-1}+k)-2 k } \\
 & = \epsilon^{\frac{2}{n-1}-r (\frac{1}{n-1}+k)  } 
\end{align*}
which tends to $0$ as
$\epsilon\to 0^+$, provided we choose
$$ 0< r <\frac{\frac{2}{n-1}}{\frac{1}{n-1} + k} 
= \frac{2}{1+(n-1)k} \leq 1 . $$
 The last inequality holds  for any  $n\geq2$.
 This proves item (1).
 
 \newcommand{\bigO}{\mathcal{O}}
 Consider now the family of functions $g_l(\epsilon,y):= \epsilon^2 ( h_l\circ h_n^{-1})(y/\epsilon^2)$ with $1\leq l<n$. For $l=1$ we have
 $$ g_1(\epsilon,y)= \frac{\epsilon^2}{n-1}\log \left( 1+(n-1)\frac{y}{\epsilon^2}\right)$$
 and over the interval $y\geq \epsilon^r$, $g_1(\epsilon,y)=\bigO(\epsilon)$ as $\epsilon\to 0$.
 The higher order derivatives of $g_1(\epsilon,y)$ are 
  $$ \frac{d^k g_1}{dy^k} (\epsilon,y)= \pm (k-1)!\left( \frac{n-1}{\epsilon} \right)^{k-1} \left( 1+(n-1)\frac{y}{\epsilon^2}\right)^{-k} .$$
  Hence  over the interval $y\geq \epsilon^r$
   $$\frac{d^k g_1}{dy^k} (\epsilon,y) =\bigO(\epsilon^{2-r k})\;\; \text{ as } \;  \epsilon\to 0 $$
   and this tends to $0$ provided $r<\frac{2}{k}$.
   
   For $2\leq l<n$ set $\theta_l=\frac{l-1}{n-1}$ and notice that
   $\theta_l<\frac{n-2}{n-1}<1$.
   A simple calculation gives
   $$ g_l(\epsilon,y)= -\frac{\epsilon^2}{l-1} +
   \frac{\epsilon^2}{l-1}  \left( 1+(n-1)\frac{y}{\epsilon^2}\right)^{\theta_l}  $$
   and over the interval $y\geq \epsilon^r$ one has $g_l(\epsilon,y)=\bigO(\epsilon^{\frac{2}{n-1} })$ as $\epsilon\to 0$. 
   For $k\geq 1$, the higher order derivatives of $g_l(\epsilon,y)$ are 
   $$ \frac{d^k g_l}{dy^k} (\epsilon,y)= \pm \frac{n-1}{l-1} \left[\prod_{j=0}^{k-1} (\theta_l-j) \right] \left( \frac{n-1}{\epsilon^2} \right)^{k-1} \left( 1+(n-1)\frac{y}{\epsilon^2}\right)^{\theta_l-k} .$$
   Hence  over the interval $y\geq \epsilon^r$
   $$\frac{d^k g_l}{dy^k} (\epsilon,y) =\bigO(\epsilon^{-r(k-\frac{n-2}{n-1})+\frac{2}{n-1}})\;\; \text{ as } \;  \epsilon\to 0 $$   
   which tends to $0$ provided $r<\frac{2}{k(n-1)-(n-2)}$.
   This proves item (2).
\end{proof}

To shorten statements about convergence in the forthcoming lemmas and theorems  we introduce some terminology.

\begin{defn}\label{remark:convergence}
Suppose we are given a family of functions $F_\epsilon$ with varying domains $\UU_\epsilon$.
Let $F$ be another function with domain $\UU$. Assume that  all these functions have the same  
target and source spaces, which are assumed to be linear spaces. 
We will say that\, 
$ \lim_{\epsilon \rightarrow 0^+} F_\epsilon = F$\,
{\em in the  $C^k$ topology},
to mean that:
\begin{enumerate}
\item {\em domain convergence}: for  every compact subset $\, K\subseteq\UU$, 
we have $K \subseteq\UU_\epsilon$ for every small enough $\epsilon>0$, and 
\item {\em uniform convergence on compact sets}:
$$ \lim_{\epsilon \rightarrow 0^+}\;
\max_{0\leq i\leq k} \sup_{ u\in K }
\left|\, D^i \left[ 
F_\epsilon(u) - F(u)
\right]\,
\right|\; =\; 0\;.$$
\end{enumerate}
Convergence in the $C^\infty$ topology means
convergence in the $C^k$ topology for all $k\geq 1$.
If in a statement $F_\epsilon$ is a composition of two or more mappings then
its domain should be understood as the composition domain.
\end{defn}

Next lemma relates the asymptotic push-forward of $X$ by $\Psi_\epsilon^X$ near a vertex $v$ with the skeleton character $\chi^v$ of $X$ at $v$, see Definition~\ref{skeletoncharacter}. It says that the vector field
$(\Psi^X_\epsilon)_\ast X$ rescaled by the factor $\epsilon^{-2}$
converges to the constant vector field $\chi^v$ on the sector $\Pi_v$.
In particular the trajectories of the  push-forward vector field
$(\Psi^X_\epsilon)_\ast X$ are asymptotically linearized to the  lines of the flow of the constant vector field $\chi^v$.
We will denote by $\Psi_{v,\epsilon}^X$ the restriction of $\Psi_\epsilon^X$ to $N_v$. Define also
\begin{equation}\label{pivepsilon}
\Pi_v(\epsilon):=\{\,y\in\Pi_v\,\colon\, y_\sigma \geq\epsilon \quad\text{for all}\quad \sigma\in F_v\, \}
\end{equation}

\begin{lemma}\label{lemma:rescal}
Consider the functions $H_\sigma$  defined in ~\eqref{def:order}.
Then 
\[(\Psi_{v,\epsilon}^X)_\ast X=\epsilon^2 
\,\left(\tilde{X}_{v,\sigma}^\epsilon\right)_{\sigma\in F} ,\]
where $$\tilde{X}_{v,\sigma}^\epsilon(y):=
\left\{ \begin{array}{cll}
-H_\sigma\left((\Psi_{v,\epsilon}^X)^{-1}(y) \right) &\text{if} & \sigma\in F_v\\
0  &\text{if} & \sigma\notin F_v\\
\end{array}  
\right. .$$
Moreover, given $k\geq 1$ there exists $r=r(k,X)>0$ such that
 the following limit holds in the $C^k$ topology
\[\lim_{\epsilon\to 0}\, ( \tilde{X}_v^\epsilon)_{|_{\Pi_v(\epsilon^r)}}= \chi^v\;.\]
\end{lemma}

\begin{proof}
	Let $F_v=\{\sigma_1,\ldots, \sigma_d\}$
	and  $(x_1,	\ldots, x_d)=(f_{\sigma1}(q),\ldots, f_{\sigma_d}(q))$
	be the coordinate system  introduced in Remark~\ref{notation-coordinate}.
	 Denote by $\nu_l$ the order of the facet $\sigma_l$. Let $H_l(x)$ be the function
	 $H_{\sigma_l}(q)$ expressed in this coordinate system. Then by~\eqref{def:order}, the equation $\frac{dq}{dt}=X(q)$ is equivalent to the system of differential equations	 
\begin{equation*}
\frac{\d x_l}{\d t}=x_l^{\nu_l}H_l(x),\quad 1\leq l\leq d\; .
\end{equation*}
In these coordinates
$$ \Psi_{v,\epsilon}^X(x_1,\ldots, x_d)= \epsilon^2\, \left(h_{\nu_1}(x_1), \ldots, h_{\nu_d}(x_d),  0,\ldots, 0\right)  .$$
Therefore, since the Jacobian of $\Psi_{v,\epsilon}^X$ can be identified with the diagonal matrix
$$D(\Psi_{v,\epsilon}^X)_x=-\epsilon^2\mathrm{diag}(x_1^{-\nu_1},\ldots, x_d^{-\nu_d} ) $$
 the first claim follows. 

 Fix $k\in\Nn$ and take $r= \min_{1\leq j\leq d} r(k,\nu_j)$,,
 where $r(k,n)$ is tha function in Lemma~\ref{hninv}.
Given $y\in \Pi_v(\epsilon^r)$, 
$$H_{\sigma_l}\left((\Psi_{v,\epsilon}^X)^{-1}(y)\right)=H_l\left( h^{-1}(\frac{y}{\epsilon^2})\right),$$
where $h^{-1}(\frac{y}{\epsilon^2}):=\left(h^{-1}_{\nu_1}( \epsilon^{-2}\,{y_{\sigma_1}}),\ldots, h^{-1}_{\nu_d}( \epsilon^{-2}\, {y_{\sigma_d}}) \right)$. 
Thus, by item (1) of Lemma~\ref{hninv} combined with Definition~\ref{skeletoncharacter},
the convergence follows.
\end{proof}

\section{Skeleton Vector Fields}
\label{skeleton}

In this section we define the skeleton of a  vector field  
$X\in \fX(\Gamma^d)$ and  its  corresponding skeleton flow map,   explaining how it is computed and its dynamics is analyzed.

\begin{defn}
\label{def skeleton vector field}
Given  $X\in\fX(\Gamma^d)$, the  {\em skeleton} of $X$
is the piecewise constant vector field $\chi$ on dual cone $\CC^\ast(\Gamma^d)$ which is constant and equal to $\chi^v$
on each sector $\Pi_v$, where
$\chi^v= (\chi^v_\sigma)_{\sigma\in F}$ is the skeleton character at $v$ introduced in Definition~\ref{skeletoncharacter}.
Notice that for every vertex $v$, the vector $\chi^v$ is tangent to $\Pi_v$.
\end{defn}

 Our goal is to study the piecewise linear flow generated by  the skeleton vector field $\chi$.  Remark~\ref{sign-character}  justifies that we call  \emph{$\chi$-repelling} a vertex $v$ such that $\chi^v_\sigma<0,\,\,\forall\sigma\in F_v$, and \emph{$\chi$-attracting} if $\chi^v>0,\,\,\forall\sigma\in F_v$.
 A vertex $v$ is said to be of \emph{saddle type} if  for some pair of facets $\sigma_1,\sigma_2\in F$ one has $\chi^v_{\sigma_1}\,\chi^v_{\sigma_2}<0$. The edges of $\Gamma^d$ are also classified as follows.

\begin{defn}\label{edges}
 Let $\gamma\in E$ be an edge  with end corners $(v,\sigma)$ and $(v',\sigma')$. We say that $\gamma$ is a \emph{defined type edge}  if either $\chi^{v}_{\sigma}\chi^{v'}_{\sigma'}\neq 0$ or else
 $\chi^{v}_{\sigma}=\chi^{v'}_{\sigma'}=0$. A defined type edge $\gamma$ is called 
 \begin{itemize}[leftmargin=2em]
 \item[$(1)$] a \emph{flowing-edge} if  $\chi^{v}_{\sigma}\chi^{v'}_{\sigma'}<0$, 
 \item[$(2)$] a \emph{neutral edge}  if $\chi^{v}_{\sigma}=\chi^{v'}_{\sigma'}=0$,
 \item[$(3)$] an \emph{attracting edge}  if  $\chi^{v}_{\sigma}<0$ and $\chi^{v'}_{\sigma'}<0$,
 \item[$(4)$]  a \emph{repelling  edge}  if  $\chi^{v}_{\sigma}>0$ and $\chi^{v'}_{\sigma'}>0$, 
 \end{itemize}
 
For a flowing-edge $\gamma$ with opposite corners $(v,\sigma)$ and $(v',\sigma')$, we  write $(v,\sigma)\buildrel\gamma\over\longrightarrow (v',\sigma')$, whenever $\chi^{v}_{\sigma}<0$ and $\chi^{v'}_{\sigma'}>0$.
 The vertexes $v$ and $v'$ are  respectively called the \emph{source} of $\gamma$, denoted by $s(\gamma)$, and  the \emph{target} of $\gamma$, denoted by $t(\gamma)$. 
 
\end{defn}

We call {\em orbit} of $\chi$ to any continuous piecewise affine function $c:I\to \CC^\ast(\Gamma^d)$, defined on some interval $I\subset\Rr$, such that 
	\begin{enumerate}
		\item  $c'(t)=\chi^v$ whenever $c(t)$ is interior to some  $\Pi_v$, 
		with $v\in V$, 
		\item   there is at most a countable set of times $t\in I$ such that $c(t)$ is not interior to any sector $\Pi_v$, with $v\in V$.
	\end{enumerate}

Writing $I=[t_0,t_n]$, a sequence of vertexes $(v_1,v_2,\ldots, v_m)$ such that for some times $t_0<t_1<\ldots < t_{n-1}<t_n$ one has
$c(t)\in \inter(\Pi_{v_j})$ for all
$t_{j-1}<t<t_j$, is called the \emph{itinerary} of the orbit segment $c$.
This implies that   $c(t_j)\in \Pi_{v_{j-1}}\cap \Pi_{v_{j}}=\Pi_{\gamma}$, where $v_{j-1}\buildrel\gamma\over\longrightarrow v_j$ is a flowing edge, 
for every  $j=1,\ldots, n-1$. If there are flowing edges $\gamma_0$ and $\gamma_n$ such that the endpoints  satisfy
$c(t_0)\in \Pi_{\gamma_0}$ and $c(t_n)\in \Pi_{\gamma_n}$ then the sequence of edges $(\gamma_0,\gamma_1,\ldots,\gamma_n)$ is also referred to as
the \emph{itinerary} of the orbit segment $c$.

\begin{defn}\label{regular-vectorfield}
We say that a vector field $X\in\fX(\Gamma^d)$ is \emph{regular} when all its edges  have defined type and 
\begin{enumerate}
\item $X$ has no singularities in $\inter(\gamma)$ 
for every  flowing edge $\gamma$,
\item $X$ vanishes along every neutral edge $\gamma$.
\end{enumerate}
\end{defn}

From now on, we will only consider  regular vector fields. Figure~\ref{edges_types} depicts the relation between
the orientation of the flow of $X$ along $\gamma$ and
the orientation of the flow of $\chi$ around $\Pi_\gamma$. 
 
\begin{figure}[h]
\begin{center}
\includegraphics[width=11cm]{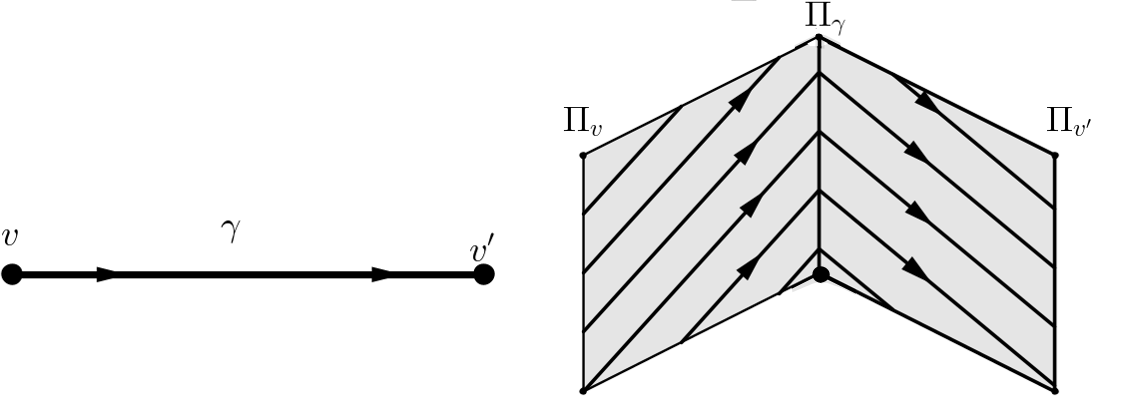}
\end{center} 
\caption{A flowing edge}
\label{edges_types}
\end{figure}

Given vertex $v$ of saddle type together with an incoming flowing-edge $v_\ast\buildrel\gamma \over\longrightarrow v$ and an outgoing flowing-edge  $v \buildrel\gamma' \over\longrightarrow v'$, denoting by $\sigma_\ast$   the facet opposed to $\gamma'$ at $v$  we define the sector  $\Pi_{\gamma,\gamma'}=\Pi_{\gamma,\gamma'}^\chi$
\begin{equation}\label{hyperplane}
\Pi_{\gamma,\gamma'} :=\left\{\, y\in {\rm int}(\Pi_{\gamma})\,\colon\,  y_\sigma-\frac{\chi^{v}_\sigma}{\chi^{v}_{\sigma_\ast}}\, y_{\sigma_\ast} > 0, \, \forall  \sigma \in F_{v},\; \sigma\ne \sigma_\ast  \; \right\} 
\end{equation}
and the linear map $L_{\gamma,\gamma'}=L^\chi_{\gamma,\gamma'}:\Pi_{\gamma,\gamma'}\to\Pi_{\gamma'}$ 
\begin{equation}\label{L:gamma:gammaprime}
L_{\gamma,\gamma'}(y):=  \left(\, y_\sigma-\frac{\chi^{v}_\sigma}{\chi^{v}_{\sigma_\ast}} \, y_{\sigma_\ast}\, \right)_{\sigma\in F}\; .
\end{equation}
 Notice that $\Pi_{\gamma'}=\{y\in \Pi_{v}: y_{\sigma_\ast}=0\} $

\begin{proposition}\label{proposition:skeletonflow}
Given a  vertex $v$ of saddle type together with incoming and outgoing (flowing) edges $\gamma,\gamma'$ as above, the sector $\Pi_{\gamma,\gamma'}$ is the set of points $y\in {\rm int} (\Pi_{\gamma})$ which are connected by
the orbit segment $\{\,c(t)=y+t\chi^v\,
\colon t\geq 0, \, c(t)\in \Pi_{v}\}$ to 
$L_{\gamma,\gamma'}(y)\in {\rm int}(\Pi_{\gamma'})$.
\end{proposition}
\begin{proof}
Straightforward. 
\end{proof}

The map $L_{\gamma,\gamma'}$ is a Poincar\'e for the flow  of $\chi$, which is represented by the following $F\times F$ matrix
\begin{equation} \label{matrix-M}
 M_{\gamma,\gamma'}=\left( \delta_{\sigma\sigma'} -  \frac{\chi^{v}_{\sigma}}{\chi^{v}_{\sigma_\ast}}\delta_{\sigma_\ast\sigma'}\right)_{\sigma,\sigma'\in F}\;, 
 \end{equation}
where $\delta$ stands for the Kronecker delta symbol.
This matrix gives  a global representation of the flow of $\chi$ which is  suitable for computational purposes. The image of the map $L_{\gamma,\gamma'}$ is the convex cone $\Pi^{-\chi}_{\gamma',\gamma}$
associated with the vector field $-\chi$ and the pair $\gamma',\gamma$ of reversed flowing edges.
Clearly $L^{-\chi}_{\gamma',\gamma}=(L^{\chi}_{\gamma,\gamma'})^{-1}$.

\begin{remark}\label{remark:skeletonflow}
 If $v$ is a saddle type vertex  then
any line parallel to $\chi^v$ through a point in  $\inter(\Pi_v)$
must intersect at least two boundary facets of $\Pi_v$.

Conversely, if an orbit segment $c(t)=p+t\,\chi^v$ through a point
$p\in\inter(\Pi_v)$ crosses  the boundary of $\Pi_v$ at two points,
$q=p+ t_0\,\chi^v$, with $t_0<0$, and $q'=p+t_1\,\chi^v$, with $t_1>0$, and if  $\sigma',\sigma_\ast\in F_v$ are  the facets of $\Pi_v$ such that  
$q_{\sigma'}=0$ and $q_{\sigma_\ast}'=0$
then  $\chi^v_{\sigma'}>0$ and $\chi^v_{\sigma_\ast}<0$. This implies that $v$ is of saddle type.

In this setting, if  $\gamma, \gamma'$ are the edges through $v$, respectively  in the corners $(v,\sigma')$ and $(v,\sigma_\ast)$, then
$q\in \Pi_\gamma=\{y\in\Pi_v\colon y_{\sigma'}=0\}$ and $q'\in\Pi_{\gamma'}=\{y\in\Pi_v\colon y_{\sigma_\ast}=0\}$. Moreover,
if both $\gamma$ and $\gamma'$ are flowing edges then
$q\in \Pi_{\gamma,\gamma'}$ and $q'=L_{\gamma,\gamma'}(q)$.

If the vertex $v$ is attracting or repelling (instead of saddle type),  \ie if all the characters  $\chi^{v}_\sigma$, with $\sigma\in F_{v}$, have the same sign, then 
$\Pi_{\gamma,\gamma'}=\emptyset$.
 In these cases, it is not possible to connect any point in $\Pi_{\gamma}$ to a point of $\Pi_{\gamma'}$ through a line parallel to the  constant vector $\chi^v$, see Figure~\ref{vertex_types}.
 \end{remark} 
 \begin{remark}
 	Points in the boundary of $\Pi_{\gamma}$ are in the intersection of three or more sectors $\Pi_{v}$ with  $v\in V$. Hence, if an orbit ends up in one of these points it is not  possible to continue it in a unique way. 
In the  sequel we disregard these types of orbits.
\end{remark}
\begin{figure}[h]
\begin{center}
\includegraphics[width=10cm]{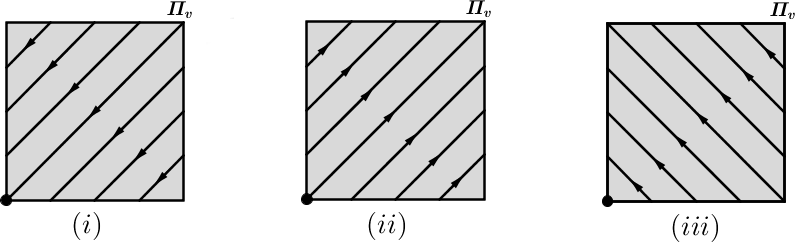}
\end{center} 
\caption{Vertex types: $(i)$  attracting, $(ii)$ repelling and $(iii)$ saddle type}
\label{vertex_types}
\end{figure} 
We now define \textit{skeleton flow maps} along  chains of saddle type vertexes.  Let 
{\small
\begin{equation}\label{path1}
v_0\buildrel\gamma_0 \over\longrightarrow  v_1 \buildrel\gamma_1 \over\longrightarrow  v_2 \longrightarrow  \ldots \longrightarrow  v_m \buildrel\gamma_m \over\longrightarrow v_{m+1}
\end{equation}}
be a chain of flowing-edges.
 The sequence $\xi =(\gamma_0,\gamma_1,\ldots, \gamma_m)$ will be called a \emph{heteroclinic path}, a \emph{heteroclinic cycle} when $\gamma_{m}=\gamma_0$.

\begin{defn}\label{poincaremap}
 Given a  heteroclinic path  $\xi =(\gamma_0,\gamma_1,\ldots, \gamma_m)$, we define the \emph{skeleton flow map (of $\chi$) along  $\xi$} to be the composition mapping $\pi_\xi:\Pi_\xi\to \Pi_{\gamma_m}$ 
 $$\pi_\xi:=L_{\gamma_{m-1},\gamma_m}\circ\ldots\circ L_{\gamma_0,\gamma_1}\; ,$$ 
 with domain   
\begin{align*}
\Pi_\xi & := {\rm int}(\Pi_{\gamma_0}) \cap 
\bigcap_{j=1}^{m} (L_{\gamma_\ast,\gamma_j}\circ\ldots\circ L_{\gamma_0,\gamma_1})^{-1}
 {\rm int}(\Pi_{\gamma_j})   \;. 
\end{align*}
\end{defn}

%
%
%
%

For every $y\in \Pi_\xi$,  $y\in \inter(\Pi_{\gamma_0})$,
$\pi_\xi(y)\in \inter(\Pi_{\gamma_m})$ and moreover  there exists an orbit segment from $y$ to $\pi_\xi(y)$ with itinerary $\xi$.

We also define the matrix 
\begin{equation}
\label{poincare_path_matrix}
M_\xi:= M_{\gamma_{m-1},\gamma_{m}}\cdots  M_{\gamma_{1},\gamma_{2}}\,  M_{\gamma_{0},\gamma_{1}}  
\end{equation}
where the factor matrices $M_{\gamma_{j-1},\gamma_{j}}$ were defined in~\eqref{matrix-M}. This matrix $M_\xi$ induces a linear endomorphism on $\Rr^F$
whose restriction to the sector $\Pi_\xi$ matches the skeleton flow map $\pi_\xi$.

 In order to analyze the dynamics of the flow of the skeleton vector field $\chi$ it is convenient to introduce the concept of {\em structural set} and  its associated skeleton flow  map. 
 
 \begin{defn}\label{structural:set}
 A non-empty set of flowing edges $S$ 
 is said to be a {\em structural set} for  $\chi$ if
 every heteroclinic cycle  contains an edge in $S$.
 \end{defn}

Notice that the structural set $S$ is in general not unique.
The concept of structural set can be  defined for general directed graphs.
It corresponds to the homonym  notion
introduced by L. Bunimovich and B. Webb~ \cite{BW2012},  but here applied to the line graph\footnote{ The line graph of a directed graph $G$, denoted by $L(G)$, is the graph whose vertices are the edges of $G$, and where $(\gamma,\gamma')\in E\times E$ is an edge of $L(G)$ if the end-point of $\gamma$ coincides with the start-point of $\gamma'$. }.

We say that a heteroclinic path $\xi=(\gamma_0,\ldots,\gamma_m)$ is a \emph{branch} of $S$, or shortly an $S$-branch,   if
\begin{enumerate}
\item $\gamma_0,\gamma_m\in S$,
\item $\gamma_j\notin S$ for all $j=1,\ldots, m-1$.
\end{enumerate}
We denote by $\Bscr_S(\chi)$ the set of all $S$-branches.

\begin{defn}
\label{def skeleton flow map}
The \textit{skeleton flow  map}  $\pi_S:D_S\to\Pi_S$ is defined by  
$$\pi_S(y):=\pi_\xi(y)\quad \text{ for all } \; y\in\Pi_\xi,$$
where 
$$\Pi_S:=\cup_{\gamma\in S}\Pi_\gamma\quad \text{ and } \quad 
D_S:=\cup_{\xi\in\Bscr_S(\chi)}\Pi_\xi.$$
\end{defn}


We now provide a sufficient condition
for  the skeleton flow map $\pi_S$ to be a closed dynamical system.

\begin{proposition}  \label{partition}
Given a  skeleton vector field $\chi$ on $\CC^\ast(\Gamma^d)$  and a structural set $S$, assume  
\begin{enumerate}
\item every edge of the polytope is either neutral or a flowing edge.
\item all vertexes are of saddle type,
\end{enumerate}
Then  $D_S$ has full Lebesgue measure in 
$ \Pi_S$.
\end{proposition}

\begin{proof} 
	The inclusion $D_S\subseteq\Pi_S$ is obvious.

	For each flowing edge $\gamma$ with $v=t(\gamma)$,
	let $D_\gamma:=\cup_{s(\gamma')=v} \Pi_{\gamma,\gamma'}$, with the
	union taken over the set of  flowing edges $\gamma'$ such that $s(\gamma')=v$. Clearly
	$D_\gamma\subset \Pi_\gamma$.
	We claim that
	$$  \Pi_\gamma\setminus D_\gamma \subseteq  \partial\Pi_\gamma \cup \bigcup_{\gamma'\colon s(\gamma')=t(\gamma)} L_{\gamma,\gamma'}^{-1} ( \partial \Pi_{\gamma'} )  $$
	which in particular implies that this set has codimension one in the sector $\Pi_\gamma$. Let us now prove the claim.
	By (2) the vertex $v$ is of saddle type. 
	Since $v=t(\gamma)$, the corner $(v,\gamma)$ has positive character.	Hence, given $q\in \inter(\Pi_\gamma)$
	the orbit segment $c(t):=q+t\chi^v$ enters $\inter(\Pi_v)$ for  $t$ positive and small. By Remark~\ref{remark:skeletonflow}, this orbit segment
	will eventually hit another boundary point $q'\in \Pi_{\gamma'}\subset \partial\Pi_v$ for some edge $\gamma'$ 	through $v$. The same remark shows that $\chi$ has opposite signs at the corners $(v,\gamma)$ and $(v,\gamma')$. Thus,  by item (1)  $\gamma'$ is also a  flowing edge and $q'=L_{\gamma,\gamma'}(q)$. Therefore,
	if $q\notin D_\gamma$ then
	$$q\in \bigcup_{\gamma'\colon s(\gamma')=v} L_{\gamma,\gamma'}^{-1} ( \partial \Pi_{\gamma'} )$$
	which proves the claim.
	
	Let  $D=\cup_\gamma D_\gamma$ and $\Pi=\cup_\gamma \Pi_\gamma$  with  the unions taken over all flowing edges. Define then a skeleton flow  map
	$\pi\colon D\to \Pi$ setting $\pi(y)=L_{\gamma,\gamma'}(y)$ whenever
	$\gamma,\gamma'$ are flowing edges such that $t(\gamma)=s(\gamma')$ and
	$y\in \Pi_{\gamma,\gamma'}$. The previous claim implies that $D$ has ful Lebesgue measure in $\Pi$. In fact it shows that  $\Pi\setminus D$ has codimension one in $\Pi$. The set 
	$D_\infty = \cap_{n\geq 0} \pi^{-n}(D)$ has also full measure because $\pi:D\to\Pi$ is locally a linear isomorphism and 
	$\Pi\setminus D_\infty=\cup_{n\geq 0} \pi^{-n}(\Pi\setminus D)$ is a countable union of  sets with zero Lebesgue measure.

	Consider now $y\in \Pi_S$, assume that $y\in D_\infty$  and consider the itinerary $(\gamma_0,\gamma_1,\ldots\,)$  of
	the corresponding (forward) infinite orbit.
	Then $\gamma_0\in S$. 
	Assumptions (1)-(2) imply that the flowing edge graph has no terminal points.
	If we had $\gamma_j\notin S$ for all $j\geq 1$,
	there would be heteroclinic cycles disjoint from $S$, which  contradicts the fact that $S$ is a structural set. Hence some initial segment $\xi=(\gamma_0,\ldots, \gamma_m)$ of this itinerary is an $S$-branch, and $y\in \Pi_\xi\subseteq D_S$. 
	This proves that $\Pi_S\setminus D_S\subset \Pi\setminus D_\infty$ has zero Lebesgue measure.
\end{proof}

\begin{remark}
The proof of Proposition~\ref{partition} shows that the maximal invariant set
$$ \hat D_S := \bigcap_{n\in\Zz} (\pi_S)^{-n}(D_S)$$
has full Lebesgue measure in $\Pi_S$. Hence the skeleton flow map  induces a homeomorphism $\pi_S:\hat D_S\to \hat D_S$ on the Baire space  $\hat D_S$.
\end{remark}


\section{Asymptotic Poincar\'e Maps }
\label{asymptotics}

In this section, we state and prove our main result. Given a structural set $S$ consider the system of cross sections
$\Sigma_S=\cup_{\gamma\in S} \Sigma_\gamma^-$ transversal to the flowing edges in $S$. Then  the  Poincar\'e  map induced by the flow of a regular vector field $X\in \fX(\Gamma^d)$  on $\Sigma_S$
 is ``asymptotically conjugate'' to the skeleton flow map $\pi_S$ of   $\chi$.

For any flowing edge $\gamma$ through a vertex $v$ define
$$\Sigma_{v,\gamma}:=(\Psi_{v,\epsilon}^X)^{-1}(\inter(\Pi_{\gamma})) . $$
This cross section is transversal to the flow of $X$ and is an inner facet of the tubular neighborhood $N_v$. We will write
$\Sigma^-_\gamma$ or $\Sigma^+_\gamma$, instead of $\Sigma_{v,\gamma}$, according to the sign of the character
$\chi$ at the corner $(v,\gamma)$.


Let $\Dscr_{\gamma}$ be the set of points $x\in \Sigma^-_\gamma$ such that the forward  orbit of $x$ has a transversal intersection with $\Sigma^+_\gamma$.

\begin{defn}
\label{Pgamma}
The global Poincar\'e map along $\gamma$, see Figure~\ref{local-global-poincare},
$$P_{\gamma}: \Dscr_{\gamma}\subset
\Sigma_\gamma^-  \to \Sigma_\gamma^+ $$
is defined by
$P_{\gamma}(x):= \varphi_X^{\tau(x)}(x)$,
where $\varphi_X^t$ stands for the flow of $X$ and
$$ \tau(x)= \min\{\, t>0\,\colon\, 
\varphi_X^t(x) \in \Sigma_{v',\gamma}\, \} \;.$$
Both functions $\tau$ and $P_\gamma$ are analytic.
\end{defn}
 Let
\begin{equation}\label{pivepsilon1}
\Pi_{\gamma} (\epsilon):=\{\,y\in\Pi_{\gamma}\,\colon\, y_{\sigma} \geq\epsilon \quad\text{whenever}\quad \gamma \subset \sigma\} .
\end{equation}
Notice that $\lim_{\epsilon\to 0} \Pi_{\gamma} (\epsilon)=\inter(\Pi_{\gamma})$. Given $k\in\Nn$   take $r=r(k,X)$ according to  Lemma~\ref{lemma:rescal}. 
 
\begin{lemma}\label{pgamma}
Given a  flowing-edge  $$(v,\sigma_0)\buildrel\gamma \over\longrightarrow (v',\sigma'),$$
 let \,\, $\UU_{\gamma}^\epsilon\subset \Pi_{\gamma}(\epsilon^r)$ be the domain of the map 
 $$F_{\gamma}^\epsilon:=\Psi_{v',\epsilon}^X\circ P_{\gamma} \circ(\Psi_{v,\epsilon}^X)^{-1}.$$ 
 Then
\begin{align*}
\lim_{\epsilon\to0^+}F^\epsilon_{\gamma|_{\UU_{\gamma}^\epsilon}}=\id_{\Pi_{\gamma}} 
\end{align*} 
in the $C^k$ topology, in the sense of Definition~\ref{remark:convergence}.
\end{lemma}
\begin{proof}
If 
$F_{v}=\{\sigma_0,\sigma_1,\sigma_2,\ldots,\sigma_{d-1}\}$ then $ F_{v'} =\{\sigma_1,\ldots,\sigma_{d-1},\sigma'\}$ and $\{\sigma\in F : \gamma\subset \sigma\}=\{\sigma_1,\ldots,\sigma_{d-1}\}$. Since inside $\Pi_{\gamma}$ we have  $y_\sigma=0$ whenever  $\gamma \not \subset \sigma$, we can express points in $\Pi_\gamma$ as
lists  $(y_1,\ldots,y_{d-1})$ where each $y_j$ abbreviates $y_{\sigma_j}$.

To simplify notations, let's use $x_l$ and $\nu_l$, respectively, for the coordinate and order associated to $\sigma_l$, where $l=1,...,d-1$. Consider the flow box $(V,(t,x_1,\ldots,x_{d-1}))$ with $V=N_\gamma$ and the coordinate system
introduced in the beginning of Section~\ref{rescaling}.
In this flow box the vector field's equation reads as:
\begin{equation}
\label{flow box}
\left\{
\begin{array}{cl}
\dot{t}  & =  1\\
\dot{x}_l  & =  x_l^{\nu_l}\, H_l(t,x),\quad l=1,\ldots,d-1,
\end{array}
\right.
\end{equation}   
where $H_l(t,x)$ is defined in~\eqref{def:order}.
Integrating  
$$ \frac{d}{dt}h_{\nu_l}(x_l) = - \dot{x}_l\,x_l^{-\nu_l} 
=-H_l(t,x)\quad l=1,\ldots,d-1, $$
yields  
$$h_{\nu_l}(x_l(t)) = h_{\nu_l}( x_l(0)) - \int_0^t H_l(\varphi^s(0,x(0)))\, d s\quad l=1,\ldots,d-1 ,$$
where 
$\varphi^t$ stands for the flow of the vector field~\eqref{flow box}.
Therefore
\begin{equation}\label{edgepoincare}
P_{\gamma} (x)=\left\{ h^{-1}_{\nu_l} \left(\, h_{\nu_l}(x_l) - \int_0^{\tau(x)} H_l(\varphi^s(0,x))\, d s \right)\,  \right\}_{l=1,\ldots,d-1}
\end{equation}
where $\tau(x)$ is the time that the orbit starting at $x\in\Sigma_\gamma^-$ takes to hit the cross-section $\Sigma_\gamma^+$.

Expressing $(\Psi^X_{v,\epsilon})^{-1}$ in the coordinate system
 $(x_0,x_1,\ldots,x_{d-1})$   on the neighborhood $N_{v}$, for every point $(y_1,\ldots,y_{d-1})\in \Pi_{\gamma}$,  
$$(\Psi^X_{v,\epsilon})^{-1}(y)   =\left(1, h^{-1}_{\nu_1}( \frac{y_1}{\epsilon^2}),\ldots ,h^{-1}_{\nu_{d-1}}( \frac{y_{d-1}}{\epsilon^2}) \,\right). $$
By~\eqref{edgepoincare} we have
\begin{align*}
P_{\gamma}\left(  \Psi_{v,\epsilon}^{-1}(y) \right)
= \left(\, h^{-1}_{\nu_l}\left[ \, h_{\nu_l}(h^{-1}_{\nu_l} ( \frac{y_l}{\epsilon^2}) ) -  \int_0^{\tau(\Psi_{v,\epsilon}^{-1}(y) )} H_{l}\, d s \,\right]\,\right)_{1\leq l\leq d-1} .
\end{align*}
Hence $F^\epsilon_{\gamma} (y_1,\ldots, y_{d-1})=(y'_1(\epsilon),\ldots, y'_{d-1}(\epsilon))$,  where by Definition~\ref{defn:vcoor.ch.}
\begin{align*}
y^\prime_l(\epsilon) & =  \epsilon^2 h_{\nu_l}\left(\, h^{-1}_{\nu_l}\left[ \, h_{\nu_l}(h^{-1}_{\nu_l} ( \frac{y_l}{\epsilon^2}) ) -  \int_0^{\tau(\Psi_{v,\epsilon}^{-1}(y))} H_{l}\, d s \,\right]\,\right) \\
& = y_l - \epsilon^2 \int_0^{\tau((\Psi_{v,\epsilon})^{-1}(y))} H_{l}(\varphi^s(\Psi_{v,\epsilon}^{-1}(y)))\, d s .
\end{align*}

Notice that $\tau$ is analytic,
and together with its derivatives is locally bounded 
in a neighborhood of $\gamma \cap \Sigma_\gamma^-$. Moreover,
every derivative $(D^k \varphi^t)_x$ is a solution of a system
of linear equations with coefficients depending on $\varphi^t(x)$ and on the lower order derivatives $(D^r \varphi^t)_x$ with $r<k$.
Arguing recursively, we can prove that for any $k\geq 0$ the $k$-th order  derivatives of  $H_l(t,\varphi^t(x))$ are uniformly bounded 
in a neighborhood of $\gamma \cap \Sigma_\gamma^-$, for $0\leq t\leq \tau(x)$. 
Hence, it follows from item (1) of Lemma~\ref{hninv}
that $F_\epsilon$ converges to the identity map in the $C^k$ topology.
\end{proof}
\begin{remark}\label{sigma-gamma}
By Lemma~\ref{pgamma}, for any flowing edge $v\buildrel\gamma\over\longrightarrow v',$ we can identify the two sections $\Sigma^-_\gamma$ and $\Sigma^+_\gamma$. We will refer to the identified section simply as $\Sigma_\gamma$.
\end{remark}

Let $\gamma$, $\gamma'$ be flowing edges such that
$t(\gamma)=s(\gamma')=v$.
We denote by  $\Dscr_{\gamma,\gamma'}$  the set of points $x\in  \Sigma_{v,\gamma}$ 
such that the  forward orbit of $x$ has a transversal intersection with $\Sigma_{v,\gamma'}$.

\begin{defn}
\label{Pgammagamma'}
The local Poincar\'e map 
$$P_{\gamma,\gamma'}: \Dscr_{\gamma,\gamma'}\subset
\Sigma_{v,\gamma}\to  \Sigma_{v,\gamma'}$$
is defined by
$P_{\gamma,\gamma'}(x):= \varphi_X^{\tau(x)}(x)$, see Figure~\ref{local-global-poincare}, where
$$ \tau(x)= \min\{\, t>0\,\colon\, 
\varphi_X^{t}(x) \in \Sigma_{v,\gamma'}\, \} \;.$$
\end{defn}

  Given $k\in\Nn$   take $r=r(k,X)$ according to  Lemma~\ref{lemma:rescal}.
\begin{lemma}\label{vertexpoincare} 
Given flowing edges $\gamma,\gamma'$  
such that $t(\gamma)=s(\gamma')=v$,
let \,$\UU_{\gamma,\gamma'}^\epsilon\subset \Pi_{\gamma}(\epsilon^r)$  be the domain of the map $$F_{\gamma,\gamma'}^\epsilon:=\Psi_{v,\epsilon}^X\circ P_{\gamma,\gamma'} \circ(\Psi_{v,\epsilon}^X)^{-1}.$$
Then
\begin{align*}
\lim_{\epsilon\to0^+}\left(F_{\gamma,\gamma'}^\epsilon\right)_{|_{\UU_{\gamma,\gamma'}^\epsilon}}=
 L_{\gamma,\gamma'} \; 
\end{align*}
in the $C^k$ topology, in the sense of Definition~\ref{remark:convergence}.
\end{lemma}

\begin{proof}
Setting $F_{v}=\{\sigma_0,\sigma_1,\ldots, \sigma_{d-1}\}$
consider the system of coordinates  $(y_0,y_1,\ldots, y_{d-1})$ on $\Pi_v$ where each $y_j$ abbreviates
$y_{\sigma_j}$.
Assume the facets in $F_v$ were ordered in a way that
$$\Pi_{\gamma}=\{y\in \Pi_{v}: y_0=0\}\quad\text{and}\quad \Pi_{\gamma'}=\{y\in \Pi_{v}: y_{d-1}=0\}.$$

Let 
\begin{align*}
&\partial_{\gamma}\Pi_{v}(\epsilon^r):=\{y\in\Pi_{v}(\epsilon^r) : y_0=\epsilon^r\}\\
&\partial_{\gamma'}\Pi_{v}(\epsilon^r):=\{y\in\Pi_{v}(\epsilon^r) : y_{d-1}=\epsilon^r\}
\end{align*}
 be the boundary facets of the sector $\Pi_{v}(\epsilon^r)$ defined in~\eqref{pivepsilon}, respectively, parallel to $\Pi_{\gamma}$ and $\Pi_{\gamma'}$. By Lemma~\ref{lemma:rescal},
the Poincar\'{e} map of the vector field
$(\Psi_{v,\epsilon}^X)_\ast X=\epsilon^2 \tilde{X}_v^\epsilon$ from
$\partial_{\gamma}\Pi_v(\epsilon^r)$ to
$\partial_{\gamma'}\Pi_v(\epsilon^r)$ converges in the
$C^k$ topology to $(L_{\gamma,\gamma^\prime})_{|_{\Pi_{\gamma,\gamma'}}}$. We are left to prove that, see Figure~\ref{localpoincareflow}, the Poincar\'{e} maps of this vector field from
$$\Pi_{\gamma}(\epsilon^r)=\{y\in\Pi_{v}: y_0=0\; \text{and}\;  y_1,\ldots, y_{d-1}\leq\epsilon^r\} $$ to $\partial_{\gamma}\Pi_v(\epsilon^r)$,   and 
from   $\partial_{\gamma'}\Pi_v(\epsilon^r)$ to 
$$\Pi_{\gamma'}(\epsilon^r)=\{y\in\Pi_{v}:   y_{d-1}=0\; \text{and}\; y_0,\ldots, y_{d-2} \leq\epsilon^r  \} $$ 
converge to the identity maps in the $C^k$ topology as $\epsilon\to 0^+$.

\begin{figure}[h]
\begin{center}
	\includegraphics[width=9cm]{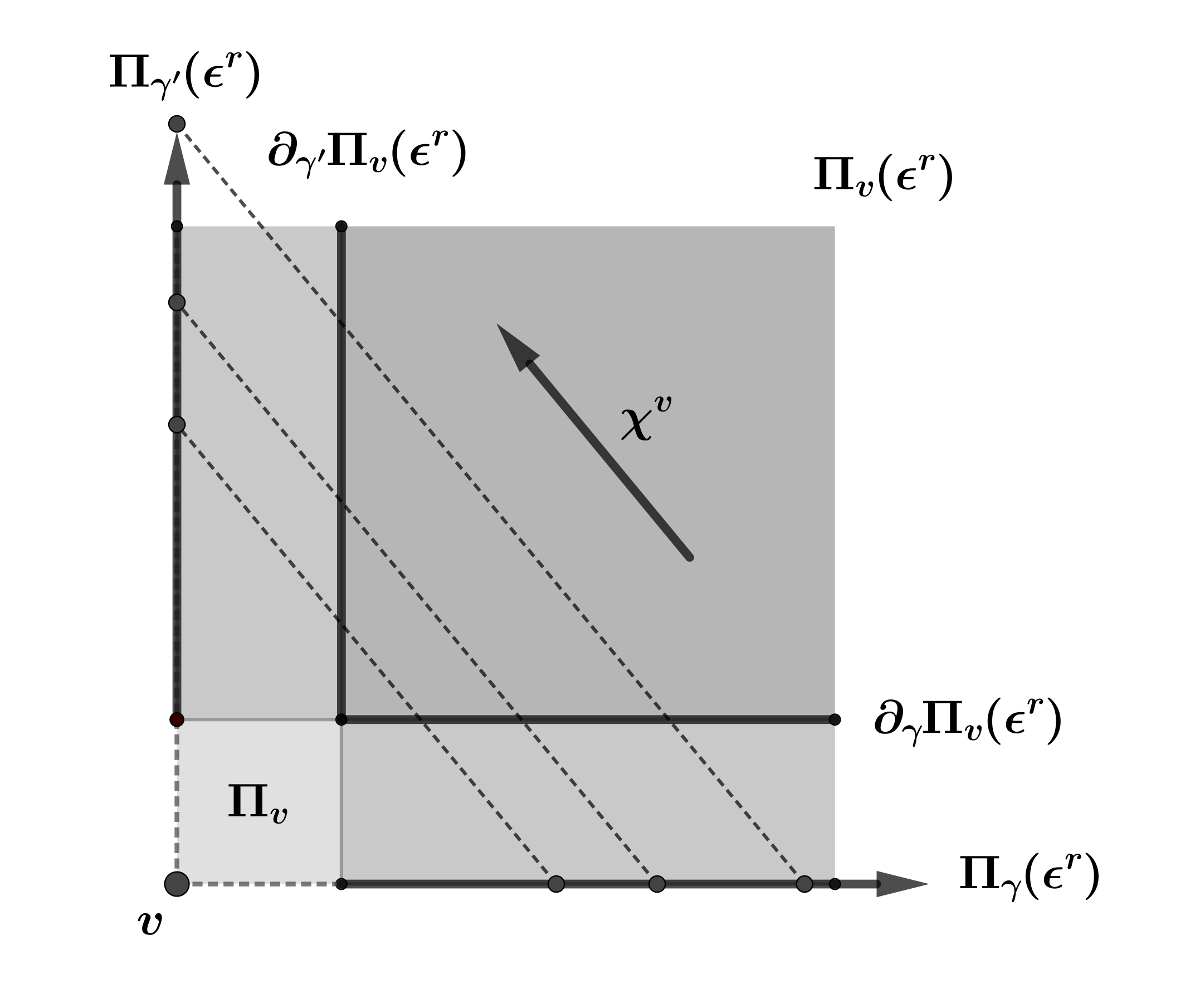}
\end{center}
\caption{The local map $L_{\gamma,\gamma'}$ factors as a composition of three projections.}
\label{localpoincareflow}
\end{figure}

 The two convergences are analogous and we only prove the first one. The argument is similar to that of Lemma~\ref{pgamma}, but  instead of~\eqref{flow box} we consider the equations  of $X$
\begin{equation}
\label{X in v-coords}
 \dot{x}_l =  x_l^{\nu_l}\, H_l(x)  \quad 0\leq l\leq d-1  
\end{equation}
represented in the system of coordinates $(x_0,x_1,\ldots, x_{d-1})$ on $N_v$.

Notice that $(\Psi_{v,\epsilon}^X)^{-1} \Pi_{\gamma}(\epsilon^r)\subset \Sigma_{v,\gamma}$ is defined by the 
conditions 
$$
x_0=1 \;\text{ and } \; 0< x_l <h^{-1}_{\nu_l}\left(\frac{1}{\epsilon^{2-r}}\right), \quad 1\leq l \leq d-1 .
$$

Likewise, $\Sigma^\epsilon_{v,\gamma}:=(\Psi_{v,\epsilon}^X)^{-1} \partial_{\gamma}\Pi_{v}(\epsilon^r)$ is defined by
$x_0=h^{-1}_{\nu_0}(\frac{1}{\epsilon^{2-r}})$ 
and the same conditions above in the remaining coordinates.
Let  $\tau^\epsilon(x)$  denote the time that the orbit starting at $x\in\Sigma_{v,\gamma}$ takes to hit the cross-section $\Sigma_{v,\gamma}^\epsilon$.  Integrating  the first component of \eqref{X in v-coords}, we have
\begin{equation}
\label{integralformula0}
h_{\nu_0}\left(h_{\nu_0}^{-1}\left(\frac{1}{\epsilon^{2-r}}\right)\right)  
 -h_{\nu_0}(1)=-\int_0^{\tau^\epsilon(x)}H_0(\varphi^s(x))d s.
\end{equation} 
Since $-H_0(v)=\chi^{v}_{\sigma_0}>0$ there exists a neighborhood $U_{v}$ of $v$ where $-H_0$ takes positive values. We can take a constant $C>0$ and shrink $U_{v}$ so that $-H_0\geq \frac{1}{C}$
and $\norm{D^r H_0}\leq C$ for all $1\leq r\leq k$   on $U_{v}$.
Without loss of generality we may assume that $\Sigma_{v,\gamma}$ is contained in $U_{v}$. From~\eqref{integralformula0} we have 
\begin{equation}\label{integralformula1}
\int_0^{\tau^\epsilon(x)}-H_0(\varphi^s(x))\, d s=\frac{1}{\epsilon^{2-r}} \, ,
\end{equation}
which implies  
$$
\tau^\epsilon(x)\leq \frac{C}{\epsilon^{2-r}}\, .
$$
Differentiating both sides of \eqref{integralformula1} with respect to $x_l$, for $1\leq l \leq d-1$, we obtain 
\begin{equation}
\label{integralformula2}
 H_0(\varphi^{\tau^\epsilon(x)}(x)) \frac{\partial \tau^\epsilon(x)}{\partial x_l} + \int_0^{\tau^\epsilon(x)}  \nabla H_0  (\varphi^s(x)) \cdot \frac{\partial \varphi^s(x)}{\partial x_l}\, d s=0 .
\end{equation} 
Similar formulas can be driven for higher order
derivatives of $\tau^\epsilon$.

Arguing as in Lemma~\ref{pgamma},
we can bound the derivatives of the flow $\varphi^s(x)$.
Since the  derivatives of $H_0$ are also bounded,
we infer from~\eqref{integralformula2}, and its higher order analogues, that the function
$\tau^\epsilon$ has bounded derivatives up to order $k$.
 Finally, repeating the argument in the proof of Lemma~\ref{pgamma}, we conclude that the Poincar\'{e} map from $\Pi_{\gamma}(\epsilon^r)$ to $\partial_{\gamma}\Pi_{v}(\epsilon^r)$ converges to identity in the $C^k$ topology as $\epsilon\to 0^+$.   
\end{proof}

\begin{defn}
Given a heteroclinic path  $\xi=(\gamma_0,\gamma_1, \ldots,\gamma_m)$, the composition
$$P_\xi:=( P _{\gamma_m}\circ P_{\gamma_{m-1},\gamma_m}) \circ\ldots\circ (P_{\gamma_1}\circ P_{\gamma_0,\gamma_1})$$
is referred to as the {\em Poincar\'e map} of the vector field $X$ along $\xi$.
The   domain of this composition is denoted by $\UU_\xi$.
\end{defn}

Lemmas~\ref{pgamma}  and Lemma~\ref{vertexpoincare} imply that  given a  heteroclinic path  $\xi$, the  asymptotic behavior of the  Poincar\'e map $P_\xi$  along $\xi$
is given by the corresponding Poincar\'e map $\pi_\xi$ of the skeleton vector field $\chi$.
More precisely, given $k\in\Nn$ and  taking $r=r(k,X)$ according to  Lemma~\ref{lemma:rescal}  we have
\begin{proposition}\label{pathpoincare}
Given a heteroclinic path  $\xi=(\gamma_0,\ldots,\gamma_m)$ with $v_0=s(\gamma_0)$ and $v_m=s(\gamma_m)$,
let $\UU_\xi^\epsilon$  be the domain of the composite map $F_\xi^\epsilon:=\Psi_{v_m,\epsilon}^X\circ P_\xi \circ(\Psi_{v_0,\epsilon}^X)^{-1}$ from $\Pi_{\gamma_0}(\epsilon^r)$
into $\Pi_{\gamma_m}(\epsilon^r)$.
Then
\begin{align*}
\lim_{\epsilon\to0^+}\left(F_{\xi}^\epsilon\right)_{|_{\UU_{\xi}^\epsilon}}=\pi_\xi 
\end{align*}
in the $C^k$ topology, in the sense of Definition~\ref{remark:convergence}.
\end{proposition}

\begin{proof}
Follows immediately from Lemmas~\ref{pgamma} and~\ref{vertexpoincare}.
\end{proof}

As mentioned in the introduction, we are interested in studying the flow of $X$ along heteroclinic cycles on the polytope's vertex-edge network. 
To encode the semi-global dynamics of the flow $\varphi_X^t$ along the cycles, we use Poincar\'e return maps to a system of cross-sections $\Sigma_\gamma$, see Remark~\ref{sigma-gamma},  placed at the edges of a structural set, see Definition~\ref{structural:set}. Any orbit of the flow  $\varphi_X^t$ that shadows some heteroclinic cycle must intersect these cross-sections  in a recurrent way.

\begin{defn}
Let $X\in\fX(\Gamma^ d)$ be a vector field with a structural set  $S\subset E$. We define the $S$-Poincar\'e map $P _S:\UU_{S} \subset \Sigma_S\to \Sigma_S$ setting  $\Sigma_S:= \cup_{\gamma\in S} \Sigma_\gamma$,
 $\UU_{S}:= \cup_{\xi\in \Bscr_S(\chi)} \UU_{\xi}$
 and $P_S(p):=P_\xi(p)$ for all $p\in \UU_{\xi}$.
Note that the domains $\UU_{\xi}$ and $\UU_{\xi'}$
are disjoint for $\xi\neq \xi'$ in $\Bscr_S(\chi)$.
\end{defn}

%
%

By construction the suspension of the $S$-Poincar\'e map
$P_{S}:D_{S} \subset \Sigma_S\to \Sigma_S$ embeds (up to a time re-parametrization) in the flow of the vector field $X$.
In this sense the dynamics of the map $P_{S}$ encapsulates the qualitative behavior
of the flow $\varphi_X^t$ of $X$ along the edges of $\Gamma^d$.

\begin{theorem}\label{asymp:main:theorem}
Let $X\in \fX(\Gamma^d)$ be a regular vector field 
 with skeleton vector field $\chi$ and a structural set $S\subset E_\chi$. Then  
$$ \lim_{\epsilon\to 0^ +}  \Psi_\epsilon \circ \Poin{X}{S}\circ
(\Psi_\epsilon)^{-1} = \skPoin{\chi}{S}  $$
 in the $C^\infty$ topology, in the sense of Definition~\ref{remark:convergence}.
\end{theorem}

\begin{proof}
Follows from Proposition~\ref{pathpoincare}.
\end{proof}


\section{Asymptotic integrals of motion}
\label{AFI}

In this section we introduce a probe space $\HH(\Gamma^d)$ for integrals of motion
of  the vector fields in $\fX(\Gamma^d)$. This space consists of analytic functions in $\inter(\Gamma^d)$ with poles at the polytope's facets.
We show that a function $h\in \HH(\Gamma^d)$ rescales
to a piecewise linear function $\eta:\CC^\ast(\Gamma^d)\to\Rr$ on the dual cone.
Moreover, if $h\in \HH(\Gamma^d)$ is an integral of motion of a
vector field $X\in\fX(\Gamma^d)$ then $\eta$ is also an integral of motion for the
piecewise linear flow of the skeleton vector field $\chi$ of $X$.

Recalling that $\{f_\sigma\}_{\sigma\in F}$ is a defining family of the polytope $\Gamma^d$, let
$\FF = \{ h_n\circ f_\sigma,:\; n\geq 1, \sigma\in F\,\}$ 
where   $h_n$ was introduced in~\eqref{defn:hn}, and  define $\HH(\Gamma^d)$ to be the linear span of $\A(\Gamma^d)\cup\FF$.
Since functions in the set $\FF$
are linearly independent and $\HH(\Gamma^d)=\A(\Gamma^d)\oplus \langle \FF \rangle$,
each $h\in\HH$ can be uniquely decomposed as
\begin{equation}
\label{funcHH}
h= g +  \;
\sum_{n=1}^\infty \;\sum_{\sigma\in F}  \mu_{n\,\sigma }\, ( h_n\circ f_\sigma )\;, 
\end{equation}
with $g\in\A(\Gamma^d)$,
where only a finite number of coefficients $\mu_{n\,\sigma}$
are nonzero.
Note that the differential $dh$ is given by the
expression:
\begin{equation}
\label{dhh}
dh= dg\;- \; \;
\sum_{n=1}^\infty \;\sum_{\sigma\in F} 
\mu_{n\,\sigma} \,\frac{d f_\sigma }{ (f_\sigma)^{n}}\;. 
\end{equation}

We define the {\em order of $h$ at $\sigma$} to be  the number 
$$ \nu^h(\sigma)=\max\{n\in\Nn\colon  \mu_{n\,\sigma}\neq 0\} ,$$
with  $\nu^h(\sigma)=0$ if all $\mu_{n\,\sigma}=0$.
The map $\nu^h:F\to \Nn$ is referred to as the {\em order function} of $h$.

\begin{defn}\label{def_skel_ham}
The character of  $h$ at $\sigma$
is the coefficient $\eta^h(\sigma)=\mu_{n\,\sigma}$ cor\-res\-pon\-ding
to the term with largest order $n=\nu^h(\sigma)$. The character is undefined
if $\nu^h(\sigma)=0$. We say that the function
$$\eta^h:\CC^\ast(\Gamma^d)\to \Rr\;,\quad \eta^h(y):=\sum_{\sigma\in F}\eta^h(\sigma)\,y_\sigma\;$$
is the  skeleton of $h$.
\end{defn} 

\begin{proposition}
	\label{la_inv}
Given   a regular vector field $X\in\fX(\Gamma^d)$ and a  function $h\in\HH(\Gamma^d)$ with the same order function $\nu:F\to \Nn$,
let $\chi$ be the skeleton of $X$ and 	$\eta:\CC^\ast(\Gamma^d)\to\Rr$
be the skeleton of $h$. Then
\begin{enumerate}
	\item 
	$\displaystyle \eta =\lim_{\epsilon \rightarrow 0^+} \epsilon^2  
	h\circ (\Psi^X_{v,\epsilon})^{-1}$  over $\inter(\Pi_v)$ for any vertex $v$, with convergence in the $C^\infty$ topology.
	\item 
	$\displaystyle d\eta = \lim_{\epsilon \rightarrow 0^+} \epsilon^2
	\left[ (\Psi^X_{v,\epsilon})^{-1} \right]^\ast \left( d h \right)$ over  $\inter(\Pi_v)$ for any vertex $v$,	with convergence in the $C^\infty$ topology. 
	\item If $h$ is invariant under the flow of $X$, \ie $dh(X)\equiv 0$,
	then $\eta$ is invariant under the skeleton  flow of $\chi$, \ie $d\eta(\chi)\equiv 0$.
\end{enumerate} 
\end{proposition}

\begin{remark}
Since $\nu$ is the order function of $X$, $\nu(\sigma)\geq 1$ for every facet $\sigma\in F$. Hence, because $\nu$ is also the order function of $h$ the skeleton character $\eta(\sigma)=\eta^h(\sigma)$ is well defined for all facets $\sigma\in F$.
\end{remark}

\begin{proof}
	Consider the system of local coordinates
	$x=(x_\sigma)_{\sigma\in F_v} =( f_\sigma(q))_{\sigma\in F_v}$ on the neighborhood $N_v$. According to decomposition~(\ref{funcHH}) we can write
	$$h(x)= g(x) +  \;
	\sum_{\sigma\in F_v} \; \sum_{n=1}^{\nu(\sigma)} \;  \mu_{n\,\sigma }\, h_n (x_\sigma) \;,$$ 
	where $g(x)$ is analytic in $N_v$.
	Therefore, by item (2) of Lemma~\ref{hninv},
	$$\epsilon^2\, h\circ (\Psi^X_\epsilon)^{-1}(y)= 
	\epsilon^2\, g\circ (\Psi^X_\epsilon)^{-1}(y) +\epsilon^2\, 
	\sum_{\sigma\in F_v} \sum_{n=1}^{\nu(\sigma)}
	\mu_{n\,\sigma} \, \left( h_n\circ h_{\nu(\sigma)}^{-1}\right)\left(\frac{y_\sigma}{\epsilon^2}\right)
	$$
	converges in the $C^\infty$ topology to
	$\sum_{\sigma\in F_v} \mu_{\nu(\sigma)\,\sigma}\, y_\sigma =\eta(y)$ over the sector $\inter(\Pi_v)$. This proves (1) and also implies (2).
	
	For item (3) we use the following abstract result. 
	Given a smooth function $h$ and a smooth vector field $X$ on a manifold $M$,
	and given a diffeomorphism $\Psi:M\to N$,
	$$ dh(X)\circ\Psi^{-1} = d(h\circ\Psi^{-1})[ \Psi_\ast X) ] . $$
	
	Since we are assuming that $dh(X)\equiv 0$, by item (2) and Lemma~\ref{lemma:rescal}
	\begin{align*}
	0 &= dh (X)\circ \Psi_{v,\epsilon}^X
	= d(h\circ (\Psi_{v,\epsilon}^X)^{-1})[(\Psi_{v,\epsilon}^X)_\ast X]\\
	&= d(\epsilon^2 h\circ (\Psi_{v,\epsilon}^X)^{-1})[\epsilon^{-2} (\Psi_{v,\epsilon}^X)_\ast X]\\
		&= d(\epsilon^2 h\circ (\Psi_{v,\epsilon}^X)^{-1}) [  \, \tilde X_v^\epsilon\,  ]
		\longrightarrow d \eta(\chi^v) 
	\end{align*}
	as $\epsilon\to 0$. This proves that the piecewise linear function $\eta$ is invariant under the 
	 flow of the skeleton vector field $\chi$.
\end{proof}

A continuous function $h:M\to \Rr$ is said to be {\em proper}  if for all real numbers $a<b$ the pre-image
$f^{-1}[a,b]\subset M $ is compact.

\begin{proposition}
\label{prop proper}
If $h\in\HH(\Gamma^d)$ is proper in $\inter(\Gamma^d)$
with order function $\nu\geq 1$ then its skeleton
$\eta:\CC^\ast(\Gamma^d)\to \Rr$ is also a proper function.
\end{proposition}

\begin{proof}
Fix a vertex $v$ and let $F_v=\{\sigma_1,\ldots, \sigma_d\}$.
Take the usual system of affine coordinates $(x_1,\ldots, x_d)$ on the neighborhood $N_v$ where $x_j=f_{\sigma_j}$. In these coordinates $h$ can be written as
$$ h(x)=g(x) + \sum_{j=1}^d \frac{p_j(x)}{x^{\nu_j}}$$
where $g(x)$ is analytic in $N_v$, $\nu_j$ is
the order of $h$ at $\sigma_j$ and each $p_j(x)$ is a polynomial function such that $\mu_j=p_j(0)\neq 0$
is the character of $h$ at $\sigma_j$. On the sector $\Pi_v$   the skeleton $\eta$ of $h$ is given  by 
$\eta(y_1,\ldots, y_d)= \sum_{j=1}^d \mu_j  y_j$.
Since $h$ is proper, the level set $h^{-1}(0)$ is compact,
which implies that $h$ does not change sign in a small neighborhood of $v$. 
Hence we can assume that $h>0$ on $N_v$.
Because
 $h(x)$ is equal to $\sum_{j=1}^d \frac{\mu_j}{x^{\nu_j}}$ 
 pus higher order terms as   $x\to v$, all coefficients $\mu_j$ must be positive. Therefore
 $$ \eta^{-1}([a,b])\cap \Pi_v \subset \{ (y_1,\ldots, y_d): y_j\geq 0,\;
 \sum_{j=1}^d \mu_j y_j\leq b \}$$
 is a compact set. Because $v$ is arbitrary,
   $\eta^{-1}([a,b])$ is also compact.
\end{proof}

\begin{remark}
	\label{rmk conservative 1}
	Polymatrix replicator systems form a large class of models
	in EGT,  that includes  replicator and bimatrix replicator systems,   falling within the scope of this work. 
	The phase space of polymatrix replicators are prisms (products of simplexes),  basic examples of  simple polytopes.	In~\cite{AD2014} the first two authors have characterized the class of Hamiltonian polymatrix replicator systems  w.r.t. a class of  algebraic Poisson structures. All these models illustrate the conclusions of propositions~\ref{la_inv} and~\ref{prop proper}.
\end{remark}

\begin{remark}
	\label{rmk conservative 2}
	If $X$ is a Hamiltonian polymatrix replicator vector field
	w.r.t. some algebraic Poisson structure in the interior of a prism $\Gamma^d$, which has  a proper  Hamiltonian function $h$,  then its skeleton flow map is volume preserving on each level set of 	the skeleton  of $h$.
	This fact will not be proved here, see more  in Section~\ref{furtherwork}.
\end{remark}

\bigskip

Throughout the rest of this section we assume:
\begin{enumerate}
\item $X\in\fX(\Gamma^d)$ is a regular vector field, with skeleton $\chi$,
such that all vertexes are of saddle type and every edge is either neutral or a flowing edge;
\item $X$ has integrals of motion $h_1,\ldots, h_k\in\HH(\Gamma^d)$, all with the same order function as $X$;
\item    $\eta_1,\ldots, \eta_k:\CC^\ast(\Gamma^d)\to\Rr$ are respectively the skeletons of  $h_1,\ldots, h_k$
and the forms $d\eta_1,\ldots, d\eta_k$ are linearly independent on every sector $\Pi_v$.
\end{enumerate}

Consider the function $\eta:\CC^\ast(\Gamma^d)\to\Rr^k$
defined by
$$ \eta(y):=\left(\eta_1(y),\ldots, \eta_k(y) \right)\,. $$

Given a structural set $S$ of $\chi$ and $c\in\Rr^k$   we define
\begin{equation}
\label{Delta S c}
 \Delta_{S,c}:=\Pi_S\cap  \eta^{-1}(c).
\end{equation}
Given an edge $\gamma$ or a branch $\xi\in \Bscr_S(\chi)$
we also define
\begin{equation}
\label{Delta gama xi c}
\Delta_{\gamma,c}:=\Pi_\gamma\cap\eta^{-1}(c)\,,
\quad \Delta_{\xi,c}:=\Pi_\xi\cap\eta^{-1}(c)\,.
\end{equation}
Notice that
\begin{equation}
\label{Delta S c = union Delta xi c}
\Delta_{S,c} =\bigcup_{\xi\in\Bscr_S(\chi)} \Delta_{\xi,c}\,. 
\end{equation}

\begin{theorem}
\label{thm level dynamcics}
 Under assumptions (1)-(3), given a structural set $S$ of $\chi$,  the skeleton flow map
	$\pi_S:D_S\to \Pi_S$ induces a closed dynamical system 
	on every level set $\Delta_{S,c}$ with $c=(c_1,\ldots, c_k)\in\Rr^k$.
\end{theorem}

\begin{proof}
Follows from  propositions~\ref{partition} and~\ref{la_inv}.
\end{proof}

\begin{theorem}
\label{thm horse-shoes}
 Under assumptions (1)-(3), given a structural set $S$ of $\chi$, 
if $p\in \Delta_{S,c}$ is a hyperbolic periodic point
	of ${\pi_S}_{\vert \Delta_{S,c}}$, and $q$ is an associated transversal homoclinic point whose orbit has a compact closure contained in $\Delta_{S,c}$, then there exists a (compact) hyperbolic basic set contained in $\Delta_{S,c}$ for the map ${\pi_S}_{\vert \Delta_{S,c}}$.
	Moreover each level set	
	$$  L_\epsilon:= \Gamma^d\cap \bigcap_{j=1}^k \left\{ h_j = \frac{c_j}{\epsilon^2}\right\} ,$$
	with $\epsilon$ sufficiently small, contains a hyperbolic basic set for the $S$-Poincar\'e map ${P_S}_{\vert L_\epsilon\cap \UU_S}$, conjugated to the previous one.
\end{theorem}

\begin{proof}
Given $p\in \Delta_{S,c}$ and its associated transversal homoclinic point $q$ consider an open neighborhood $U$ of the $\pi_S$-orbits of $p$ and $q$
whose closure satisfies
$$\overline{U}\subset D_S=\bigcup_{\xi\in\Bscr_S(\chi)}\Pi_\xi .$$
Because $\Lambda_0:=\{\pi_S^j(p):j\in\Zz\}\cup \{\pi_S^j(q):j\in\Zz\}\subset U$
is a hyperbolic set, reducing the size of $U$, the maximal invariant set 
$\Lambda =\cap_{j\in\Zz} \pi_S^{-j}(\overline{U})$ is a hyperbolic basic set for ${\pi_S}_{\vert \Delta_{S,c} }$.

Consider now the system of cross-sections
$\Sigma_S:= \cup_{\gamma\in S} \Sigma_\gamma^-$ transversal to the flow of  $X$ and let $P_S$ denote the induced Poincar\'e map
on $\Sigma_S$. By Theorem~\ref{asymp:main:theorem}, the conjugated Poincar\'e map
$\tilde P_S^\epsilon:=\Psi_\epsilon \circ \Poin{X}{S}\circ
(\Psi_\epsilon)^{-1}$ on the (invariant)  level set
$$ \Pi_S \cap \Psi^X_\epsilon(L_\epsilon)=   \Pi_S\cap \bigcap_{j=1}^k \left\{ \epsilon^2\, h_j \circ(\Psi^X_\epsilon)^{-1} = c_j \right\} $$
can be seen as a small perturbation of the skeleton flow map ${\pi_S}_{ \vert \Delta_{S,c} }$. Notice that, according to Proposition~\ref{la_inv}
the level set $\Pi_S \cap \Psi^X_\epsilon(L_\epsilon)$ converges to
$\Delta_{S,c}$ as $\epsilon\to 0$. Thus, because hyperbolic basic sets are structurally stable, see\cite[Theorem 8.3]{Shub1987},  there exists a hyperbolic basic set
$\tilde{\Lambda}_\epsilon$ for the conjugated Poincar\'e map
$\tilde P_S^\epsilon$ on $\Pi_S \cap \Psi^X_\epsilon(L_\epsilon)$.
Finally by conjugacy $\Lambda_\epsilon:= (\Psi^X_\epsilon)^{-1}(\tilde \Lambda_\epsilon)\subset L_\epsilon$ is a hyperbolic basic set for the Poincar\'e map ${P_S}_{\vert\Sigma_S\cap L_\epsilon}$ of the flow of $X$.
\end{proof}

\section{ Procedure to analyze the dynamics}
\label{dyn-analysis}

In this section we briefly describe the computational steps
that through  Theorem~\ref{thm horse-shoes} lead to the detection of hyperbolic basic sets.

\bigskip

\textbf{Input data:} \; 
The polytope $\Gamma^d$ and the vector field
$X\in\fX(\Gamma^d)$.

\bigskip

\newtheorem{step}{Step}

\begin{step}
Compute the character $\chi$ of $X$ and draw its flowing-edge graph.  \emph{This step  involves computing some derivatives at the vertex singularities. 
It can be done  through a computer algebra system algorithm.}
\end{step}

\begin{step}
Find a structural set $S$ for $\chi$. \emph{The search can be done by inspection if the flowing-edge graph is simple, or else using an algorithm for that purpose.}
\end{step}

\begin{step}
Determine all $S$-branches of $\chi$. \emph{ Once the structural set is known, a simple algorithm
determines its branches.}
\end{step}

\begin{step}
\label{step 1st integrals}
Find the integrals of motion of $X$ in $\HH(\Gamma^d)$ and determine their skeletons.
\emph{For instance if $X$ is Hamiltonian with respect to some Poisson structure, join to the Hamiltonian function of $X$ all the Casimirs of its  Poisson structure.}
\end{step}

\begin{step}
Make explicit the skeleton flow map $\pi_S:\Pi_S\to\Pi_S$.
\emph{ Use an algorithm to compute for each branch $\xi\in \Bscr(\chi)$ 
the matrix $M_\xi$  as well as the inequalities defining the  domain $\Pi_\xi$. Then represent (computationally)  the flow map $\pi_S$ as a function defined by cases.}
\end{step}

\begin{step}
Compute some random orbits of $\pi_S$ 	and determine their i\-ti\-ne\-ra\-ries, \emph{
using the previous step  representation of the flow map $\pi_S$.}
\end{step}

\begin{step}
Pick a few heteroclinic cycles $\xi$ from the previous itineraries  and  compute the eigenvalues and eigenvectors of $M_\xi$.\emph{ Use an algorithm to compute a matrix's eigenvalues and eigenvectors. Every matrix $M_\xi$ is a projection of $\Rr^F$ onto a $(d-1)$-dimensional subspace and hence has exactly $\vert F\vert -d+1$ zero eigenvalues. 
If  $k$ integrals of motion were found   in Step~\ref{step 1st integrals}, the eigenspace of $M_\xi$ associated with eigenvalue $1$, $\Ker(M_\xi-I)$, must have at least dimension $k$.}
\end{step}

\begin{step}
Among the positive eigenvectors associated with eigenvalue $1$ of $M_\xi$ look for saddle type periodic  points $p=\pi_S^n(p)$.
\emph{ Any eigenvector $y\in \Ker(M_\xi-I)$ with non-negative entries belongs to the dual cone and is a prospective periodic point of $\pi_S$, but one still needs to verify that $y\in \Pi_\xi$.
}
\end{step}

\begin{step}
Fix the level $c$ such that $p\in \Delta_{S,c}$.\emph{}
\end{step}

\begin{step}
Compute the local stable and unstable manifolds of
	$p$ inside the component $\Delta_{\xi,c}\subseteq \Delta_{S,c}$ that contains $p$.\emph{}
\end{step}

\begin{step}
Iterate the local stable manifold  backward and the local unstable manifold  forward,
looking for transversal intersections.

\emph{}
\end{step}

\bigskip


\section{Examples}
\label{example section}

We will now present two examples, both replicators, 
illustrating the procedure  detailed in the previous section.
The second example belongs to a class of systems studied by Wang~\etal\cite{wang-wu-ruan}.

By Theorem~\ref{thm horse-shoes} the dynamics of these two systems  are chaotic, \ie their flows contain horse-shoes, in  sufficiently  large levels.


\bigskip
\subsection{Example 1}

Consider the replicator system defined by matrix
\[A=\left(\begin{array}{rrrrrr}
			0  & -2 & 2  & 0  & 0  & 3  \\
			2  & 0  & -2 & 0  & 0  & 0  \\
			-2 & 2  & 0  & -2 & 2  & 0  \\
			0  & 0  & 2  & 0  & -2 & 0  \\
			0  & 0  & -2 & 2  & 0  & -3 \\
			-3 & 0  & 0  & 0  & 3  & 0
			  \end{array}\right)\,.\]

We denote by $X_A$ the vector field associated to this replicator defined on the simplex  $\Delta^5\,.$
The point
$$ q= \left(\frac{72}{245},\frac{33}{280},\frac{72}{245},\frac{33}{280},\frac{72}{245},-\frac{23}{196}\right)\in\Rr^6 $$
satisfies
\begin{itemize}
   \item[(1)] $(Aq)_1=(Aq)_2=(Aq)_3=(Aq)_4=(Aq)_5=0$;
   \item[(2)] $q_1+q_2+q_3+q_4+q_5=1$\,,
\end{itemize}
and hence is an equilibrium of $X_A$, see~\cite{ADP2015}*{Definition 4.1}.
Since matrix $A$ is skew-symmetric, the associated replicator is conservative, \ie $X_A$ is Hamiltonian with respect to some stratified Poisson structure on $\Delta^5$, see~\cite{ADP2015}*{Definition 4.3, Proposition 12}.

The polytope $\Delta^5$ has five faces labeled by an index $j$ ranging from $1$ to $6$, and 
designated by $\sigma_1,\ldots, \sigma_6$.
The vertexes of the  phase space  $\Delta^5$  are also labeled by $i\in \{1,\dots,6\}$, where the label $i$ stands for the point $e_i\in \Delta^5$.
To simplify the notation we designate the simplex's vertexes by  $v_1,\ldots, v_6$.
The skeleton character $\chi_A$ of $X_A$ is displayed in Table~\ref{ex1:chars}.

\begin{table}[h] 
\centering
\begin{tabular}{c||rrrrrr}
 
\,$\chi^v_\sigma$ & $\sigma_1$  & $\sigma_2$  & $\sigma_3$  & $\sigma_4$  & $\sigma_5$ & $\sigma_6$ \\
\hline \hline \\[-4mm]
\,$v_1$           &  $0$   &  $-2$  &  $2$   &  $0$   &  $0$   &  $3$   \\
\,$v_2$           &  $2$   &  $0$   &  $-2$  &  $0$   &  $0$   &  $0$   \\
\,$v_3$           &  $-2$  &  $2$   &  $0$   &  $-2$  &  $2$   &  $0$   \\
\,$v_4$           &  $0$   &  $0$   &  $2$   &  $0$   &  $-2$  &  $0$   \\
\,$v_5$           &  $0$   &  $0$   &  $-2$  &  $2$   &  $0$   &  $-3$  \\
\,$v_6$           &  $-3$  &  $0$   &  $0$   &  $0$   &  $3$   &  $0$   \\
\hline \hline
\end{tabular}
\vspace{.3cm}
\caption{\footnotesize{The skeleton character $\chi_A$ of $X_A$.}}
\label{ex1:chars}
\end{table}

The edges of $\Delta^5$ are designated by $\gamma_1,\ldots, \gamma_{15}$,
according to Table~\ref{ex1:edges}, where we write $\gamma=(i\,j)$ to mean that
$\gamma$ is an edge connecting the vertexes $v_i$ and $v_j$.
This model has $15$ edges: $7$ neutral edges, $\gamma_3, \gamma_4, \gamma_7, \gamma_8, \gamma_9, \gamma_{12}, \gamma_{14}$,
and $8$ flowing-edges, $\gamma_1, \gamma_2, \gamma_5 \gamma_6, \gamma_{10}, \gamma_{11}, \gamma_{13}, \gamma_{15}$. The flowing-edge directed graph of  $\chi_A$  is depicted in Figure~\ref{oriented_graph_type5+1}.

\begin{table}[h]
\centering
\begin{tabular}[c]{rrrrr}
\\
\hline
\\[-3mm]
$\gamma_1=(1\, 2)$    & $\gamma_2=(1\,3)$    & $\gamma_3=(1\,4)$    & $\gamma_4=(1\,5)$    & $\gamma_5=(1\,6)$\\
$\gamma_6=(2\,3)$    & $\gamma_7=(2\,4)$    & $\gamma_8=(2\,5)$    & $\gamma_9=(2\,6)$    & $\gamma_{10}=(3\,4)$\\
$\gamma_{11}=(3\,5)$ & $\gamma_{12}=(3\,6)$ & $\gamma_{13}=(4\,5)$ & $\gamma_{14}=(4\,6)$ & $\gamma_{15}=(5\,6)$
\hspace{-.27cm} \vspace{1mm} \\
\hline \vspace{-.2cm}
\end{tabular}
\caption{\footnotesize{Edge labels.}}\label{ex1:edges}
\end{table}

From this graph we can see that
$$ S=\{ \, \gamma_6=(2\,3), \, \gamma_{10}=(3\,4) \, \}  $$
is a  structural set for $\chi_A$, see Definition~\ref{structural:set},
whose $S$-branches denoted by $\xi_1,\dots,\xi_5$ are displayed in Table~\ref{branch_table_ex1},
where we write $\xi_i=(j\, k\, l\, \dots)$ to mean that
$\xi_i$ is a path from vertex $v_j$ passing along vertices $v_k, v_l, \dots$\,.

\begin{figure}[h]
\begin{center}
\includefigure{width=8cm}{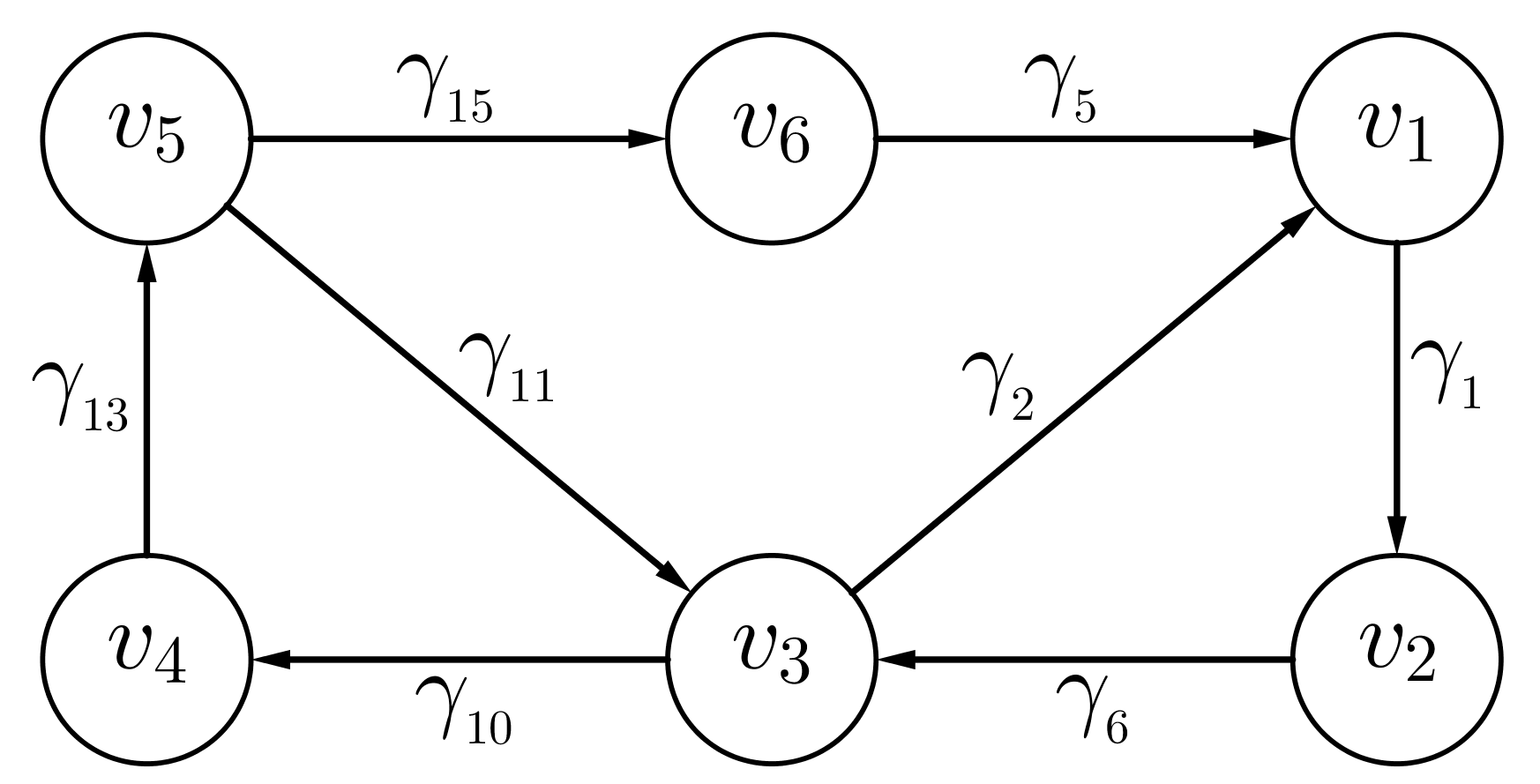}
\vspace{-.3cm}
\caption{\footnotesize{The oriented graph of $\chi_A$.}} \label{oriented_graph_type5+1}
\end{center} 
\end{figure}

\begin{table}[h]
\centering
\begin{tabular}{c|c|c}
From\textbackslash To & $\gamma_6=(2\,3)$ & $\gamma_{10}=(3\,4)$ \vspace{1mm} \\
\hline 
\\[-3mm]
$\gamma_6=(2\,3)$ & $\xi_1=(2\, 3\, 1\, 2\, 3)$ &
\; $\xi_2=(2\, 3\, 4)$  \; \vspace{1mm} \\
\hline
\\[-3mm]
\multirow{2}{*}{$\gamma_{10}=(3\,4)$} & $\xi_3= (3\, 4\, 5\, 3\, 1\, 2\, 3)$ &
 $\xi_5=(3\, 4\, 5\, 3\, 4)$ \\
 & $\xi_4=(3\,4\,5\,6\,1\,2\,3)$ &  \; \vspace{1mm} \\
\hline
\end{tabular}
\vspace{.3cm}
\caption{\footnotesize{$S$-branches of  $\chi_A$.}} \label{branch_table_ex1}
\end{table}

Consider now the subspaces of $\Rr^6$
$$ H=\{x\in\Rr^6\,\,:\,\,\sum_{i=1}^6 x_i=1\}\; \text{ and }
\; H_0=\{x\in\Rr^6\,\,:\,\,\sum_{i=1}^6 x_i=0\}\,. $$

For the given matrix $A$, its null space $\Ker(A)$ has dimension $2$.
Take a non-zero vector $w\in \Ker(A)\cap H_0$. The set of equilibria  
of the natural extension of $X_A$ to the affine hyperplane $H$
is
$$ \Eq(X_A)=\Ker(A)\cap H=\{q+tw \colon t\in\Rr\}\,. $$

The Hamiltonian of $X_A$ is the function $h_q:\Delta^5\to\Rr$  
$$ h_q(x):=\sum_{i=1}^6 q_i\log x_i\,,$$
where $q_i$ is the $i$-th component of the equilibrium point $q$.
Another  integral of motion of $X_A$ is the function $h_w:\Delta^5\to\Rr$ 
$$ h_w(x):=\sum_{i=1}^6 w_i\log x_i\,,$$
where $w_i$ is the $i$-th component of $w$,
which is a Casimir of the underlying Poisson structure.

The skeletons  of $h_q$ and
$h_w$ are respectively
$\eta_q, \eta_w:\CC^\ast(\Delta^5)\to\Rr$,
$$ \eta_q(y):=\sum_{i=1}^6 q_iy_i \quad \textrm{and} \quad
\eta_w(y):=\sum_{i=1}^6 w_i y_i\, , $$
which we use to define $\eta:\CC^\ast(\Delta^5)\to\Rr^2$,
$\eta(y):=(\eta_q(y),\eta_w(y))$.

Consider the skeleton flow map  $\skPoin{\chi}{S}:\skDomPoin{\chi}{S} \to\Pi_{S}$ of $\chi_A$,
see Definition~\ref{def skeleton flow map}.
Notice that $\skDomPoin{\chi}{S}=\Pi_{\gamma_6}\cup \Pi_{\gamma_{10}}$, where by Proposition~\ref{partition}
$\Pi_{\gamma_6}=\skDomPoin{\chi}{\xi_1}\cup \skDomPoin{\chi}{\xi_2}\pmod{0}$
and
$\Pi_{\gamma_{10}}=\skDomPoin{\chi}{\xi_3}\cup \skDomPoin{\chi}{\xi_4}\cup \skDomPoin{\chi}{\xi_5}\pmod{0}$.
By Proposition~\ref{la_inv}, the function  $\eta$ is invariant under $\pi_S$.
For all $i=1,\dots,5$, the polyhedral cone $\Pi_{\xi_i}$ has dimension $4$.
Hence, each polytope $\Delta_{\xi_i,c}:=\Pi_{\xi_i}\cap \eta^{-1}(c)$   is a $2$-dimensional polygon.

By invariance of   $\eta$,
the set $\Delta_{S,c}$ is also invariant under $\pi_S$. 
Consider now the restriction ${\pi_S}_{\vert \Delta_{S,c}}$ of $\skPoin{\chi}{S}$ to $\Delta_{S,c}$.
This is a piecewise affine area preserving map, see Remark~\ref{rmk conservative 2}.
Figure~\ref{stab_unstab_equilib_ex1} shows 
the domain  $\Delta_{S,c}$ and $10\,000$  iterates by  ${\pi_S}$  of a point in $\Delta_{S,c}$.
Following the itinerary of a random point we have picked the following
heteroclinic cycle consisting of $11$ $S$-branches
\begin{align*}
\label{concat_path_ex1}
& \xi:=(\xi_1,\xi_2,\xi_4,\xi_2,\xi_4,\xi_2,\xi_4,\xi_2,\xi_4,\xi_2,\xi_4)\,.
\end{align*}

The   map $\pi_\xi$   is
represented by the matrix,
 see definitions~\eqref{matrix-M} and~\eqref{poincare_path_matrix},
\[M_\xi =\left(\begin{array}{cccccc}
 -10 & -5 & 0 & 6 & 1 & -\frac{10}{3} \\
 0 & 0 & 0 & 0 & 0 & 0 \\
 0 & 0 & 0 & 0 & 0 & 0 \\
 1 & 1 & 0 & 0 & 0 & \frac{2}{3} \\
 10 & 5 & 1 & -5 & 0 & \frac{8}{3} \\
 -\frac{3}{2} & 0 & 0 & \frac{3}{2} & 0 & 0 \\
						\end{array}\right)\,.\]

The eigenvalues of $M_\xi$, besides $0$ and $1$ (both with geometric multiplicity $2$), are (approximately)
$$ \lambda_u=-11.9161, \quad \textrm{and} \quad \lambda_s=-0.0839202\,. $$
The corresponding eigenvectors are
$$ w_u=(-0.734728, 0., 0., 0.067307, 0.667421, -0.10096)\,, $$
$$ w_s=(0.328842, 0., 0., 0.358966, -0.687808, -0.538449)\,. $$
An eigenvector associated to the eigenvalue $1$ is
$$ \mathbf{p_{\scriptsize{0}}}=(0.20512, 0, 0, 0.325586, 0.905134, 0.180699)\,. $$
Notice that this  $\mathbf{p_{\scriptsize{0}}}$ is not unique because $\dim( \Ker(M_\xi-I))=2$.
We  have chosen  $c:=(c_1,c_2)=(0.343447,-0.242852)$ so that $\eta(\mathbf{p_{\scriptsize{0}}})=c$, \ie
$\mathbf{p_{\scriptsize{0}}}\in \Delta_{S,c}$.
 In fact we have $\mathbf{p_{\scriptsize{0}}}\in\Delta_{\xi_1,c}\subset\Delta_{\gamma_6,c}$.  Hence $\mathbf{p_{\scriptsize{0}}}$ is a periodic point of the skeleton flow map $\skPoin{\chi}{S}$ with period $11$.

 Figure~\ref{stab_unstab_equilib_ex1} also depicts
 the polygons  $\Delta_{\xi_1,c}, \Delta_{\xi_2,c}$ contained in $\Delta_{\gamma_6}$, and
$\Delta_{\xi_4,c}, \Delta_{\xi_5,c}$ contained in $\Delta_{\gamma_{10}} $.
The set $\Delta_{\xi_3,c}$ is empty for this choice of $c$.
The orbit of  $\mathbf{p_{\scriptsize{0}}}$
is represented by the white dots in Figure~\ref{stab_unstab_equilib_ex1}.

\begin{figure}[!]
\begin{center}
\includefigure{width=13cm}{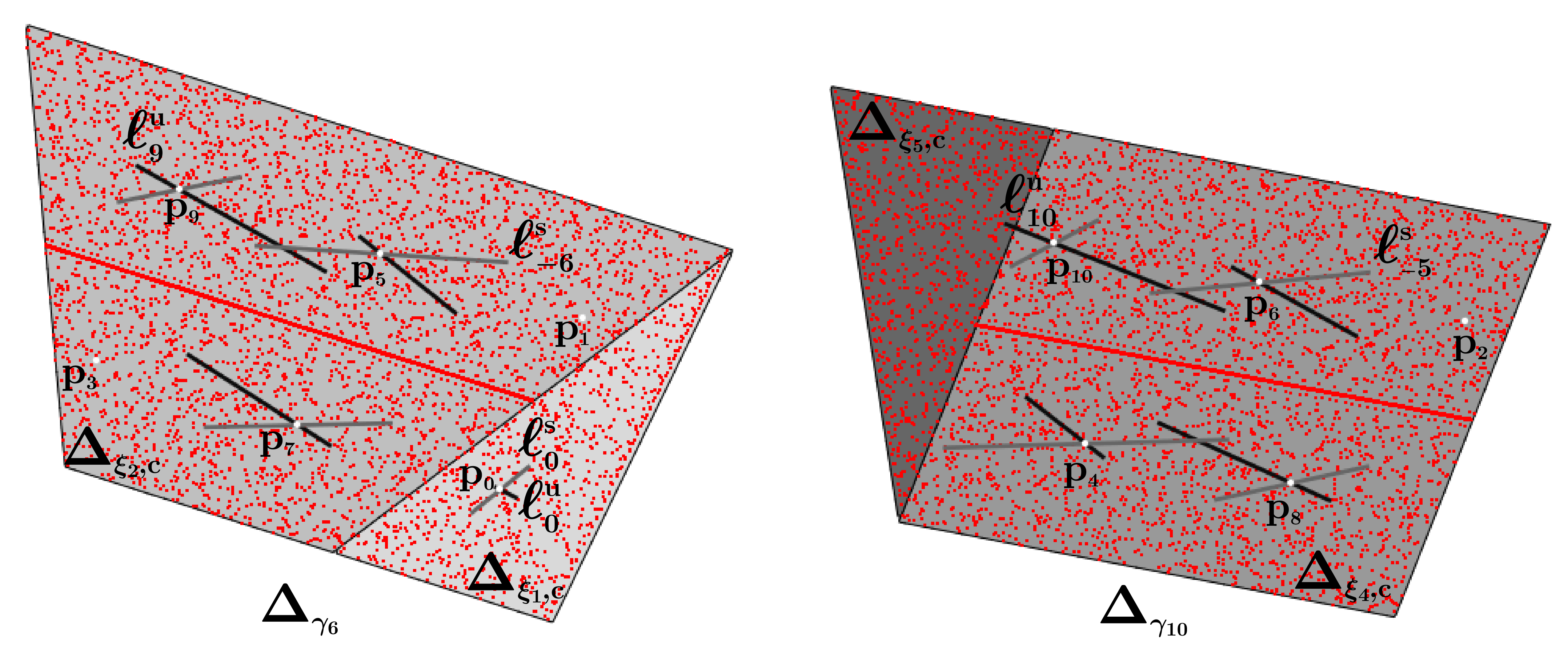}
\caption{ Homoclinic intersections 
	for the periodic orbit of $\mathbf{p_{\scriptsize{0}}}$
	under the  area preserving map ${\pi_S}_{\vert \Delta_{S,c}}$.	
}\label{stab_unstab_equilib_ex1}
\end{center} 
\end{figure}

Let $\ell_0^u$ and $\ell_0^s$ be  line segments through $\mathbf{p_{\scriptsize{0}}}$, contained in $\Delta_{\xi_1,c}$, respectively aligned with the eigen-directions $w_s$
and $w_u$.
We denote by $\ell_n^u$ the $n$-th forward $\pi_{S}$-iterate of $\ell_0^u$ and by $\ell_{-m}^s$ the $m$-th backward
$\pi_{S}$-iterate of $\ell_0^s$, \ie
$$ \ell_n^u:=\pi_{S}^n(\ell_0^u) \quad \textrm{and} \quad \ell_{-m}^s:=\pi_{S}^{-m}(\ell_0^s)\,. $$

Let $\mathbf{p_{\scriptsize{k}}}=\skPoin{\chi}{S}^k(\mathbf{p_{\scriptsize{0}}})$ and notice that $\mathbf{p_{\scriptsize{10}}}=\skPoin{\chi}{S}^{10}(\mathbf{p_{\scriptsize{0}}})=
\skPoin{\chi}{S}^{-1}(\mathbf{p_{\scriptsize{0}}})=\mathbf{p_{\scriptsize{-1}}}$.
Figure~\ref{stab_unstab_equilib_ex1} also shows that in a few iterates 
transversal intersections occur between the ``local stable'' and the ``local unstable'' manifolds of different points of the periodic orbit of $\mathbf{p_{\scriptsize{0}}}$.
Namely, $\ell_{-6}^s\cap\ell_9^u\neq\emptyset$ and 
$\ell_{-5}^s\cap\ell_{10}^u\neq\emptyset$.

By Theorem~\ref{thm horse-shoes} this implies the existence of
chaotic behavior for the flow of $X_A$ in some level set
$h_q^{-1}(c_1/\epsilon)\cap h_w^{-1}(c_2/\epsilon)$, with $c=(c_1,c_2)$ chosen above and for all small enough  $\epsilon>0$.


\bigskip
\subsection{Example 2}

Consider the replicator system defined by matrix
\[B=\left(\begin{array}{cccccc}
 0 & 1 & -2 & 0 & 2 & -1 \\
 -1 & 0 & 1 & -2 & 0 & 2 \\
 2 & -1 & 0 & 1 & -2 & 0 \\
 0 & 2 & -1 & 0 & 1 & -2 \\
 -2 & 0 & 2 & -1 & 0 & 1 \\
 1 & -2 & 0 & 2 & -1 & 0 \\
			\end{array}\right)\,.\]

We denote by $X_B$ the vector field associated to this replicator defined on the simplex  $\Delta^5\,.$
The point
$$ q= \left(\frac{1}{6},\frac{1}{6},\frac{1}{6},\frac{1}{6},\frac{1}{6},\frac{1}{6}\right)\in\Rr^6 $$
 is an  equilibrium of the replicator $X_B$.
Since matrix $B$ is skew-symmetric, the associated replicator is conservative, \ie  $X_B$ is Hamiltonian with respect to some stratified Poisson structure on $\Delta^5$.

Using the notation of the previous example,
the skeleton character $\chi_B$ of $X_B$ is displayed in Table~\ref{ex2:chars}.
This model has $15$ edges:
$3$ neutral edges, $\gamma_3$, $\gamma_8$, $\gamma_{12}$,
and $12$ flowing-edges, $\gamma_1$, $\gamma_2$, $\gamma_4$, $\gamma_5$, $\gamma_6$, $\gamma_7$, $\gamma_9$, $\gamma_{10}$, $\gamma_{11}$, $\gamma_{13}$, $\gamma_{14}$, $\gamma_{15}$.
The flowing-edge directed graph of  $\chi$  is represented in  Figure~\ref{oriented_graph_wang}.

\begin{table}[h] 
\centering
\begin{tabular}{c||rrrrrr}
 
\,$\chi^v_\sigma$ & $\sigma_1$  & $\sigma_2$  & $\sigma_3$  & $\sigma_4$  & $\sigma_5$ & $\sigma_6$ \\
\hline \hline \\[-4mm]
\,$v_1$           &  $0$   &  $1$    &  $-2$   &  $0$    &  $2$    &  $-1$   \\
\,$v_2$           &  $-1$  &  $0$    &  $1$    &  $-2$   &  $0$    &  $2$   \\
\,$v_3$           &  $2$   &  $-1$   &  $0$    &  $1$    &  $-2$   &  $0$   \\
\,$v_4$           &  $0$   &  $2$    &  $-1$   &  $0$    &  $1$    &  $-2$   \\
\,$v_5$           &  $-2$  &  $0$    &  $2$    &  $-1$   &  $0$    &  $1$  \\
\,$v_6$           &  $1$   &  $-2$   &  $0$    &  $2$    &  $-1$   &  $0$   \\
\hline \hline
\end{tabular}
\vspace{.3cm}
\caption{\footnotesize{The skeleton character of $X_B$.}}
\label{ex2:chars}
\end{table}


From this graph we can see that
$$ S=\{ \, \gamma_1=(1\,2), \, \gamma_4=(1\,5), \, \gamma_7=(2\,4), \, \gamma_{10}=(5\,4) \, \}  $$
is a  structural set for $\chi_B$,
whose $S$-branches denoted by $\xi_1,\dots,\xi_{32}$ are displayed in Table~\ref{branch_table_ex2}.

\begin{figure}[h]
\begin{center}
\includefigure{width=7cm}{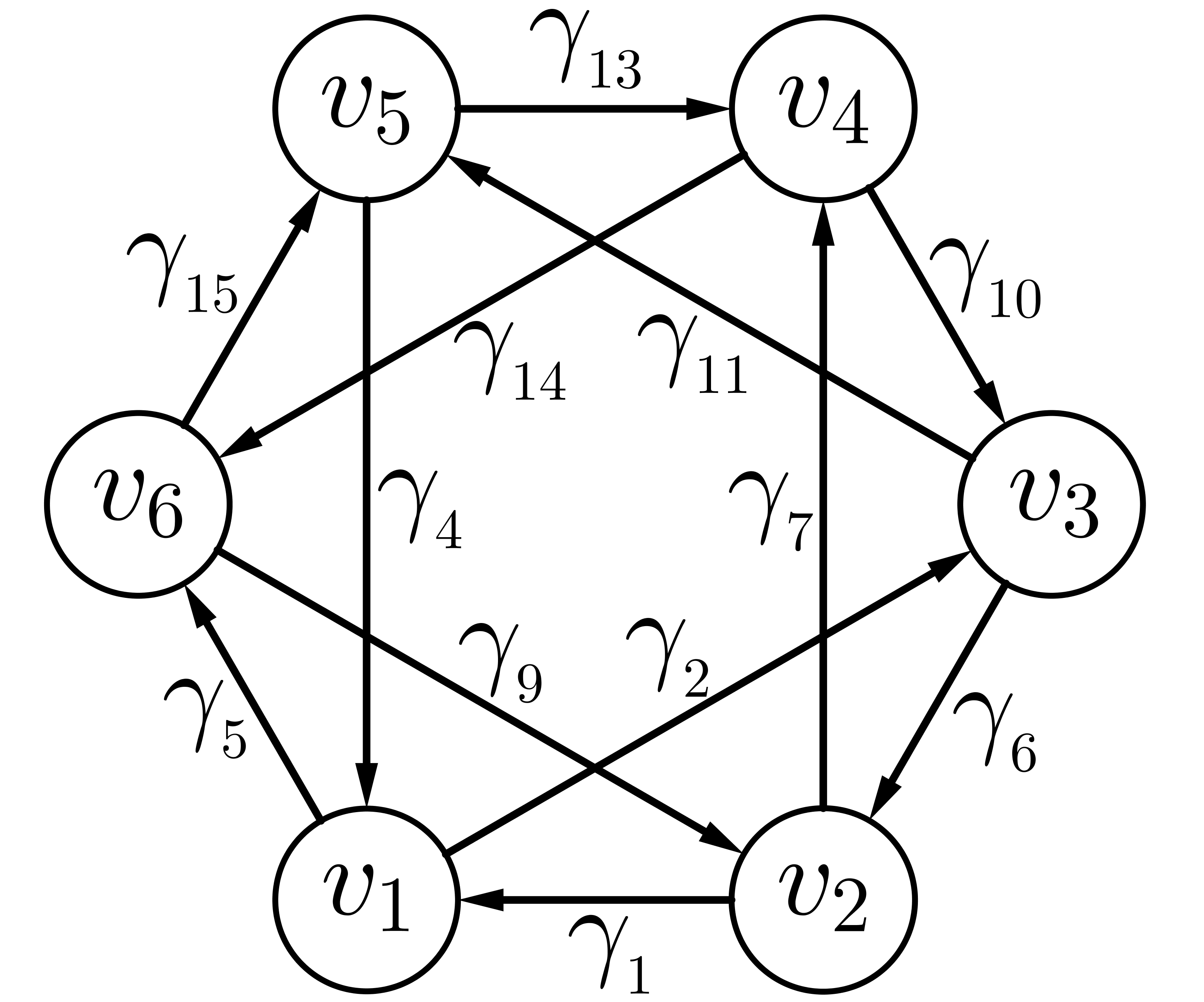}
\vspace{-.3cm}
\caption{\footnotesize{The oriented graph of $\chi$.}} \label{oriented_graph_wang}
\end{center} 
\end{figure}

\begin{table}[h]
\scriptsize{
\centering
\begin{tabular}{c|c|c|c|c}
From\textbackslash To & $\gamma_1=(2\,1)$ & $\gamma_4=(5\,1)$ & $\gamma_7=(2\,4)$ & $\gamma_{10}=(5\,4)$ \vspace{1mm} \\
\hline 
\\[-3mm]
\multirow{2}{*}{$\gamma_1=(2\,1)$} & $\xi_1= (2\,1\,3\,2\,1) $ & $\xi_3=(2\,1\,3\,5\,1)$ &
$\xi_5= (2\,1\,3\,2\,4)$ & $\xi_7=(2\,1\,3\,5\,4)$ \\
& $\xi_2= (2\,1\,6\,2\,1)$ & $\xi_4=(2\,1\,6\,5\,1)$ & $\xi_6= (2\,1\,6\,2\,4)$ &
 $\xi_8=(2\,1\,6\,5\,4)$   \\
\hline
\\[-3mm]
\multirow{2}{*}{$\gamma_4=(5\,1)$} & $\xi_9= (5\,1\,3\,2\,1)$ & $\xi_{11}=(5\,1\,3\,5\,1)$ &
$\xi_{13}= (5\,1\,3\,2\,4)$ & $\xi_{15}=(5\,1\,3\,5\,4)$ \\
& $\xi_{10}= (5\,1\,6\,2\,1)$ & $\xi_{12}=(5\,1\,6\,5\,1)$ & $\xi_{14}= (5\,1\,6\,2\,4)$ &
 $\xi_{16}=(5\,1\,6\,5\,4)$  \\
\hline
\\[-3mm]
\multirow{2}{*}{$\gamma_7=(2\,4)$} & $\xi_{17}= (2\,4\,3\,2\,1)$ & $\xi_{19}=(2\,4\,3\,5\,1)$ &
$\xi_{21}= (2\,4\,3\,2\,4)$ & $\xi_{23}=(2\,4\,3\,5\,4)$ \\
& $\xi_{18}= (2\,4\,6\,2\,1)$ & $\xi_{20}=(2\,4\,6\,5\,1)$ & $\xi_{22}= (2\,4\,6\,2\,4)$ &
 $\xi_{24}=(2\,4\,6\,5\,4)$  \\
\hline
\\[-3mm]
\multirow{2}{*}{$\gamma_{10}=(5\,4)$} & $\xi_{25}= (5\,4\,3\,2\,1)$ & $\xi_{27}=(5\,4\,3\,5\,1)$ &
$\xi_{29}= (5\,4\,3\,2\,4)$ & $\xi_{31}=(5\,4\,3\,5\,4)$ \\
& $\xi_{26}= (5\,4\,6\,2\,1)$ & $\xi_{28}=(5\,4\,6\,5\,1)$ & $\xi_{30}= (5\,4\,6\,2\,4)$ &
 $\xi_{32}=(5\,4\,6\,5\,4)$  \\
\hline
\end{tabular}
\vspace{.3cm}
\caption{\footnotesize{$S$-branches of  $\chi$.}} \label{branch_table_ex2}
}
\end{table}

Consider the affine subspaces $H,H_0\subset \Rr^6$
of the previous example. For the given matrix $B$, its null space $\Ker(B)$ has dimension $2$.
Take a non-zero vector $w\in \Ker(B)\cap H_0$. As before,
$\{q+tw \colon t\in\Rr\}$ is the set of equilibria   of $X_B$ on the affine hyperplane $H$. The same functions $h_q$ and $h_w$
are respectively the Hamiltonian and an  integral of motion.
The skeletons of these functions  are respectively
$\eta_q$ and $ \eta_w$, which we take as the components of a piecewise linear
function $\eta:\CC^\ast(\Delta^5)\to\Rr^2$.
By Proposition~\ref{la_inv},   $\eta$  is invariant under $\pi_S$.
The  map  $\pi_{S}$ acts on  
 $\Pi_{S}=\Pi_{\gamma_1}\cup\Pi_{\gamma_4}\cup\Pi_{\gamma_7}\cup\Pi_{\gamma_{10}}$,
where
\begin{itemize}
	\item $\Pi_{\gamma_1}=\skDomPoin{\chi}{\xi_1} \cup \skDomPoin{\chi}{\xi_2} \cup \dots \cup \skDomPoin{\chi}{\xi_8}\pmod{0}$,
	\item $\Pi_{\gamma_4}=\skDomPoin{\chi}{\xi_9} \cup \skDomPoin{\chi}{\xi_{10}} \cup \dots \cup \skDomPoin{\chi}{\xi_{16}}\pmod{0}$,
	\item $\Pi_{\gamma_7}=\skDomPoin{\chi}{\xi_{17}} \cup \skDomPoin{\chi}{\xi_{18}} \cup \dots \cup \skDomPoin{\chi}{\xi_{24}}\pmod{0}$,
	\item $\Pi_{\gamma_{10}}=\skDomPoin{\chi}{\xi_{25}} \cup \skDomPoin{\chi}{\xi_{26}} \cup \dots \cup \skDomPoin{\chi}{\xi_{32}}\pmod{0}$.
\end{itemize}
For all $i=1,\ldots, 32$, the polyhedral cone  $\Pi_{\xi_i}$ has dimension $4$ while $\Delta_{\xi_i,c}$ is a $2$-dimensional polygon.
The $2$-dimensional  level set $\Delta_{S,c}$ is invariant under $\pi_S$ and we denote by
${\pi_S}_{\vert \Delta_{S,c}}$ the restriction of $\skPoin{\chi}{S}$ to $\Delta_{S,c}$.
Figure~\ref{stab_unstab_equilib_ex2}  shows the domain $\Delta_{S,c}$ and $25\, 000$ iterates by $\skPoin{\chi}{S}$   of a point in $\Delta_{S,c}$
with random like distribution.
Following the itinerary of a random point we have picked a heteroclinic cycle $\xi$ consisting of $13$ $S$-branches
\begin{align*}
\label{concat_path_ex2}
& \xi:=(\xi_{31},\xi_{32},\xi_{32},\xi_{28},\xi_{12},\xi_{10},\xi_2,\xi_1,\xi_5,\xi_{21},\xi_{21},\xi_{23},\xi_{31})\,.
\end{align*}
The   skeleton flow map $\pi_\xi$   is
represented by the matrix
$$ M_\xi =\left(\begin{array}{cccccc}
 1 & \frac{3}{2} & \frac{51}{8} & -\frac{35}{4} & -\frac{33}{8} & -\frac{15}{4} \\
 0 & -\frac{1}{2} & -\frac{21}{8} & \frac{21}{4} & \frac{23}{8} & \frac{9}{4} \\
 0 & -\frac{3}{2} & -\frac{43}{8} & \frac{35}{4} & \frac{41}{8} & \frac{15}{4} \\
 0 & 0 & 0 & 0 & 0 & 0 \\
 0 & 0 & 0 & 0 & 0 & 0 \\
 0 & \frac{3}{2} & \frac{21}{8} & -\frac{17}{4} & -\frac{23}{8} & -\frac{5}{4} \\
						\end{array}\right)\,. $$

The eigenvalues of $M_\xi$, besides $0$ and $1$ (both with geometric multiplicity $2$), are 
$$ \lambda_u= -8, \quad \textrm{and} \quad \lambda_s=-\frac{1}{8}\,. $$
The corresponding eigenvectors are
$$ w_u=(2, -1, -2, 0, 0, 1)\,, $$
$$ w_s=(-1, -1, 1, 0, 0, 1)\,. $$
An eigenvector associated to the eigenvalue $1$ is
$$ \mathbf{p_{\scriptsize{0}}}=(0.62, 0.304, 0.152, 0, 0, 0.38)\,. $$
Notice that this  $\mathbf{p_{\scriptsize{0}}}$ is not unique because $\dim( \Ker(M_\xi-I))=2$.
We  have chosen  $c:=(c_1,c_2)=(0.242667, -0.088)$ so that $\eta(\mathbf{p_{\scriptsize{0}}})=c$, \ie
$\mathbf{p_{\scriptsize{0}}}\in \Delta_{S,c}$.
In fact we have $\mathbf{p_{\scriptsize{0}}}\in\Delta_{\xi_{31},c}\subset\Delta_{\gamma_{10},c}$.  Hence $\mathbf{p_{\scriptsize{0}}}$ is a periodic point of the skeleton flow map $\pi_{S}$ with period $13$.

Figure~\ref{stab_unstab_equilib_ex2} also depicts
the polygons $\Delta_{\xi_1,c}, \Delta_{\xi_2,c},\Delta_{\xi_5,c}$ contained in $\Delta_{\gamma_1}$,
$\Delta_{\xi_{10},c}, \Delta_{\xi_{11},c},\Delta_{\xi_{12},c}$ contained in $\Delta_{\gamma_4}$,
$\Delta_{\xi_{21},c}, \Delta_{\xi_{22},c},\Delta_{\xi_{23},c}$ contained in $\Delta_{\gamma_7}$ and 
$\Delta_{\xi_{28},c}, \Delta_{\xi_{31},c},\Delta_{\xi_{32},c}$ contained in $\Delta_{\gamma_{10}}$.
The remaining sets $\Delta_{\xi_i,c}$
are empty for this choice of $c$.
The orbit of  $\mathbf{p_{\scriptsize{0}}}$
is represented by the white dots in Figure~\ref{stab_unstab_equilib_ex2}.

\begin{figure}[!]
\begin{center}
\includefigure{width=12cm}{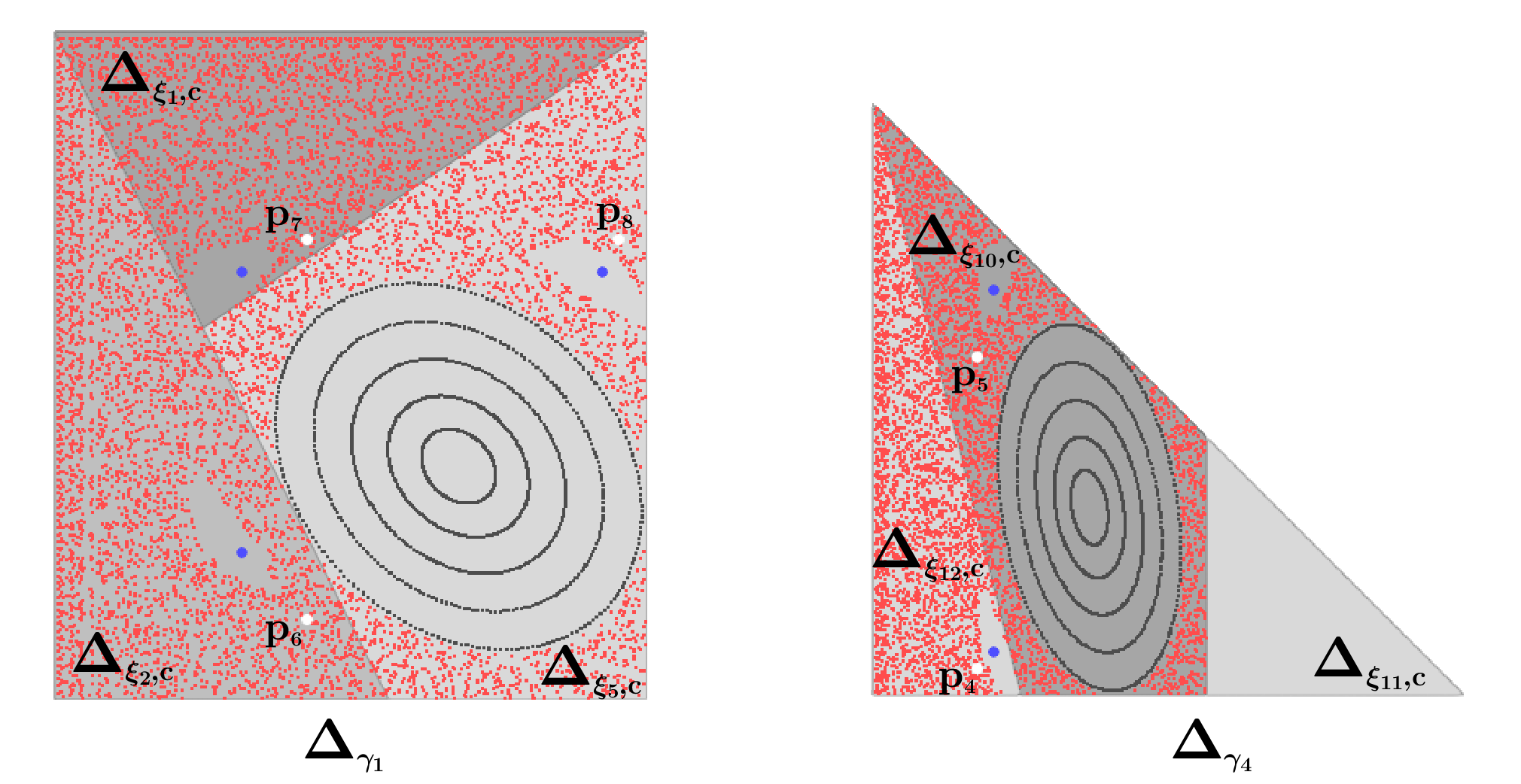}
\includefigure{width=12cm}{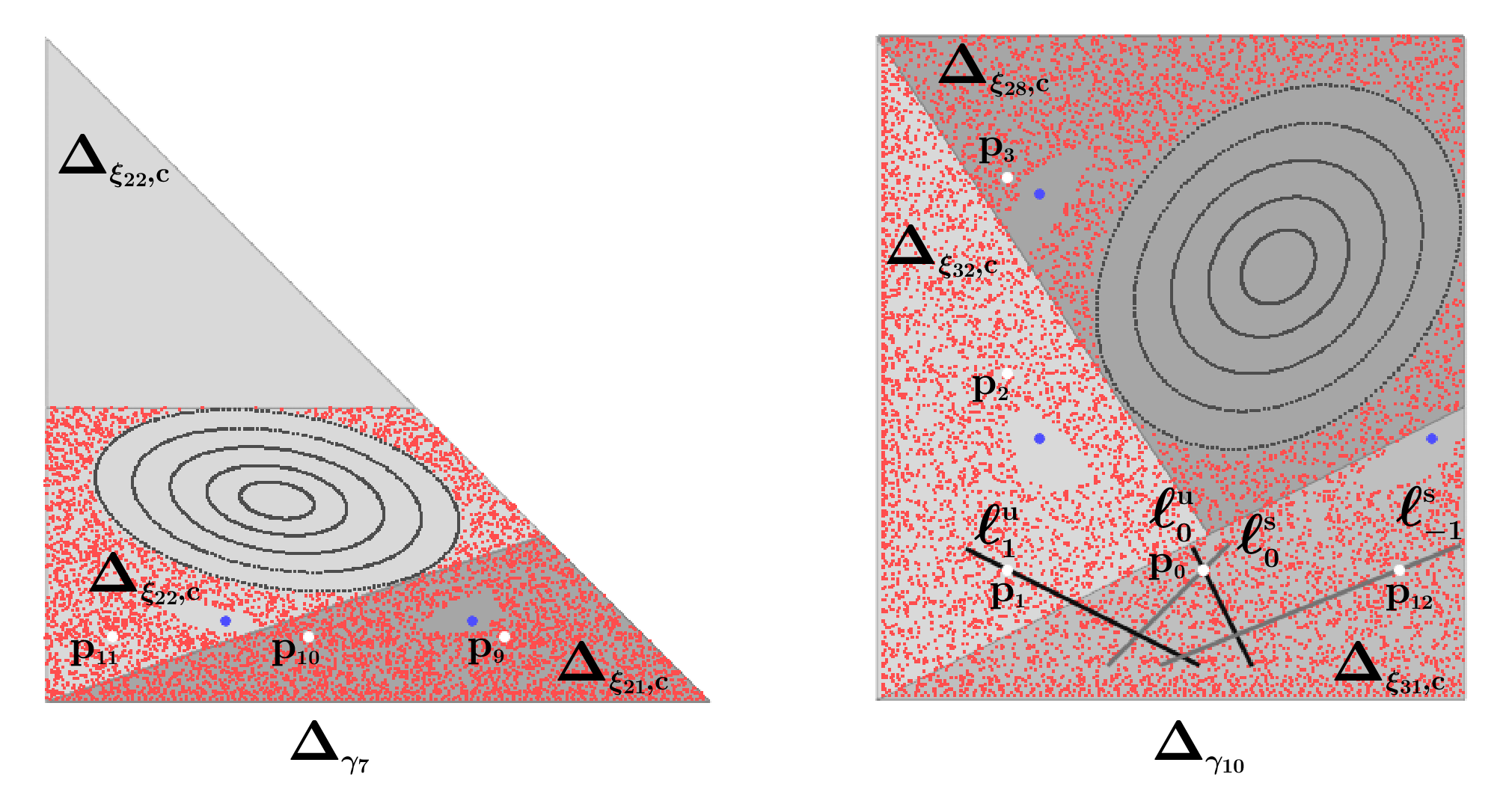}
\caption{Homoclinic intersections 
	for the periodic orbit of $\mathbf{p_{\scriptsize{0}}}$
	under the  area preserving map ${\pi_S}_{\vert \Delta_{S,c}}$.	
}\label{stab_unstab_equilib_ex2}
\end{center} 
\end{figure}

Let $\ell_n^u$ and $\ell_{-m}^s$ denote the stable and unstable
local manifolds along the orbit of $\mathbf{p_{\scriptsize{0}}}$, the notation introduced in the previous example.
Write $\mathbf{p_{\scriptsize{k}}}=\skPoin{\chi}{S}^k(\mathbf{p_{\scriptsize{0}}})$ and notice that $\mathbf{p_{\scriptsize{12}}}=\skPoin{\chi}{S}^{12}(\mathbf{p_{\scriptsize{0}}})=
\skPoin{\chi}{S}^{-1}(\mathbf{p_{\scriptsize{0}}})=\mathbf{p_{\scriptsize{-1}}}$.

Figure~\ref{stab_unstab_equilib_ex2} also shows that in the first forward and backward iterate (by the skeleton flow map)
transversal intersections occur between the ``local stable'' and the ``local unstable'' manifolds of different points of the periodic orbit of $\mathbf{p_{\scriptsize{0}}}$.
Namely, $\ell_1^u\cap\ell_0^s\neq\emptyset$, $\ell_{-1}^s\cap\ell_0^u\neq\emptyset$ and 
$\ell_1^u\cap\ell_{-1}^s\neq\emptyset$.

By Theorem~\ref{thm horse-shoes} this implies the existence of
chaotic behavior for the flow of the replicator $X_B$.

\section{Conclusions and furtherwork}
\label{furtherwork}

For the Hamiltonian polymatrix replicator systems,
 alluded in Remark~\ref{rmk conservative 1} and studied in~\cite{AD2014} by the first two authors, their invariant algebraic Poisson structures induce stratified piecewise constant  Poisson structures on the dual cone, preserved by the corresponding skeleton flow.
 In other words, the skeleton flow inherits the conservative Hamiltonian  nature of the original polymatrix replicator vector field. This subject will be addressed  by the authors in a future work. Remark~\ref{rmk conservative 2} follows from this theory.

In the  Hamiltonian examples discussed in Section~\ref{example section}, chaotic  (hyperbolic) and regular  (elliptic) behavior co-exist. In both of them, the skeleton flow maps   are piecewise linear area preserving maps.
Example 2 exhibits a few elliptic periodic points surrounded by invariant curves
that we will refer to as \textit{elliptic domains}\footnote{ These are not KAM islands because the piecewise linearity of $\pi_S$ does not allow any twist.}, see Figure~\ref{stab_unstab_equilib_ex2}.
Notice how the invariant curves break up as they touch the boundary 
of the associated domain $\Delta_{\xi,c}$. Outside these elliptic domains, a chaotic sea (`random' orbits  with  positive Lyapunov exponents) seems to prevail. In Figure~\ref{stab_unstab_equilib_ex2} we can also see 
a couple of triangular components, $\Delta_{\xi_{11},c}$, $\Delta_{\xi_{22},c}$, consisting of $\pi_S$-fixed points  and  
$10$  cyclically permuted small islands  each being a  \textit{continuum} of periodic points. 
These examples solicit  the development of an ergodic theory for the class of piecewise linear area preserving maps, and  more generally for the class of piecewise linear symplectic maps.
We mention a few natural questions about the generic behavior of these systems: \textit{Can  the number of elliptic domains be infinite? 
Is the complement of the elliptic domains non-uniformly hyperbolic with respect to Lebesgue measure? Is this complement typically ergodic?  Can it have infinitely many ergodic components?}
In the context of smooth area preserving maps these are  very hard open problems, but due to their dynamical  rigidity  these sort of problems might be much more feasible for piecewise linear area preserving maps.
Such a theory would provide a good insight on the asymptotic dynamics (along the vertex-edge network) for the classes of Hamiltonian systems on polytopes mentioned above.

General vector fields $X\in\fX(\Gamma^d)$ typically  do not have  any integral of motion and the analysis of their dynamics must be different from the conservative case. The skeleton flow map $\pi_S:\Pi_S\to\Pi_S$ can be projectivized  as follows. Take $\eta:\CC^\ast(\Gamma^d)\to\RR$ to be the piecewise linear function
$\eta(y):=\sum_{\sigma\in F} y_\sigma$ and define
$\Delta_\gamma:=\Pi_\gamma\cap \eta^{-1}(1)$,
$\Delta_\xi:=\Pi_\xi\cap \eta^{-1}(1)$ as before.
The simplex $\Delta_\gamma$ can be viewed as the projectivization of the sector $\Pi_\gamma$  because every half-line through the origin
in $\Pi_\gamma$  intersects $\Delta_\gamma$ at a single point. Likewise  $\Delta_\xi$ is the projectivization of  $\Pi_\xi$. 
If $\xi$ is a heteroclinic path ending at  some flowing edge $\gamma$ then  the linear map
$\pi_\xi:\Pi_\xi\to \Pi_\gamma$ induces a projective map
$\hat \pi_S:\Delta_\xi\to\Delta_\gamma$ 
defined by $\hat \pi_\xi(y):=\eta(\pi_\xi(y))^{-1} \pi_\xi(y)$.
These are the branches of the \textit{projective skeleton map}
$\hat \pi_S:\Delta_S\to \Delta_S$ defined on 
$\Delta_S:=\cup_{\xi\in\Bscr(\chi)} \Delta_\xi$ by
$\hat \pi_S(y):=\hat \pi_\xi(y)$ if $y\in \Delta_\xi$
for some $S$-branch $\xi$. The suspension of the projective map
$\hat \pi_S$ on $\Delta_S$ can be viewed as a \textit{blowup} of the flow $\varphi_X^t$ along the  polytope's boundary, \ie $\hat \pi_S$
extends the dynamics of $\varphi_X^t$ to the \textit{blown-up} polytope's boundary.  In the conservative case, if $h\in\HH(\Gamma^d)$ is a proper $X$-invariant function with skeleton $\eta$ and the same order function as $X$, the projective skeleton map $\hat \pi_S$
rules the common dynamics on all level sets of $\eta$.

The map  $\pi_S$  factors through $\hat \pi_S$ acting linearly on the fibers.
Hence $\pi_S$  may be regarded as a $1$-dimensional linear cocycle over $\hat \pi_S$, where the sign of its Lyapunov exponent gives  the repelling \textit{vs} attracting nature of the asymptotic boundary dynamics.
Given a heteroclinic cycle $\xi$,
if $v\in \Delta_\xi$ is an eigenvector of $M_\xi$ with a positive eigenvalue then $v$ is a periodic point of $\hat \pi_S$ 
whose nature can be read from the spectrum of $M_\xi$.
This spectrum also determines whether the heteroclinic cycle $\xi$
is attracting  or repelling. If a compact  $\hat \pi_S$-invariant set is partially hyperbolic (with a central direction of co-dimension $1$) regarded as an invariant subset of the blown-up boundary of the flow $\varphi_X^t$ then it determines a local strong stable/unstable foliation in the polytope's interior. The special case where this
compact invariant set is a single  periodic orbit provides an invariant (local stable/unstable) manifold of the heteroclinic cycle associated with the periodic orbit. These dynamical  foliations and invariant manifolds are useful tools to analyze the dynamics in the polytope's interior, part of a theory being developed in a work under preparation.
This theory could for instance help to provide sufficient conditions for permanence, an important concept in EGT. 
	In this spirit, a theorem of Jansen~\cite{Jan1987} with a game flavored sufficient criteria for permanence, in the framework of replicator dynamics, was recently extended by the third author to the broader class of  polymatrix replicators~\cite{TP2018}.

Although the piecewise linear maps $\pi_S$ are in general discontinuous, because orbits in adjacent domains  eventually diverge, in some cases $\pi_S$ is continuous throughout several  neighboring domains. This implies that the rescaling along the vertex-edge polytope' skeleton can be augmented to include some higher 
dimensional faces of the polytope. An extreme example is the $3$-dimensional Hamiltonian  depicted in Figure~\ref{cons222},
which has a globally continuous  skeleton flow. In this model the
rescaling can be augmented to include the whole cube's boundary.  This  will the subject of another future work.

This theory can be applied to most ODE models in EGT.
For systems depending on many parameters, an algorithmic analysis  of the skeleton (asymptotic) dynamics 
can split the space of parameters into regions where the
dynamics of the skeleton flow maps are qualitatively similar.
This would help to understand the bifurcations taking place in the polytope's interior as the parameters cross the boundary between adjacent regions. For example, higher dimensional cases of the systems studied at \cite{wang-wu-ruan} could be investigated.
In each parametric region, the mentioned tools can be used to  detect and characterize some of its invariant dynamical structures such as  heteroclinic cycles, periodic points, hyperbolic invariant sets, invariant manifolds and invariant foliations, which are 
essential to understand the model's dynamics in the polytope's interior. In some future work the authors plan  to illustrate
this approach with the analysis of some concrete EGT model.




\section*{Acknowledgments}

The first author was supported by mathematics department of UFMG. 
The second author was supported  by Funda\c{c}\~{a}o para a Ci\^{e}ncia e a Tecnologia, under the project: UID/MAT/04561/2013.
The third author was supported by FCT scholarship  SFRH/BD/72755/2010 and
by the Project CEMAPRE - UID/MULTI/00491/2013 financed by FCT/MCTES through national funds.


\bibliographystyle{amsplain} 
\bibliography{references}

\end{document}